\documentclass[a4paper,leqno,11pt]{amsart}

\usepackage{dynkin-diagrams}
\input Dynkin.sty

\allowdisplaybreaks[4]

\usepackage[top=1.5in,bottom=1.3in,left=1.3in,right=1.3in,marginparwidth=1.5cm]{geometry}

\usepackage{amsrefs}
\usepackage{mathrsfs}
\usepackage{wasysym}
\usepackage{enumerate}
\usepackage[shortlabels]{enumitem}

\usepackage[new]{old-arrows}
\usepackage{times}
\usepackage{tensor}
\usepackage[all]{xy}

\usepackage{amsmath, amssymb, amsfonts, latexsym, mdwlist, amsthm}
\usepackage{mathrsfs,mathtools}
\usepackage{subfig}
\usepackage{graphicx}
\usepackage{longtable}

\usepackage{mathtools}

\usepackage[colorlinks=true, citecolor=blue, urlcolor=blue, linkcolor=blue
]{hyperref}

\usepackage{multirow}
\usepackage{diagbox}
\usepackage{arydshln}
\usepackage{bbold}
\usepackage{fix-cm} 

\usepackage[english]{babel}
\usepackage[utf8x]{inputenc}
\usepackage{enumerate}
\usepackage{tikz}
\usepackage{tikz-cd}
\usepackage{xfrac}    
\usepackage{faktor}

\graphicspath{{./pictures/}}

\DeclareRobustCommand\longtwoheadrightarrow
     {\relbar\joinrel\twoheadrightarrow}

\newtheorem{thm}{Theorem}[section]
\newtheorem{cor}[thm]{Corollary}
\newtheorem{lem}[thm]{Lemma}
\newtheorem{prop}[thm]{Proposition}

\newtheorem{rem}[thm]{Remark}
\newtheorem{defn}[thm]{Definition}

\newtheorem{mainthm}{Theorem}

\newtheorem{mainconj}{Conjecture}

\newcommand{\mc}{\mathcal}
\newcommand{\mf}{\mathfrak}

\newcommand{\C}{\mathbb{C}}
\newcommand{\OO}{\mathbb{O}}
\newcommand{\U}{U}
\newcommand{\Z}{\mathbb{Z}}

\newcommand{\W}{\mathcal{W}}

\newcommand{\la}{\langle}
\newcommand{\ra}{\rangle}

\newcommand{\lam}{\lambda}
\newcommand{\g}{\mf{g}}
\newcommand{\affg}{\widehat{\mf{g}}}
\newcommand{\affh}{\widehat{\mf{h}}}
\newcommand{\extg}{\widetilde{\mf{g}}}
\newcommand{\exth}{\widetilde{\mf{h}}}
\newcommand{\affn}{\widehat{\mf{n}}}

\newcommand{\h}{\mf{h}}
\newcommand{\n}{\mf{n}}
\newcommand{\gl}{\mf{gl}}
\newcommand{\sll}{\mf{sl}}
\newcommand{\spp}{\mf{sp}}
\newcommand{\sing}{\text{\rm sing}}

\newcommand{\cmmnt}[1]{}

\newcommand{\weyl}{\mathrm{W}}

\DeclareMathOperator{\Span}{span}

\def \<{\langle}
\def \>{\rangle }

\usepackage{hyperref} 
\hypersetup{
	colorlinks=true,
	linkcolor=blue,
	urlcolor=red,
	pdftitle={Fusion rules},
}

\title
[On some simple orbifold affine VOAs at non-admissible level]
{On some simple orbifold affine VOAs at non-admissible level 
arising from rank one 4D SCFTs}

\author{Tomoyuki Arakawa}
\address[Tomoyuki Arakawa]{School of Mathematics and Statistics, Ningbo University, Ningbo 315211, China}
\address
{Research Institute for Mathematical Sciences, Kyoto University, Kyoto 606-8502, Japan}
\email{arakawa@kurims.kyoto-u.ac.jp}

\author{Xuanzhong Dai}
\address[Xuanzhong Dai]{Research Institute for Mathematical Sciences, Kyoto University, Kyoto 606-8502, Japan}
\email{xzdai@kurims.kyoto-u.ac.jp}

\author{Justine Fasquel}
\address[Justine Fasquel]{School of Mathematics and Statistics, University of Melbourne, Parkville, 3010, Australia}
\email{justine.fasquel@unimelb.edu.au}

\author{Bohan Li}
\address[Bohan Li]{Yau Mathematics Sciences Center, Tsinghua University,  Beijing, 100084,China}
\email{libh19@mails.tsinghua.edu.cn}

\author{Anne Moreau}
\address[Anne Moreau]{Universit\'{e} Paris-Saclay, CNRS, Laboratoire de Math\'{e}matiques d'Orsay, 
Rue Michel Magat, B\^{a}t. 307, 
91405 Orsay, France}
\email{anne.moreau@universite-paris-saclay.fr}

\begin{document}
\maketitle

\begin{abstract}
We study the representations  
of some simple 
affine vertex algebras at non-admissible level arising from rank one 4D SCFTs. 
In particular, we classify the irreducible highest weight modules of 
$L_{-2}(G_2)$ and $L_{-2}(B_3)$. 
It is known by the works of Adamovi\'{c} and Per\v{s}e that these 
vertex algebras 
can be conformally embedded into $L_{-2}(D_4)$. 
We also compute the associated variety of $L_{-2}(G_2)$, 
and show that it is the orbifold of the associated variety of $L_{-2}(D_4)$
by the symmetric group of degree 3 which is the Dynkin diagram 
automorphism group of $D_4$. 
This provides a new interesting example of associated variety  
satisfying 
a number of conjectures in the context of orbifold vertex algebras.

%
\end{abstract}
		
		
\section{Introduction} 
\label{Sec:Intro}
We consider in this article the simple affine vertex algebras 
$L_{-2}(G_2)$ and $L_{-2}(B_3)$ 
that appear as orbifold vertex algebras of the 
simple affine vertex algebra $L_{-2}(D_4)$. 
We are interested 
in their representations in the category $\mc{O}$; 
those of $L_{-2}(D_4)$ were previously studied in \cite{AraMor1}. 
We also compute the associated variety of $L_{-2}(G_2)$; 
that of $L_{-2}(B_3)$ and of $L_{-2}(D_4)$ were described in \cite{AraMor1,AraMor3}.  

\medskip  


We now provide more detail on the results and the motivations. 
\medskip 

The symmetric group $\mathfrak{S}_3$ acts as Dynkin 
diagram automorphisms on $D_4=\mf{so}_8(\C)$. 
It is well known that the subalgebra of $\mathfrak{S}_3$-invariants 
has type $G_2$ while the 
subalgebra of $\langle \sigma \rangle$-invariants, 
with $\sigma$ a two-order element in $\mathfrak{S}_3$,  
has type $B_3$. 
So we get the following embeddings of Lie algebras:
\begin{align}
\label{eq:embed}
G_2  \stackrel{\iota_2}{\longhookrightarrow} B_3 
\stackrel{\iota_3}{\longhookrightarrow}  D_4. 
\end{align}

For an arbitrary simple Lie algebra $\g$,  
consider the corresponding extended affine Kac--Moody Lie algebra 
$$\extg = \g\otimes \C[t,t^{-1}] \oplus \C K \oplus \C D$$  
with usual Lie bracket, 
see Section \ref{Sec:Pre}. 
Let $V^k(\g)$ be 
the universal affine vertex algebra  
associated  with $\g$ at level $k$, and $L_k(\g)$ its simple quotient. 
It will be always assumed in the article that $k \not=-h_\g^\vee$ is not critical, 
where $h_\g^\vee$ is the dual Coxeter number of $\g$. 
Thus $V^k(\g)$ is conformal with conformal 
grading given by the semisimple element $L_0=-D$. 

The inclusions \eqref{eq:embed}  
induce embeddings 
for the corresponding universal affine vertex algebras 
at any level:  
$$V^k(G_2) \longhookrightarrow V^k(B_3) \longhookrightarrow V^k(D_4).$$ 
Remarkably,
the following conformal embeddings for the 
simple quotients 
at negative integer level $k=-2$ 
were established by Adamovi\'{c} and Per\v{s}e in \cite{AP}:
\begin{align}\label{eq:simples}
L_{-2}(G_2) \longhookrightarrow L_{-2}(B_3) 
\longhookrightarrow L_{-2}(D_4).
\end{align}
Here a conformal vertex algebra $U$ 
is said to be {\em conformally
embedded} into a conformal vertex algebra 
$V$ if $U$ can be realized as a vertex subalgebra of $V$ with the same
conformal vector.  
Moreover, Adamovi\'{c} and Per\v{s}e 
proved that $L_{-2}(D_4)$ is a finite extension of $L_{-2}(B_3)$ 
and that 
$L_{-2}(B_3)$ is a finite extension of $L_{-2}(G_2)$.  
The group $\mf{S_3}$ naturally acts on $V^k(D_4)$ 
and, 
according to \cite{AP}, 
we also have $L_{-2}(G_2) = L_{-2}(D_4)^{\mf{S}_3}$, 
where for $V$ a vertex algebra and $G$ a finite 
subgroup of its automorphism group, $V^G$ 
denotes the fixed point vertex subalgebra. 
Similarly, $L_{-2}(B_3) = L_{-2}(D_4)^{\langle \sigma \rangle}$. 
Thus $L_{-2}(G_2)$ and $L_{-2}(B_3)$ 
are both orbifold vertex subalgebras of $L_{-2}(D_4)$.

Recall that to an arbitrary vertex algebra $V$ one attaches, in a functorial manner, 
a certain affine Poisson variety $X_V$ referred to as the 
{\em associated variety}~\cite{Ara12}, see Section \ref{Sub:assoc}. 
The simple affine vertex algebras $L_{-2}(B_3)$  and 
$L_{-2}(D_4)$ are known to be {\em quasi-lisse}. 
In this context, it means that the associated variety is contained in the nilpotent cone 
of the corresponding Lie algebra.  
They provided   
one of the first examples of quasi-lisse simple affine vertex algebras 
at non-admissible levels. 
Both are part of a larger family of such examples. 
First, $L_{-2}(D_4)$ is part of the family $L_k(\g)$ 
where $\g$ belongs to the Deligne exceptional series, 
$$A_1 \subset A_2 \subset G_2 \subset D_4 \subset F_4 \subset 
E_6 \subset E_7 \subset E_8,$$
and $k=-h_\g^\vee/6-1$, with $h_\g^\vee$ the dual Coxeter number. 
For such $L_k(\g)$, it is known that the associated variety is the closure of the minimal 
nilpotent orbit \cite{AraMor1}. 
In the case where $\g$ is simply-laced, the above vertex algebras 
$L_{k}(\g)$ 
come from four-dimensional $\mathcal{N}=2$ superconformal field theories 
in physics (see Section \ref{Sub:physics}). 
For the moment, the problem of whether the vertex algebras $L_{-5/3}(G_2)$ 
and $L_{-5/2}(F_4)$ of this series have a physical meaning is not solved.  
As for $L_{-2}(B_{3})$, it is part of the family 
$L_{-2}(B_r)$, $r\geqslant 3$, 
for which the associated variety has been described 
in \cite{AraMor3}. 

Recently, it was noticed by Li, Xie and Yan \cite{LXY} that the vertex algebras 
$L_{k}(\g)$ for $\g$ belonging to the series, 
$$A_1 \subset A_2 \subset G_2\subset B_3 \subset D_4 \subset F_4 \subset 
E_6 \subset E_7 \subset E_8,$$
and  $k=-h/6-1$, where $h$ is the Coxeter number, 
also come from four-dimensional $\mathcal{N}=2$ superconformal field theories 
in the context of the Argyres--Douglas theory 
(see Section~\ref{Sub:physics}).  
In particular, it is natural to focus on the following two series 
$$L_{-2}(G_2) \hookrightarrow  L_{-2}(B_3) \hookrightarrow  L_{-2}(D_4) 
\quad  \text{and}\quad  
L_{-3}(F_4) \hookrightarrow  L_{-3}(E_6),$$
which involve non simply-laced cases. 
Note that the second embedding is also conformal by \cite{AP}. 
The associated variety and the representations in the category $\mc{O}$ of 
$L_{-3}(E_6)$ have been determined in \cite{AraMor1}. 

As explained below, our technics are based on one hand on 
the explicit computation of a singular vector fir $G_2$ 
and on the other hand on the OPEs of the subregular 
$W$-algebra associated with $G_2$. 
Unfortunately, so far, our methods do not apply for $L_{-3}(F_4)$.
However, it was conjectured in  \cite{LXY} that $X_{L_{-3}(F_4)}=\overline{\mathbb{O}_{22}}$,
the unique nilpotent orbit of dimension 22 in~$F_4$.

\subsection{Main results}

To describe our result about the associated variety,  
note that the embedding $\iota_2 \colon G_2 \hookrightarrow D_4$ 
induces a projection map $D_4^* \twoheadrightarrow G_2^*$. 
Hence we get a linear map 
$$\pi_2 \colon D_4 \longtwoheadrightarrow G_2,$$  
identifying $D_4$ and $G_2$ with 
their duals through their respective Killing forms. 
Denoting by $\OO_\text{sreg}$ the subregular nilpotent orbit 
in $G_2$, by $\OO_\text{min}$ the minimal nilpotent 
orbit in $D_4$ and by 
$\overline{\OO}_\text{sreg}$, $\overline{\OO}_\text{min}$ 
their Zariski closures, we have by \cite{LS},  
\begin{align}
\label{eq:D4_G2}
\overline{\OO}_\text{sreg} = \pi_2 (\overline{\OO}_\text{min}). 
\end{align} 
Note that 
${\OO_\text{sreg}}$ 
and ${\OO_\text{min}}$ have both  dimension 10. 
Furthermore,  by 
\cite{AraMor1}, $\overline{\OO}_\text{min}$ 
is precisely the associated variety of the vertex algebra $L_{-2}(D_4)$. 

The following result was conjectured in \cite[Conjecture 4.5]{Fasquel-OPE},
and agrees with the physical expectation \cite[Table 4]{LXY}, 
see also Section \ref{Sub:physics} below. 

\begin{mainthm}
\label{Th:mainA}
The associated variety of $L_{-2}(G_2)$ is 
$\overline{\OO}_{\text{\em sreg}}.$ 
\end{mainthm}
Likewise, the embedding $\iota_3 \colon B_3 \hookrightarrow D_4$ 
induces a linear projection map 
$$\pi_3 \colon D_4 \longtwoheadrightarrow B_3.$$ 
The associated variety of $L_{-2}(B_3)$ 
was obtained in \cite{AraMor3}:  
\begin{align}
\label{eq:AV_B3}
X_{L_{-2}(B_3)}=\overline{\OO}_{\text{short}}
= \pi_3 (\overline{\OO}_{\text{min}}),
\end{align}
where $\OO_{\text{short}}$ is the unique short nilpotent orbit  
in $B_3$ and  $\overline{\OO}_{\text{short}}$ is its Zariski closure. 
Here, a nilpotent element $f$ of a simple Lie algebra $\g$ is called {\em short} if for 
$(e, f, h)$ an $\mf{sl}_2$-triple,
$$\g = \g_{-1}\oplus \g_0 \oplus \g_1,$$
where $\g_j = \{x \in \g \colon  [h, x] = 2 j x\}$. 
In $B_3=\mf{so}_7(\C)$, 
the partition corresponding  to $\OO_{\text{short}}$ is~$(3,1^4)$ 
and $\OO_{\text{short}}$ has dimension $10$, too. 
Similarly to the computations of $X_{L_{-2}(D_4)}$ 
and $X_{L_{-2}(B_3)}$, the proof of Theorem \ref{Th:mainA} 
is based on the analysis of singular vectors and the 
theory of $\W$-algebras. 
Here, obtaining a singular vector is much 
harder, and the novelty 
is the use of the explicit OPE's between the 
generators of the subregular 
$\W$-algebra in $G_2$ computed 
in \cite{Fasquel-OPE}. 
It was observed in \cite[Corollary 4.2]{Fasquel-OPE} 
that $\W_{-2}(G_2,f_{\text{sreg}}) \cong\C$   
which prompted   to conjecture 
$X_{L_{-2}(G_2)}=\overline{\OO}_{\text{sreg}}$. 

Our next results give a complete classification of 
the simple highest-weight $L_{-2}(\g)$-modules and 
the simple ordinary 
$L_{-2}(\g)$-modules for $\g=G_2$ and $\g=B_3$. 
Here, a module is called {\em 
ordinary} if $L_0$ acts
semisimply on it with finite-dimensional graded components 
and a grading bounded from below. 
Let us denote by $L_{\g}(k,\mu)$ the irreducible 
highest-weight modules of $\extg$ at level $k$
with highest-weight $\mu+k \Lambda_0$, 
where $\mu$ is in the dual of the Cartan subalgebra of 
$\g$ and $\Lambda_0$ is the dual of the central element $K$ 
in the dual of the Cartan of~$\extg$. 

\begin{mainthm}
\label{Th:mainB}
The set
$ \{L_{G_2}(-2,\mu_{i}) \colon i=1,\ldots,20\},$
where the $\mu_i$'s are given by Proposition~\ref{Pro:mu_i-G2}, 
provides the complete list of irreducible 
$L_{-2}(G_2)$-modules from the category $\mathcal{O}$.  
Among them, 
$L_{G_{2}}(-2,0)$, $L_{G_{2}}(-2,\varpi_{1})$ and $L_{G_{2}}(-2,\varpi_{2})$ 
are precisely the irreducible ordinary modules of $L_{-2}(G_{2})$.
\end{mainthm}

{Exploiting our singular vector in $V^{-2}(G_2)$ 
and the notion of subsingular vectors (see Definition \ref{Def:subsingular}) 
in $V^{-2}(B_3)$, we succeed in describing 
the maximal ideal of $V^{-2}(B_3)$. 
This leads us to the following classification result.}

\begin{mainthm}
\label{Th:mainC}
The set
$
\{L_{B_{3}}(-2,\mu_{i})\colon i=1,\ldots,13\},
$
where the $\mu_i$'s are given by Proposition~\ref{prop2.5},
provides the complete list of irreducible 
$L_{-2}(B_3)$-modules from the category $\mathcal{O}$. 
Among them, $L_{B_{3}}(-2,0)$ and $L_{B_{3}}(-2,\varpi_{1})$ 
are precisely the irreducible ordinary modules for 
$L_{-2}(B_3)$. 
\end{mainthm}

{
	The classifications of relaxed modules with finite-dimensional weight spaces for $L_{-2}(G_{2})$ and $L_{-2}(B_3)$ are deduced from these classifications using an algorithm presented in \cite{KR22} (see Corollaries \ref{cor:relaxedG2} and \ref{cor:relaxedB3}).
	Relaxed modules are a type of generalized highest-weight modules where highest-weight vectors no longer needs to be annihilated by the positive root vectors of the horizontal subalgebra \cite{FST98}. 
	They are strongly believed to play an important role in the representation theory of affine vertex algebras at non-rational levels, see for instance \cite{AM,Gab,FRR,KRW1,Ridout1,RSW} and references therein.
	
	The algorithm in \cite{KR22} provides a convenient way to construct the simple weight modules over the Zhu algebra $A(L_{k}(\g))$, and thus the irreducible $\Z_{\geqslant 0}$-graded $L_k(\g)$-modules, using the classification of highest-weight modules.
	Relaxed $L_{-2}(D_4)$-modules were previously classified in \cite{KR22} using this algorithm based on the classification of $L_{-2}(D_4)$-modules in the category $\mathcal{O}$ that appeared first in \cite{Per}.
	These modules are also relaxed $L_{-2}(G_{2})$-modules by restriction but as $L_{-2}(G_{2})$-modules, the weight spaces are  infinite-dimensional most of the time.
}

We also establish the following result.
\begin{mainthm}
\label{Th:mainD}
We have the following decomposition 
\begin{align*}
L_{D_4}(-2,-2\varpi_1)=L_{B_3}(-2,-2\varpi_1) \oplus L_{B_3}(-2, -3\varpi_1)
\end{align*}
as $ L_{-2}(B_3) $-modules.
\end{mainthm} 
Note that $L_{D_4}(-2,-2\varpi_1)$ is not an ordinary module since its $L_0$-eigenspaces are not finite-dimensional.
Since $L_{-2}(B_3) 
\hookrightarrow L_{-2}(D_4)$ is a conformal embedding 
and since both $L_{B_3}(-2,-2\varpi_1) $ and $L_{B_3}(-2, -3\varpi_1)$ have the same conformal dimension $-1$ (see the formula \eqref{eq:conf_dim}), 
Theorem \ref{Th:mainD} in particular implies the non-trivial decomposition 
$$L_{D_4}(-2\varpi_1)=L_{B_3}(-2\varpi_1) \oplus  L_{B_3}(-3\varpi_1)$$
of an infinite-dimensional representation of the finite-dimensional Lie algebra $B_3$,
where 
 $L_{\g}(\mu)$ denotes the irreducible 
highest-weight modules of $\g$ 
with highest-weight $\mu$.
The authors do not know whether this has been known in the literature.

\subsection{Motivations from physics}
\label{Sub:physics}
The vertex algebras
$L_{-2}(G_2) $,
$ L_{-2}(B_3) $
and
$ L_{-2}(D_4)$ 
are neither rational nor lisse. 
Therefore, they are not related in any sense 
with rational conformal field theories in two dimensions.
However,
they are remarkably related with superconformal field theories
in
{\em four} dimensions, via the 4D/2D correspondence discovered in \cite{BeeLemLie15}.

In more detail, for any four-dimensional $\mathcal{N}=2$ superconformal field theory (SCFT), 
there is a subsector which can be described by a two-dimensional 
vertex operator algebra (VOA).  
The 
normalized 
character of the corresponding VOA reproduces the special limit of the 
superconformal index, called the {\em Schur index}. 
On the one hand, four-dimensional SCFTs lead to some interesting 
conjectures for large classes of VOAs. 
For example, it is expected  \cite{BR}
that the Higgs branch of such a 4D theory is the associated variety $X_{V}$ of 
the corresponding VOA $V$.  
{This in particular implies all the VOAs coming from 4D theories are quasi-lisse.}
On the other hand, the representation theory of the  
VOA produces new physical observables of the 4D SCFT, 
such as the ordinary Schur index and the Schur index in the presence of boundary conditions, 
line defects and surface defects. 

One of the major advances in the last ten years is that one 
can engineer a large class of new 4D SCFTs by geometric methods. 

The classification of $\mathcal{N}=2$ rank one SCFTs have been studied based 
on the analysis of their Coulomb branch geometries and all possible deformations of
planar special K{\"a}hler singularities, 
labeled by their Kodaira type which are consistent 
with the low energy Dirac quantization condition \cite{ALLM1,ALLM2}.  
{Here the rank of a SCFT refers to the dimension of the Coulomb branch.}
One particularly interesting class of 
theories in this framework is the class of {Argyres--Douglas  theories} \cite{DG,DX} 
which cannot be studied like perturbating quantum field theory, 
since they are strongly coupled interacting 4D SCFT which have no known Lagrangian description in general. 

In \cite{LXY}, the authors found a universal formula for the rank of the theory so that a complete search is possible. 
	They listed all rank one, rank two, rank three Argyres--Douglas theories based on this formula 
	and found the corresponding VOA and the 
	associated Higgs branch for these theories. 
	{This classification gives some very interesting rank one SCFTs 
	such that the Higgs branches are not given by 
	one-instanton moduli spaces on $\mathbb{R}^4$ for a flavor symmetry group $G$}.  
{Those SCFTs with Higgs branches given by one-instanton moduli spaces for $G$ instantons are more easier to be understood than the general cases. They can be realized as the low-energy limit of the worldvolume theory on a single D3 brane probing a singularity in F-theory $7$ brane with gauge group $G$.}
	All these rank one Argyres--Douglas theories coincide with \cite{ALLM2} but arise from entirely different constructions. 	
	For exemple, the simple affine vertex algebras $L_{-2}(G_{2})$ and $L_{-2}(B_{3})$ 
	appear as the vertex operator algebras corresponding to
	rank one Argyres--Douglas theories 
	in four dimension with flavour symmetry $G_{2}$ and $B_{3}$.  
However, not all quasi-lisse VOAs have corresponding 4D counterparts.
So far, from the classification results of the 4D theory, we are no able to find the 4D theory such that the corresponding VOA is $L_{-1}(G_{2})$, $L_{-1}(B_{3})$ or $L_{-2}(B_{n})$,
	$n \geqslant 4$. 
	
	One expects \cite{SXY} that the representations of these vertex algebras are closely connected with 
	the Coulomb branch of the  circle compactified corresponding 4D theory;
	for $L_{-2}(D_4)$, the corresponding Coulomb branch is 
	related to \cite{GMN}
	the moduli space of the ${\rm SL}_2$-Higgs bundles on the sphere with four punctures. 
	However, for the other two theories it seems there are no precise description 
	of the corresponding Coulomb branches at the moment.
	
	In this context, 
	studying the representation theory of 
	these simple affine vertex algebras 
	becomes very important. 
	For example, {the decomposition in 
	Theorem \ref{Th:mainD} 
	suggests the decomposability of generalized Schur index.} 

\subsection{Connections with mathematical conjectures}
First of all, 
$-2$ is a not an admissible level for $G_2$. 
Therefore, Theorem \ref{Th:mainA} provides 
a new example of a vertex algebra whose associated variety 
has a finite number of symplectic leaves outside the admissible levels. 
Indeed, in the setting of affine vertex algebras associated with $\g$, 
this condition is equivalent to that of being contained in the nilpotent cone of $\g$ 
(see for instance \cite[Proposition 12.1]{AraMor-book}), 
and the symplectic leaves are nothing but the coadjoint orbits of 
$\g^*$, identified with the adjoint orbits of $\g$ through the Killing form. 

Vertex algebras whose associated variety 
has a finite number of symplectic leaves are 
referred to as {\em quasi-lisse vertex algebras} \cite{AK}, 
the lisse ones corresponding to the case where the associated variety 
has dimension zero. 
The following was conjectured in \cite{AraMor3}. 

\begin{mainconj}
\label{Conj:irr}
If $V$ is a simple quasi-lisse conformal vertex algebra, 
then $X_V$ is irreducible.
\end{mainconj} 
 
Theorem \ref{Th:mainA} thus gives a new example 
where Conjecture~\ref{Conj:irr} holds. 

Our result is also interesting in the context of orbifold vertex algebras. 
By \eqref{eq:D4_G2}, 
$\overline{\OO}_{\text{sreg}}$ is 
the orbifold of $\overline{\OO}_{\text{min}}$ 
by the symmetric group $\mf{S_3}$, the group of 
Dynkin diagram automorphisms 
of $D_4$. 
{As mentioned before, according to \cite{AP}, 
we have $L_{-2}(G_2) = L_{-2}(D_4)^{\mf{S}_3}$.}
Hence, Theorem \ref{Th:mainA} 
can be reformulated as follows: 
$$X_{V^{\mf{S}_3}} = X_{V} / {\mf{S}_3},$$
for $V:=L_{-2}(D_4)$. 
In general, {there is no reason 
for an arbitrary vertex algebra $V$ acted by a finite group $G$ 
that $X_{V^G}= X_V /G.$} 
The following was proved by Miyamoto in \cite{Miyamoto} though: 
if $V$ is a lisse simple conformal vertex 
algebra and $G$ is a finite solvable subgroup of the automorphism group of $V$, 
then $V^G$ is also lisse. 
Next conjecture is natural to expect, as suggested 
by Dra\v{z}en Adamovi\'{c}.

\begin{mainconj}
\label{Conj:orbifold}
Let $V=\bigoplus_{n \geqslant 0} V_n$ 
be a simple positively graded quasi-lisse vertex 
algebra such that $V_0\cong\C$ and $G$ a finite solvable automorphism group of $V$, 
then $V^G$ is also quasi-lisse. 
\end{mainconj}

Theorem \ref{Th:mainA} supports Conjecture \ref{Conj:orbifold}, 
and also the equalities \eqref{eq:AV_B3}. 
Indeed, by \cite{AP}, we also 
have $L_{-2}(B_3) = L_{-2}(D_4)^{\langle \sigma \rangle}$ where 
$\sigma$ is an order two element the group of 
Dynkin diagram automorphisms 
of $D_4$. 

Then, our result gives new evidences for the following conjecture stated in \cite{AvEM}. 

\begin{mainconj}
\label{Conj:fin_ext}
If $W$ is a finite extension of the vertex algebra $V$ then the corresponding
morphism of Poisson algebraic varieties 
$\pi \colon  X_W \to X_V$ is a dominant morphism. 
\end{mainconj}
As mentioned above, by \cite{AP}, $L_{-2}(D_4)$ 
is a finite extension of both $L_{-2}(G_2)$ 
and $L_{-2}(B_3)$ and the restriction of 
$\pi_2$ (resp.~$\pi_3$) 
to $\overline{\OO}_{\text{min}}$ 
is precisely the corresponding morphism 
between $X_{L_{-2}(D_4)}$ and $X_{L_{-2}(G_2)}$ 
(resp.~$X_{L_{-2}(B_3)}$). 
Since 
$$\dim \overline{\OO}_{\text{min}} 
= \dim \overline{\OO}_{\text{sreg}}
=\dim \overline{\OO}_{\text{short}} =10,$$
Theorem \ref{Th:mainA} and the equalities \eqref{eq:AV_B3} furnish  
new examples where Conjecture \ref{Conj:fin_ext} holds. 
Most examples so far occurred between 
simple affine vertex algebras and $\W$-algebras 
at admissible levels \cite{AvEM}. 

Finally, in the course of the proof 
of Theorem \ref{Th:mainA}, 
it will be proved that 
\begin{align}
\label{eq:H0}
H^{0}_{DS,f_{\text{sreg}}}(L_{-2}(G_2)) = \W_{-2}(G_2,f_{\text{sreg}}) 
\cong \C,
\end{align} 
where $H^0_{DS,f}(-)$ denotes the Drinfeld--Sokolov reduction 
with respect to the nilpotent element $f$ of $\g$, 
$\W_k(\g,f)$ is the simple quotient of the universal $\W$-algebra 
$\W^k(\g,f):=H^0_{DS,f}(V^k(\g))$ 
associated with $\g$ and $f$, 
and $f_{\text{sreg}}$ is an element of the subregular 
nilpotent orbit of $G_2$, see~\S\ref{Sub:W-algebras} and 
Section~\ref{Sec:G2_assoc}. 
The next conjecture ({\cite{KacRoaWak03,KacWak08})  
was proved for many cases, but mainly 
for $k$ an admissible level. 

\begin{mainconj}
\label{Conj:H0}
$H^0_{DS,f}(L_k(\g))$
is either zero or isomorphic to $\W_k(\g, f)$.
\end{mainconj}

The identities \eqref{eq:H0} 
give a new case where 
Conjecture \ref{Conj:H0} holds 
for a non-admissible level.

\subsection{Organization of the paper}
The rest of the article is organized as follows. 
Section~\ref{Sec:Pre} 
regroups a few preliminary results on 
the Zhu algebra 
and Zhu's correspondence, 
associated varieties and $\W$-algebras. 
We fix in this section the main notation 
of the article. 
In Section~\ref{Sec:G2_rep}, 
we study the representations in the category $\mc{O}$ 
of the simple affine $L_{-2}(G_2)$. 
This is based on the computation of a singular vector. 
The computation of the associated variety 
of $L_{-2}(G_2)$ is achieved in Section~\ref{Sec:G2_assoc}. 
Section~\ref{Sec:B3_rep} is about the 
representations of $L_{-2}(B_3)$. 
We first study the representations in 
the category $\mc{O}$ exploiting the results about $G_2$. 
Furthermore, we study 
non-ordinary modules using spectral 
flows from ordinary modules of $L_{-2}(B_3)$. 
There are two appendices: 
Appendix~\ref{App:singular_G2} gives the explicit 
formulas of a singular vector in $V^{-2}(G_2)$ 
and of its image in the Zhu algebra. 
Appendix~\ref{App:pol_B3} describes 
useful polynomials in the symmetric algebra of $B_3$ 
related to subsingular vectors in $V^{-2}(B_3)$.

\medskip

\noindent
Throughout the article, all Lie algebras are defined over $\C$ 
and all topological terms refer to the Zariski
topology.

\subsection*{Acknowledgements}
{First of all, we thank the anonymous referees for their helpful comments 
and suggestions.} 
T.A.~is partially supported by JSPS  KAKENHI Grant Numbers 21H04993 and 19KK0065.
X.D.~is supported by JSPS KAKENHI Grant Numbers 21H04993 and 23K19008.
J.F.~is supported by a University of Melbourne Establishment Grant. 
A.M.~is partially supported by the European Research Council (ERC) 
under the European Union's Horizon 2020 research innovation programme 
(ERC-2020-SyG-854361-HyperK). 
X.D. wishes to thank Bailin Song for useful discussions and suggestions, and thank Thomas Creutzig, Shashank Kanade, Jinwei Yang and Yongchang Zhu for valuable comments.
B.L. wishes to express his gratitude to his supervisor Wenbin Yan for long term support and encouragement. He especially thanks Yiwen Pan for useful discussions and many suggestions. He also acknowledges RIMS, Kyoto University for hospitality at an early stage of this work. 
{J.F. thanks David Ridout for interesting discussions on the construction of relaxed modules.}

\section{Preliminaries}
\label{Sec:Pre}
Let $\g$ be a simple Lie algebra with Killing form $\kappa_\g$ 
as in the introduction, 
and let 
$\extg=\g[t,t^{-1}] \oplus \C K\oplus \C D$ 
be the extended affine Kac-Moody Lie algebra
associated with $\g$ and the inner product 
$$(-|-)=\dfrac{1}{2h_\g^\vee} \times \kappa_\g,$$ 
with the commutation relations
\begin{align*}
 [x{(m)},y{(n})]=[x,y]{(m+n)}+ m(x|y)\delta_{m+n,0}K,\quad
 [D,x{(m)}]=m x{(m)},\quad
 [K,\extg]=0,
\end{align*}
for $m,n\in\Z$ and $x,y \in \g$, where $x{(m)}=x\otimes t^m$. 

Let $\affg= [\extg,\extg]  =\g[t,t^{-1}] \oplus \C K$. 
Fix  a triangular decomposition $\g= \n_+ \oplus \h \oplus \n_-$
so that 
$$\extg = \affn_- \oplus \exth \oplus \affn_+ \quad \text{and}
\quad  \affg = \affn_- \oplus \affh \oplus \affn_+$$
are triangular decompositions  
for $\extg$ and $\affg$, respectively, 
with $\affn_- = \n_-+ t^{-1} \g[t^{-1}]$, 
$\affn_+ = \n_+ + t \g[t]$,  
$\exth = \h \oplus \C K \oplus \C D$ 
and $\affh = \h \oplus \C K$. 
The Cartan subalgebra $\exth$ 
is equipped with a bilinear form extending 
that on $\h$ given by 
$$(K|D) =1,\quad (\h | \C K \oplus \C D) = (K|K) = (D|D) =0.$$
We write $\Lambda_0$ and $\delta$ for the elements 
of $\exth^*$ orthogonal to $\h^*$ 
and dual to $K$ and $D$, respectively. 
Let $\Delta$ be the root system of $(\g,\h)$ 
with basis $\Pi=\{\alpha_1,\ldots,\alpha_\ell\}$, 
and denote by 
$\theta$ the highest positive root. 
We write $\varpi_1,\ldots,\varpi_\ell$ for 
the fundamental weights of $\g$ with respect 
to $\alpha_1,\ldots,\alpha_\ell$, 
and $\Lambda_0,\Lambda_1,\ldots,\Lambda_\ell$ for 
those of $\extg$. 

For $k\in \C$, set
\begin{align*}
 V^k(\g)=U(\extg)\otimes_{U(\g[t] \oplus  \C K \oplus \C D)}\C_k,
\end{align*}
where 
$\C_k$ is the one-dimensional representation of
$\g [t]\oplus \C K \oplus \C D$ on which $\g [t]\oplus \C D$ acts trivially 
and $K$ acts as multiplication by $k$.
The space $V^k(\g)$ is naturally a vertex algebra, 
called the {\em universal affine vertex algebra associated with
$\g $ at level $k$}. 
By the PBW theorem, we have $V^k(\g ) \cong U(\g [t^{-1}]t^{-1})$ as $\C$-vector spaces. 

The vertex algebra $V^k(\g )$ is graded by $D$:
\begin{align*}
V^k(\g )=\bigoplus_{d\in \Z_{\geqslant 0}}V^k(\g )_d,\quad 
V^k(\g )_d=\{a\in V^k(\g )\colon { D a = - da}\}. 
\end{align*} 
This grading gives a conformal structure provided that $k$ 
is not critical, that is, $k \not=-h_\g^\vee$. 
A~$V^k(\g )$-module is the same as a smooth $\extg$-module of level
$k$, where a $\extg$-module $M$ is called smooth if 
$x{(n)}m=0$ for $n$ sufficiently large for all $x\in \g$, $m\in M$.

\subsection{Singular vectors and highest-weight modules}

For each $\alpha\in\Delta$, fix a nonzero root vector $e_\alpha$. 
Recall that a vector $v \in V^{k}(\g)$ 
is called {\em singular} if $e_{\alpha}{(0)} v=0$ for all $\alpha \in \Pi$ 
and $e_{-\theta}{(1)} v=0$. 
In other words, $v$ is a singular vector if $v$ is singular for $\extg$ with 
respect to $\affn_+$.  
If $v$ is singular for $V^k(\g)$, denote by
$\langle v \rangle $ the ideal in $V^{k}(\g)$ generated by $v$, 
that is, $ \langle v \rangle  = U(\extg) v$. 
We set 
\begin{align}
\label{eq:tildeV}
\widetilde{V}_{k}\left( \g\right)=V^{k}(\g)/ \langle v \rangle, 
\end{align}
the associated quotient vertex algebra. 

Let $L_k(\g )$ be the unique simple graded quotient of $V^k(\g)$. 
As a $\extg$-module, $L_k(\g)$ is isomorphic to the irreducible highest-weight
representation of $\extg$ with highest-weight $k\Lambda_0$. 
If $N_k$ denotes the unique maximal ideal of $V^k(\g)$, then 
\begin{align*}
 L_k(\g)=V^k(\g)/N_k, 
\end{align*} 
and $L_k(\g)$ is a quotient of $\widetilde{V}_{k}\left( \g\right)$. 
We will also make use of the notion of {\em subsingular vector}. 

\begin{defn}
\label{Def:subsingular} 
A vector $v_{\text{\em sub}} \in N_k$ is 
{\em subsingular} if there exists 
a proper submodule $N'_k$ 
of $N_k$ such that the following conditions hold:
\begin{align*}
v_{\text{\em sub}} \notin N'_k, 
\quad e_\alpha(0) v_{\text{\em sub}} \in N'_k \quad \text{for all}
\; \alpha \in \Pi, 
\quad 
e_{-\theta}(1)  v_{\text{\em sub}}  \in N'_k.
\end{align*}
\end{defn}
Note that the image of a subsingular vector in 
$V^k(\g)/N'_k$ is a singular vector 
of $V^k(\g)/N'_k$. 

For $\lam \in \h^*$, we denote by $L_\g(\lam )$ 
the irreducible highest-weight representation of $\g$ 
with highest-weight $\lam $.  
Similarly, for $\tilde{\lam } \in \exth^*$ we denote by $L_{\extg}(\tilde{\lam })$  
the irreducible highest-weight representation of $\extg$. 
In the case where $\tilde{\lam } = \lam  + k \Lambda_0$, 
we shall sometimes write $L_\g(k,\lam)$ instead of $L_{\extg}(\tilde{\lam})$. 
In this way, we have 
$$L_k(\g) = L_{\extg}(k \Lambda_0) = L_{\g}(k,0).$$

A finitely generated module $M$ over a conformal vertex algebra 
$V$ is called {\em ordinary} 
if $L_0$ acts semisimply, $M_d$ being finite-dimensional 
for all $d$, 
where 
$$M_d = \{m \in M\colon L_0 m =dm\},$$
and the conformal weights of $M$ 
are bounded from below, i.e. there exists $d_0$ so that $M_d=0$ 
for $d\leqslant d_0$. 
Call the {\em conformal dimension} 
of a simple ordinary $V$-module $M$ 
the minimum conformal weight of $M$. 
More generally, a $V$-module $M$ 
is said to be of {\em positive energy} if it 
is $\Z_{\geqslant 0}$-graded, 
$M=\bigoplus\limits_{d \in \Z_{\geqslant 0}} 
M_{d_0+d}$, with $M_{d_0}\not=0$,  
such that  
$a{(n)}M_k \subset M_{k-n}$, 
where for $a \in V$ 
of conformal weight $\Delta$ 
we write $a(z)=\sum\limits_{n\in\Z}a(n)z^{-n-\Delta}$. 

The highest-weight $\extg$-module $L_\g(k,\lam)$,  
regarded as a 
$V^k(\g)$-module,  
has conformal dimension 
\begin{align}
\label{eq:conf_dim}
h_{L(\lam)} =\dfrac{(\lam|\lam+2\rho)}{2(k+h_\g^\vee)},
\end{align} 
where $\rho$ is the half--sum of positive roots. 

\subsection{Zhu's algebra and the characteristic variety}
For a positively $\Z$-graded vertex algebra $V=\bigoplus_{d}V_d$, 
let $A(V)$ be  the Zhu algebra of $V$, 
\begin{align*}
A(V)=V/V\circ V,
\end{align*}
where $V\circ V$ is the $\C$-span of the vectors
\begin{align*}
a\circ b:=\sum_{i\geqslant 0}\begin{pmatrix}
			 \Delta\\ i
			\end{pmatrix}a_{(i-2)}b
\end{align*}
for $a\in V_{\Delta}$, $\Delta\in \Z_{\geqslant 0}$, $b \in V$,
and
$V\rightarrow  ({\rm End} V)[\![z,z^{-1}]\!]$,  $a\mapsto \sum_{n\in \Z}a_{(n)}z^{-n-1}$,
denotes the state-field correspondence. 
The space $A(V)$ is a unital associative algebra
with respect to the multiplication defined by
\begin{align*}
 a * b:=\sum_{i\geqslant 0}\begin{pmatrix}
			 \Delta\\ i
			\end{pmatrix}a_{(i-1)}b
\end{align*}
for $a\in V_{\Delta}$, $\Delta\in \Z_{\geqslant 0}$, $b \in V$. 
Denote by $[a]$ the image of $a \in V$ in $A(V)$.

Let $M=\bigoplus\limits_{d \in \Z_{\geqslant 0}}M_{d_0+d}$, with $M_{d_0}\ne 0$,
be a positive energy representation of $V$. 
Then $A(V)$ naturally acts on its top weight space
$M_{\text{top}}:=M_{d_0}$, and
the correspondence $M\mapsto M_{top}$ defines a bijection between
isomorphism 
classes of
simple positive energy representations of $V$ and simple 
$A(V)$-modules \cite{Z}. 

The Zhu algebra
 $A(V^k(\g))$ is naturally isomorphic to 
the universal enveloping algebra $U(\g)$ \cite{FZ}, 
where the isomorphism 
$F \colon A(V^{k}(\g)) \to U(\g)$ is given by 
\begin{align} 
\label{eq:Zhu} 
F ([a_{1}{(-n_1-1)} \ldots a_m{(-n_m-1)} 
\mathbf{1}]) = (-1)^{n_1+ \cdots + n_m} a_m \ldots a_1,
\end{align}
for $a_1, \ldots ,a_m \in \g$ 
and $n_1, \ldots , n_m \in \mathbb{Z}_{\geqslant 0}$.

We have an exact sequence
$$A(N_k)\rightarrow U(\g)\rightarrow A(L_k(\g))\rightarrow 0$$ since the functor $A(-)$ is
right exact, and thus
$A(L_k(\g))$ is the quotient of $U(\g)$ by the image
$J_k$ of the maximal ideal $N_k$ in $A(V^k(\g))=U(\g)$:
\begin{align*}
 A(L_k(\g))=U(\g)/J_k. 
\end{align*} 
In particular, if $v$ is a singular vector, 
$$A(\widetilde{V}_{k}(\g)) \cong U(\g)/ \langle v'  \rangle ,$$
where $\langle v'  \rangle $ is the two-sided ideal in $U(\g)$ generated by the vector 
$$v':= F([v]).$$
The top degree component 
of $L_{\extg} (\lambda)$ is 
 $L_{\g}(\bar \lambda)$,
where $\bar \lambda$ is the 
restriction of $\lambda$ to $\mf{h}$.
Hence,
by Zhu's correspondence, 
a level $k$ representation 
$L_{\extg}(\lambda)$, that is $\lambda(K)=k$, 
is an $L_k(\g)$-module if and only if
$J_k L_{\g}(\bar\lambda)=0$.

Set $U(\g)^\h := \{u \in U(\g) \colon [h,u]=0 \text{ for all } h \in \h \}$ and 
let 
\begin{align}
\label{eq:Harish-Chandra}
\Upsilon \colon U(\g)^\h \to U(\h)
\end{align}
be the {\em Harish--Chandra projection map} which is the 
restriction of the projection map $U(\g)=U(\h) \oplus (\n_- U(\g) + U(\g)\n_+) \to U(\h)$ 
to $U(\g)^\h$. 
It is known that $\Upsilon$ is an algebra homomorphism. 
For a two-sided ideal $I$ of $U(\g)$, the {\em characteristic variety of $I$} 
is defined as \cite{Jos}:
$$\mathscr{X}(I) = \{\lambda \in \h^* \colon p(\lam)=0 \text{ for all } p\in \Upsilon(I^\h)\},$$
where $I^\h =I\cap U(\g)^\h$.  
}
Identifying $\g^*$ with $\g$ through $(-|-)$, and thus $\h^*$ with $\h$, 
we view $\mathscr{X}(I)$ as a subset of $\h$. 	
{Then using \cite{Jos} (see also \cite[Lemma 2.1]{Ara}), it is easy to see that 
for $\lambda\in\h^{*}$, $\lambda\in\mathscr{X}(I)$ if and only if 
$I L_\g(\lambda)=0$.} 
In other words, the characteristic variety $\mathscr{X}(I)$ classifies the simple
$U(\g)/I$-modules in category $\mc{O}^\g$, 
where $\mc{O}^\g$ is the BGG category $\mc{O}$ of $\g$. 

According to \cite{A-PhD, AM, Ara}, we have the following result. 
\begin{prop} 
\label{Class-affine-quotient}
Let $v \in V^k(\g)$ 
be a singular vector, $\widetilde{V}_{k}( \g) = V^k(\g)/\langle v\rangle $ 
as in \eqref{eq:tildeV}, 
$v':=F([v])$ 
the corresponding image in $U(\g)$ 
and $R$ the $U(\g )$-submodule
of $U(\g )$ generated by the vector~$v'$. 
The following statements are equivalent:
\begin{enumerate}[{\rm (i)}]
\item $L_\g(\mu)$ is an $A( \widetilde{V}_{k}( \g) ) $-module,
\item $RL_\g(\mu)=0$,
\item $R^{\h}v_{\mu}=0$, where $R^{\h}:=R \cap \U(\g)^\h$,
\end{enumerate}
where $v_\mu$ is a highest-weight vector of $L_\g(\mu)$. 
\end{prop}
		
In the notation of the Proposition~\ref{Class-affine-quotient}, 
given $r \in R^{\h}$, there exists a unique polynomial
$p_{r}\in \Upsilon(R^{\h})$ such that $rv_{\mu}=p_{r}(\mu)v_{\mu} 
$. 
Define the polynomial set of $\h$ by 
\begin{align}
\label{eq:P_0}
\mathscr{P}_{v} =\{ p_{r} \colon r \in R^{\h} \}.
\end{align} 
If $v$ is a subsingular vector, one can define similarly 
$\mathscr{P}_{v}$ using 
the $U(\g )$-submodule
of $U(\g )$ generated by the vector~$v':=F([v])$. 

As a consequence of Proposition \ref{Class-affine-quotient}, 
we obtain: 

\begin{cor} 
\label{zhuclass} 
Let $v \in V^k(\g)$ 
be a singular vector and {$\widetilde{V}_{k}( \g) = V^k(\g)/  \langle v\rangle $. }
There is a one-to-one correspondence between 
the irreducible $A(\widetilde{V}_{k}( \g))$-modules
in the category $\mathcal{O}^\g$ 
and 
the weights $\mu \in {\h}^{*}$ such that
$p(\mu)=0$ for all $p \in \mathscr{P}_{v}$.
\end{cor}

Define the left-adjoint action on
$U(\g )$ by 
\begin{align}
\label{eq:left-adj}
x_L f=[x,f] \text{ for }x \in
\g \text{ and }f \in U(\g ).
\end{align} 
This action extends to $\U(\g)$ and we still denote it by $x_L f$ 
for $x\in U(\g)$ and $f \in U(\g)$. 

\subsection{Associated variety}
\label{Sub:assoc}
As in the introduction, 
let $X_V$ be the {\em associated variety} \cite{Ara12} 
of a vertex algebra $V$, 
that is the reduced scheme associated with the {\em Zhu $C_2$-algebra} of $V$ 
$$R_V := V/C_2(V),$$
with $C_2(V) ={\rm span}_\C\{a_{(-2)}b \colon a,b \in V\}$. 
In the case that $V$ is a quotient of $V^k(\g)$,
$V/C_2(V)=V/\g [t^{-1}]t^{-2} V$ 
and we have a surjective Poisson algebra homomorphism
\begin{align}
\label{eq:Sym}
 \C[\g^*]=S(\g)\longtwoheadrightarrow R_V=V/\g [t^{-1}]t^{-2}V,\quad
x\mapsto \overline{x(-1)}+\g [t^{-1}]t^{-2}V,
\end{align}
where $\overline{x{(-1)}}$ denotes the image of $x{(-1)}$ in the quotient 
$R_V$. 
Then $X_V$ is just the zero locus of the kernel of the above map in
$\g^*$. 
It is a $G$-invariant and conic subvariety 
of $\g^*$, with $G$ the adjoint group of~$\g$. 
As for the characteristic variety, 
identifying $\g^*$ with $\g$ through $(-|-)$, 
we view it as a subset of $\g$. 

For $V=V^k(\g)$, we get 
\begin{align*}
R_{V^{k}(\g)} \cong S(\g)
\end{align*}
under the algebra isomorphism \eqref{eq:Sym}. 
For $v \in V^{k}(\g)$, denote by $v''$ the image of $\overline{v}$ 
in $S(\g)$ by the above isomorphism. 
If $v$ is a singular vector of $V^k(\g)$, then 
\begin{align*}
R_{\widetilde{V}_{k}(\g)} \cong S(\g) / I_{M},
\end{align*}
where $M$ is the $\g$-module generated by $v''$ under the adjoint action, 
and $I_M$ is the ideal of $S(\g)$ generated by $M$.

It will be also useful to consider the {\em Chevalley projection map} 
\begin{align}
\label{eq:Chevalley}
\Psi \colon S(\g)^{\h}\rightarrow S(\h), 
\end{align} 
where $S(\g)^{\mathfrak{h}}= 
\{x\in S(\g) \colon [h,x]=0 \,\text{ for all } \, h \in\mathfrak{h}\}$. 
{This is the restriction to $S(\g)^{\h}$ of the 
projection map from $S(\g)$ 
to $S(\h)$ with respect to the decomposition 
$S(\g) = S(\h) \oplus \left(\n_- S(\g) + S(\g) \n_+\right)$.}

\subsection{Affine $\W$-algebras}
\label{Sub:W-algebras}
For a nilpotent element  $f$ of $\g$,
let $\W^k(\g,f)$ be the universal 
$\W$-algebra associated with
$(\g,f)$ at level $k$,
defined by the generalized quantized Drinfeld--Sokolov reduction
\cite{FF90,KacRoaWak03}:
\begin{align*}
\W^k(\g,f)=H^0_{DS,f} (V^k(\g)), 
\end{align*}
where $H^0_{DS,f}(M)$ is the corresponding BRST
cohomology with coefficient in a $\extg$-module $M$.
We have a natural Poisson algebra 
isomorphism
$R_{\W^k(\g,f)}\cong \C[\mathscr{S}_{f}]$, 
where $\mathscr{S}_{f} =f+\g^{e}$, with $\g^{e}=\{x\in\g\colon [x,e]=0\}$,  
is the Slodowy slice associated with an $\mf{sl}_2$-triple $(e,h,f)$ \cite{DeSole-Kac,Ara15}.
It follows that
\begin{align*}
X_{\W^k(\g,f)} \cong \mathscr{S}_{f}.
\end{align*}
Let $\W_k(\g,f)$ be the unique simple quotient of $\W^k(\g,f)$.
Then $X_{\W_k(\g,f)}$
 is a $\C^*$-invariant closed Poisson subvariety of $\mathscr{S}_f$.
Let $\mc{O}_k$ be the category $\mc{O}$ of
$\extg$ at level $k$.
We have a functor
\begin{align*}
\mc{O}_k \rightarrow  \W^k(\g,f)\text{-Mod} ,\quad M\mapsto
H^0_{DS,f}(M),
\end{align*}
where 
$\W^k(\g,f)\text{-Mod}$ denotes the category
of $\W^k(\g,f)$-modules. 
According to \cite{Ara15}, 
for any quotient $V$ of $V^k(\g)$, 
 $X_{H^0_{DS,f}(V)}$ is isomorphic, as a Poisson variety, to the 
intersection
 $X_{V}\cap \mathscr{S}_f$. 
 In particular, 
$H^0_{DS,f}(V) \ne 0$ if and only if
$\overline{G.f}\subset X_V$  
and $H^0_{DS,f}(V)$ is lisse if 
$X_V = \overline{G.f}$.

\section{On the representations of $L_{-2}(G_2)$}
\label{Sec:G2_rep}
In this section, $\g$ is the simple exceptional Lie algebra of type $G_2$ 
with simple roots $\alpha_1,\alpha_2$ 
and Dynkin diagram 
$$\begin{Dynkin}
\Dbloc{\Dcirc\Deast\Ddoubleeast\Dtext{t}{\alpha_1}}
\Dleftarrow\Dbloc{\Dcirc\Dwest\Ddoublewest\Dtext{t}{\alpha_2}}
\end{Dynkin}$$
In particular, $\alpha_1$ is the simple short root. 
One can fix the root vectors so that 
\begin{align*}
&[e_{\alpha_1},e_{\alpha_2}]=e_{\alpha_1+\alpha_2}, 
\quad  [e_{\alpha_1},e_{\alpha_1+\alpha_2}]=2e_{2\alpha_1+\alpha_2}, &\\
& [e_{\alpha_1},e_{2\alpha_1+\alpha_2}]=3e_{3\alpha_1+\alpha_2},
\quad  [e_{\alpha_2},e_{3\alpha_1+\alpha_2}]=e_{3\alpha_1+2\alpha_2}.&
\end{align*} 
All other commutation relations can be obtained by using the Jacobi identity. 
It will be convenient to number the other positive roots as follows:
$$\alpha_3=\alpha_1+\alpha_2,\quad  
\alpha_4 = 2\alpha_1+\alpha_2, \quad  
\alpha_5 = 3\alpha_1 +\alpha_2,\quad  
\alpha_6=3\alpha_1+2\alpha_2=\theta.
$$
Denote by $\varpi_1=2\alpha_1+\alpha_2,\varpi_2=3\alpha_1+2\alpha_2$ the fundamental weights of $G_2$ 
with respect to $\alpha_1,\alpha_2$, and by 
$\{h_1=\alpha_1^\vee,h_2=\alpha_2^\vee\}$ 
a basis of the Cartan subalgebra. 
The Weyl group $\weyl_{G_2}$ of $G_2$ is the dihedral group of order 12 generated by the Weyl reflections $s_{i}\in\weyl_{G_2}$ ($i=1,\dots,6$).

\begin{thm} 
\label{Th:singular_G2}
There is a singular vector $v_{\sing}$ of $V^{-2}(G_2)$ 
of weight $-2\Lambda_{0}+4\varpi_{1}-6\delta$. 
In particular, $v_{\sing}$ has conformal weight six  
and there is no singular vector of conformal weight strictly smaller than six.
\end{thm}

Due to its complexity, 
we leave the explicit form of such a singular 
vector to Appendix~\ref{App:singular_G2}.

\begin{proof}
The affine space $\{\lam + k\Lambda_{0} \colon \lam \in \h^{*}\}$ 
is identified with an affine subvariety of $\exth^*$ 
by the correspondence 
$$
\lambda+k\Lambda_{0} \longmapsto \lambda+k \Lambda_{0}-h_{L(\lambda)}\delta, 
$$
where $h_{L(\lambda)}$ is the conformal dimension given by 
\eqref{eq:conf_dim} and  
$h_{G_2}^\vee=4$. 

The strategy in order to find a singular 
vector is the following.  
We search for a $G_2$-weight of a potential singular vector 
$v$ that makes the conformal dimension an integer. 
Then we test the conditions of a singular vector,  
$$e_{\alpha_1}(0)v =0,\quad e_{\alpha_2}(0)v=0,
\quad e_{-\theta}(1)v=0,$$
with $\theta = 3\alpha_1+2\alpha_2$ 
from smaller to larger conformal dimensions. 
For any $\lambda= a_1 \varpi_{1}+a_2 \varpi_{2}$, we have 
$$
{h_{L(\lambda)}=\frac{a_1^2}{6}+\frac{a_1 a_2}{2}+\frac{5 a_1}{6}+\frac{a_2^2}{2}+\frac{3 a_2}{2}.}
$$
We list the integer solutions for the conformal dimension from $2$ to $6$ in Table \ref{table:1}.
\begin{table}[htb]   
	\begin{center}   
		\begin{tabular}{|c|c|}   
			\hline   conformal dimension & weight  \\   
			\hline   0 & 0\\
			\hline   1 & $\varpi_1$\\
			\hline   2 & $\varpi_{2}$ \\ 
			\hline   3 & no solution \\  
			\hline   4 & 3$\varpi_{1}$    \\ 
			\hline   5 & 2$\varpi_{2}$ 	  \\
			\hline   6 & 4$\varpi_{1}$     \\
			\hline   
		\end{tabular} \\[0.25em]
		\caption{The integer solutions for the conformal dimension}  
	\label{table:1}  
	\end{center}   
\end{table}	
We observe that there is no corresponding solution 
for a singular vector with conformal weight $1,2,4,5$. 
However, there is indeed a singular vector with conformal weight 6. 
Our candidate, that we denote by $v_\sing$, 
is described in Appendix~\ref{App:singular_G2} and has $G_2$-weight 
$4\varpi_1$. 		
Then it is straightforward 
to check that 
$$e_{\alpha_1}(0) v_{\sing}=0,\quad  
e_{\alpha_2}(0)v_{\sing}, 
\quad e_{-\theta}(1)v_{\sing}=0.$$
Alternatively, we can compare the first few terms of 
the character formulas of 
$V^{-2}(G_2)$ and $L_{-2}(G_2)$ by using Kazhdan--Lusztig 
polynomials to determine the existence 
of a singular vector with conformal weight 6.\footnote{This was 
suggested to one of the authors by Yiwen Pan in a private communication.} 
\end{proof}

Keep the notation relative to $v_\sing$ as in Section \ref{Sec:Pre}: 
$[v_\sing]$ denotes its image in the Zhu algebra of $V^{-2}(G_2)$, 
$v'_\sing$ the corresponding element of $U(G_2)$ via the 
isomorphism~\eqref{eq:Zhu}, $\overline{v}_\sing$ the image of $v_\sing$ 
in the Zhu $C_2$-algebra and $v''_\sing$ the 
corresponding element of $S(G_2)$ through the isomorphism 
\eqref{eq:Sym}.

Let also $\langle v_{\sing}\rangle  $ be the submodule of $V^{-2}(G_{2})$ 
generated by $v_{\sing}$, and 
$\widetilde{V}_{-2}(G_{2})=V^{-2}(G_{2})/\langle v_{\sing}\rangle  $ 
the associated quotient vertex algebra. 
Then the Zhu algebra $A(\widetilde{V}_{-2}(G_{2}))$ is isomorphic to 
$U(G_{2})/\langle v'_{\sing}\rangle $, where $\langle v'_{\sing}\rangle  $ is 
the two-sided ideal in $U(G_{2})$ generated by the vector $v'_{\sing}$. 
 The explicit form of  $v'_{\sing}$ can be found in Appendix~\ref{App:singular_G2} 
 as well.

\begin{lem}
\label{lm:weightzero_G2}
The zero-weight subspace 
$L_{G_2}(4\varpi_{1})^{\h}$ 
has dimension $8$, 
spanned by the linearly independent vectors below
\begin{align*}
		&v_{1}=(e_{-\alpha_4}^4)_{L}v'_{\sing}, \quad v_{2}=(e_{-\alpha_3}e_{-\alpha_4}^2e_{-\alpha_5})_{L}v'_{\sing}, \quad v_{3}=(e_{-\alpha_3}^2e_{-\alpha_5}^2)_{L}v'_{\sing}\\
		&v_{4}=(e_{-\alpha_2}e_{-\alpha_4}e_{-\alpha_5}^2)_{L}v'_{\sing}, \quad v_{5}=(e_{-\alpha_1}e_{-\alpha_4}^2e_{-\theta})_{L}v'_{\sing}\\
		&v_{6}=(e_{-\alpha_1}e_{-\alpha_3}e_{-\alpha_5}e_{-\theta})_{L}v'_{\sing}, \quad v_{7}=(e_{-\alpha_1}^2 e_{-\theta}^2)_{L}v'_{\sing}, \quad v_{8}=(e_{-\alpha_1}e_{-\alpha_3}e_{-\alpha_4}^3)_Lv'_{\sing}
		\end{align*}
	\end{lem}
  
\begin{lem}
Let $p_{i}\in U(\h)$ be the image of $v_{i}$ for $i=1,\dots,8$ 
by the Harish--Chandra projection~\eqref{eq:Harish-Chandra}. 
Then, the polynomial set $\{p_{1},\dots,p_{7}\}$ 
forms a basis for $\mathscr{P}_{v_{\text{\em sing}}}$, 
where 
{\footnotesize
			\begin{align*}
				p_{1}(h)=&-24 \left(2 h_1+3 h_2+3\right) (8 h_1^5+60 h_2 h_1^4+36 h_1^4+178 h_2^2 h_1^3+212 h_2 h_1^3+24 h_1^3+261
				h_2^3 h_1^2\\
				&+438 h_2^2 h_1^2+48 h_2 h_1^2-44 h_1^2+189 h_2^4 h_1+369 h_2^3 h_1-28 h_2^2 h_1-152 h_2 h_1-24
				h_1+54 h_2^5\\
				&+108 h_2^4-42 h_2^3-96 h_2^2-24 h_2)\\
				p_{2}(h)=&\; 2 (32 h_1^6+294 h_2 h_1^5+192 h_1^5+1081 h_2^2 h_1^4+1398 h_2 h_1^4+312 h_1^4+2070 h_2^3 h_1^3+3928 h_2^2
				h_1^3\\
				&+1542 h_2 h_1^3-32 h_1^3+2223 h_2^4 h_1^2+5346 h_2^3 h_1^2+2557 h_2^2 h_1^2-630 h_2 h_1^2-360 h_1^2+1296 
				h_2^5 h_1\\
				&+3582 h_2^4 h_1+1656 h_2^3 h_1-1554 h_2^2 h_1-1164 h_2 h_1-144 h_1+324 h_2^6+972 h_2^5+396 h_2^4\\
				&-828h_2^3-720 h_2^2-144 h_2)\\
				p_{3}(h)=&\; -4(16 h_1^6+161 h_2 h_1^5+96 h_1^5+634 h_2^2 h_1^4+763 h_2 h_1^4+156 h_1^4+1254 h_2^3 h_1^3+2260 h_2^2
				h_1^3\\
				&+865 h_2 h_1^3-16 h_1^3  +1332 h_2^4 h_1^2+3123 h_2^3 h_1^2+1558 h_2^2 h_1^2-271 h_2 h_1^2-180 h_1^2+729
				h_2^5 h_1\\
				&+2016 h_2^4 h_1+1041 h_2^3 h_1-708 h_2^2 h_1-582 h_2 h_1-72 h_1+162 h_2^6+486 h_2^5+198 h_2^4\\
				&-414h_2^3-360 h_2^2-72 h_2)\\
				p_{4}(h)= &\; 6 h_1 h_2 \left(h_1+h_2+1\right) \left(h_1+2 h_2+2\right) \left(h_1+3 h_2+3\right) \left(2 h_1+3 h_2\right)\\
				p_{5}(h) = & \; 2 h_1 \left(h_1+2 h_2+2\right)(32 h_1^4+218 h_2 h_1^3+128 h_1^3+555 h_2^2 h_1^2+614 h_2 h_1^2+56 h_1^2+612
				h_2^3 h_1\\
				&+882 h_2^2 h_1-10 h_2 h_1 -144 h_1+243 h_2^4+360 h_2^3-249 h_2^2-390 h_2-72)\\ 
				p_{6}(h)=&-2 h_1 \left(h_1+2 h_2+2\right)(16 h_1^4+109 h_2 h_1^3+64 h_1^3+264 h_2^2 h_1^2+289 h_2 h_1^2+28 h_1^2+279
				h_2^3 h_1\\
				&+396 h_2^2 h_1-17 h_2 h_1-72 h_1+108 h_2^4+153 h_2^3-132 h_2^2-189 h_2-36)\\
				p_{7}(h)=&-4(h_1-1) h_1(h_1+2 h_2+2)(h_1+2 h_2+3)(16 h_1^2+63 h_2 h_1+32 h_1+63
				h_2^2+63 h_2+12)\\
				p_{8}(h)=&\; 6(64 h_1^6+534 h_2 h_1^5+384 h_1^5+1829 h_2^2 h_1^4+2556 h_2 h_1^4+624 h_1^4+3168 h_2^3 h_1^3+6210 h_2^2 h_1^3\\
				&+2594 h_2
				h_1^3-64 h_1^3+2727 h_2^4 h_1^2+6426 h_2^3 h_1^2+2863 h_2^2 h_1^2-1412 h_2 h_1^2-720 h_1^2+918 h_2^5 h_1\\
				&+2322 h_2^4 h_1+402 h_2^3
				h_1-2538 h_2^2 h_1-1824 h_2 h_1-288 h_1)
\end{align*}}
\end{lem}

\begin{proof}
According to Lemma \ref{lm:weightzero_G2}, 
we have $\text{dim}\,L_{G_2}(4\varpi_{1})^{\h}=8$. 
Furthermore, one obtains by direct calculations that 
$v_{i}\equiv p_{i}(h) \mod \n_{-}U(G_{2})+U(G_{2})\n_{+}$ 
for $i=1,\ldots,8$. 
It is easily checked that $\{p_1,\ldots,p_8\}$ 
is linearly dependent, 
whereas $\{p_{1},\ldots,p_{7}\}$ 
form a linearly independent set.
\end{proof}

Corollary \ref{zhuclass} implies that the highest-weights $\lambda$ of irreducible 
$A(\widetilde{V}_{-2}(G_{2}))$-modules from the category $\mathcal{O}$ are given 
by the solutions of the polynomial equations:
\begin{align*}
\lambda(p_{i}(h))=0, \quad i=1,2,\ldots,7.
\end{align*}

\begin{prop}
\label{Pro:mu_i-G2}
The complete list of irreducible $A(\widetilde{V}_{-2}(G_{2}))$-modules in the category 
$\mathcal{O}$ is given by the  set 
$\{L_{G_2}(\mu_{i}) \colon i=1,2,\ldots,20\},$
where the $\mu_i$'s are given by Table \ref{Tab:mu_G2}. 

\medskip

\begin{table}[htb] 
\begin{center}
\begin{tabular}{|l|l||l|l|}
\hline
$\mu_{1}$ & $0$ & $\mu_{11}$ & $-\frac{1}{3}\varpi_{2}$\\[0.1em]
$\mu_{2}$ & $\varpi_{1}$ & $\mu_{12}$ & $-\frac{2}{3}\varpi_{2}$\\[0.1em]
$\mu_{3}$ & $\varpi_{2}$ & $\mu_{13}$ & $-\frac{3}{2}\varpi_{1}+\frac{1}{2}\varpi_{2}$\\[0.1em]
$\mu_{4}$ & $-2\varpi_{1}$ & $\mu_{14}$ & $-\frac{1}{2}\varpi_{1}-\frac{1}{2}\varpi_{2}$\\[0.1em]
$\mu_{5}$ & $-3\varpi_{1}$ & $\mu_{15}$ & $\varpi_{1}-\frac{3}{2}\varpi_{2}$\\[0.1em]
$\mu_{6}$ & $-\varpi_{2}$ & $\mu_{16}$ & $\varpi_{1}-\frac{4}{3}\varpi_{2}$\\[0.1em]
$\mu_{7}$ & $-2\varpi_{2}$ & $\mu_{17}$ & $\varpi_{1}-\frac{2}{3}\varpi_{2}$\\[0.1em]
$\mu_{8}$ & $\varpi_{1}-2\varpi_{2}$ & $\mu_{18}$ & $2\varpi_{1}-\frac{5}{3}\varpi_{2}$\\[0.1em]
$\mu_{9}$ & $-\frac{1}{2}\varpi_{1}$ & $\mu_{19}$ & $2\varpi_{1}-\frac{4}{3}\varpi_{2}$\\[0.1em]
$\mu_{10}$ & $-\frac{3}{2}\varpi_{1}$ & $\mu_{20}$ & $3\varpi_{1}-\frac{5}{2}\varpi_{2}$\\[0.1em]
\hline
\end{tabular}\\[0.2em]
\caption{The weights $\mu_i$ for $G_2$}
\label{Tab:mu_G2}
\end{center}
\end{table}

\end{prop}

From Zhu's correspondence, we deduce the following result. 

\begin{thm} \label{theorem1.5}
The set $\{L_{G_{2}}(-2,\mu_{i})\colon i=1,\ldots,20\}$ 
provides the complete list of irreducible $\widetilde{V}_{-2}(G_{2})$-modules 
from the category $\mathcal{O}$, 
and 
the set
$\{L_{G_{2}}(-2,0), \, L_{G_{2}}(-2,\varpi_{1}), \, L_{G_{2}}(-2,\varpi_{2})\}$
provides the complete list of irreducible ordinary modules for $\widetilde{V}_{-2}(G_{2})$.
\end{thm}

\begin{proof}
	The first part of the Theorem follows directly from the Proposition \ref{Pro:mu_i-G2} and Zhu's correspondence. 
	We look for the irreducible ordinary $\widetilde{V}_{-2}(G_{2})$-modules among the list of modules in the category $\mathcal{O}$. An ordinary $\widetilde{V}_{-2}(G_{2})$-module $M$ has finite-dimensional graded spaces.  
In particular, the space corresponding to the 
graded space 
with the minimal conformal weight 
is a finite-dimensional $G_2$-module. 
Hence, the irreducible ordinary $\widetilde{V}_{-2}(G_{2})$-modules 
correspond exactly to the dominant integral weights in the Table~\ref{Tab:mu_G2}.
\end{proof}

\begin{thm}\label{theorem1.6}
The vertex algebra $\widetilde{V}_{-2}(G_{2})$ is simple, and hence $\widetilde{V}_{-2}(G_{2})=L_{-2}(G_2)$.
\end{thm}
\begin{proof}
According to Theorem \ref{theorem1.5}, the only possible 
$G_2$-weights for a subsingular 
vector (see Definition \ref{Def:subsingular}) in $V^{-2}(G_2)$ 
{with respect to $\langle v_\sing\rangle$} 
besides 
$0$ 
are $\varpi_1$ and $\varpi_2$. 
The conformal weights of those potential subsingular vectors are respectively:  
\begin{equation}
\dfrac{(\varpi_1,\varpi_1+2\rho)}{2(k+h_{G_2}^\vee)}=1, 
\quad 
\dfrac{(\varpi_2,\varpi_2+2\rho)}{2(k+h_{G_2}^\vee)}=2,
\end{equation}
with $h_{G_2}^\vee=4$ and $k=-2$.  
On the other hand,
by Theorem \ref{Th:singular_G2},
it is clear that there are no subsingular vectors in 
the subspace $V^{-2}(G_{2})_d$
with $d\leqslant 5$.
%
\end{proof}

Combining Theorem \ref{theorem1.5} and Theorem~\ref{theorem1.6}, 
we obtain 
the classification of the irreducible modules of 
$L_{-2}(G_2)$ in the category $\mathcal O$. 
This completes the proof of Theorem \ref{Th:mainB}. 
%

{In addition, an algorithm is provided in \cite{KR22} to construct the relaxed modules with finite-dimensional weight spaces of an affine vertex algebra based on the classification of its simple highest-weight modules.
More precisely, the classification of all simple weight modules (with finite-dimensional weight spaces) over the Zhu algebra is returned. 
Using Zhu's correspondence we then obtain the simple $\Z_{\geqslant 0}$-graded relaxed modules. 
The other relaxed modules are constructed applying spectral flow twists.
In the following, we outline the construction of the relaxed modules following the algorithm of \cite{KR22}.

The construction of the $A(L_{-2}(G_{2}))$-modules uses the classification of parabolic subalgebras of $G_2$ whose Levi is of type $A$ or $C$. Since $G_2$ is not of type $AC$, there are three distinct ones up to twisting by a element of the Weyl group $\weyl_{G_2}$:
\begin{itemize}
	\item the Borel subalgebra $\mathfrak{b}=\Span_\C\{h_1,h_2,e_{\alpha_i} (i=1,\dots,6)\}$ with Levi subalgebra $\h$,
	\item the subalgebra $\mathfrak{p}_1=\mathfrak{b}\oplus\C e_{-\alpha_1}$ with Levi $\mathfrak{l}_1=\Span_\C\{e_{\pm \alpha_1},h_1,h_2\}\simeq\sll_2\oplus\gl_1$,
	\item  the subalgebra $\mathfrak{p}_2=\mathfrak{b}\oplus\C e_{-\alpha_2}$ with Levi $\mathfrak{l}_2=\Span_\C\{e_{\pm \alpha_2},h_1,h_2\}\simeq\sll_2\oplus\gl_1$.
\end{itemize}

The choice of parabolic subalgebra $\mathfrak{b}$ returns the irreducible highest-weights modules $L_{G_{2}}(\mu_i)$ listed in Proposition~\ref{Pro:mu_i-G2} as well as their twists by the elements in the Weyl group $\weyl_{G_2}$ that are naturally extended into automorphisms of the Lie algebra $G_2$. The latter are highest-weight $\weyl_{G_2}$-modules with respect to a different choice of Borel. The simple highest-weight $A(L_{-2}(G_{2}))$-modules consist in:
\begin{itemize}
	\item the three finite-dimensional highest-weight modules $L_{G_{2}}(\mu_i)$, $i=1,2,3$,
	\item the $11\times 6$ infinite-dimensional highest-weight modules $w(L_{G_{2}}(\mu_i))$, $w\in\weyl_{G_2}/\la s_1\ra$, $i\in\{6,7,8,11,12\}\cup\{15,\dots,20\}$,
	\item the $4\times6$ infinite-dimensional highest-weight modules $w(L_{G_{2}}(\mu_i))$, $w\in\weyl_{G_2}/\la s_2\ra$, $i\in\{4,5,9,10\}$,
	\item the $2\times12$ infinite-dimensional highest-weight modules $w(L_{G_{2}}(\mu_i))$, $w\in\weyl_{G_2}$, $i=13,14$.
\end{itemize}

The choice of parabolic subalgebra $\mathfrak{p}_j$ ($j=1,2$) leads to the construction of certain irreducible semisimple \emph{coherent families} of $\mathfrak{l}_j$-modules
$$\mathcal{C}\simeq\bigoplus_{[\lambda]}\mathcal{C}^{[\lambda]}$$
where the sum runs over the cosets of the weight support of $\mathcal{C}$ modulo $\Z\alpha_j$.
A coherent family $\mathcal{C}$ of $\g$-modules is a weight $\g$-module whose the dimension of the weight spaces $\mathcal{C}_{\nu}$ is independent of $\nu\in\h^*$ and the trace of the action on $\mathcal{C}_{\nu}$ of every $u$ in the centralizer $U(\g)^{\h}$ is a polynomial in $\nu$. It is said irreducible if some $\mathcal{C}^{[\lambda]}$ is irreducible and semisimple if every $\mathcal{C}^{[\lambda]}$ is semisimple.

For $\g=\mathfrak{l}_j\simeq\sll_2\oplus\gl_1$, the coherent families are obtained via the localization of infinite-dimensional highest-weight $\sll_2$-modules. We refer to \cite{Mat} for the detailed construction. Hence irreducible semisimple coherent families classified by the simple infinite-dimensional highest-weight $\sll_2$-modules up to the action of the Weyl group $\Z_2=\la s_j\ra$ of $\mathfrak{l}_j$. Denote $\mathcal{C}(\mu)$ the irreducible semisimple coherent family corresponding to the highest-weight $\mu$.
The direct summands $\mathcal{C}{[\lambda]}(\mu)$ is irreducible unless $[\lambda]=[\mu]$ or $[\lambda]=[s\cdot\mu]$ where $s_j\cdot\mu=s_j(\mu+\tfrac{1}{2}\alpha_j)-\tfrac{1}{2}\alpha_j$.
The irreducible quotient $\mathcal{S}^{[\lambda]}_{G_2}(\mu)$ of the $G_2$-module induced from $\mathcal{C}{[\lambda]}(\mu)$ is called a \emph{semidense} $G_2$-module as its weight support fulfills a half-plane of $\h^*$, that is
$$\mu+\Z\alpha_1-\Z_{\geqslant0}\alpha_2\quad\text{or}\quad\mu-\Z_{\geqslant 0}\alpha_1+\Z\alpha_2$$
depending whether $\mathcal{C}^{[\lambda]}(\mu)$ comes from $\mathfrak{p}_1$ or $\mathfrak{p}_2$ respectively.

To determine which coherent families induced simple semidense $A(L_{-2}(G_{2}))$-modules with finite-dimensional weight spaces, one check for each $i=1,\dots,20$ and $j=1,2$ whether the highest-weight module $L_{G_{2}}(\mu_i)$ will result in an infinite-dimensional highest-weight $\sll_2$-modules. This can be done by projecting the weight $\mu_i$ onto the weight-space of the simple ideal $\sll_2$ of the Levi $\mathfrak{l}_j$. We denote $\pi_j(\mu_i)$ these projections.

There are four one-parameter families of semidense modules resulting from the choice of the parabolic $\mathfrak{p}_1$:
\begin{equation*}
	\mathcal{S}^{[\lambda]}_{G_2}({\pi_1(\mu_4)}),\,
	\mathcal{S}^{[\lambda]}_{G_2}({\pi_1(\mu_5)}),\,
	\mathcal{S}^{[\lambda]}_{G_2}({\pi_1(\mu_9)})=\mathcal{S}^{[\lambda']}_{G_2}({\pi_1(\mu_{13})}),\,
	\mathcal{S}^{[\lambda]}_{G_2}({\pi_1(\mu_{10})})=\mathcal{S}^{[\lambda']}_{G_2}({\pi_1(\mu_{14})}),
\end{equation*}
together with their twists by the elements in $\weyl_{G_2}/\la s_1\ra$,
and six resulting from the choice of the parabolic $\mathfrak{p}_2$:
\begin{equation*}
	\begin{gathered}
		\mathcal{S}^{[\lambda]}_{G_2}({\pi_2(\mu_6)}),\,
		\mathcal{S}^{[\lambda]}_{G_2}({\pi_2(\mu_7)}),\,
		\mathcal{S}^{[\lambda]}_{G_2}({\pi_2(\mu_{11})})\simeq\mathcal{S}^{[\lambda']}_{G_2}({\pi_2(\mu_{18})}),\,
		\mathcal{S}^{[\lambda]}_{G_2}({\pi_2(\mu_{12})})\simeq\mathcal{S}^{[\lambda']}_{G_2}({\pi_2(\mu_{16})}),\,\\
		\mathcal{S}^{[\lambda]}_{G_2}({\pi_2(\mu_{13})})\simeq\mathcal{S}^{[\lambda']}_{G_2}({\pi_2(\mu_{20})}),\,
		\mathcal{S}^{[\lambda]}_{G_2}({\pi_2(\mu_{14})})\simeq\mathcal{S}^{[\lambda']}_{G_2}({\pi_2(\mu_{15})})
	\end{gathered}
\end{equation*}
with their twists by the elements in $\weyl_{G_2}/\la s_2\ra$.
This complete the classification of simple weight $A(L_{-2}(G_{2}))$-modules with finite-dimensional weight spaces.
The classification of the simple relaxed $L_{-2}(G_2)$-modules follows by applying Zhu's functor and spectral flow twist.

\begin{cor}\label{cor:relaxedG2}
	The irreducible 
	relaxed $L_{-2}(G_2)$-modules with finite-dimensional weight spaces are obtained as spectral flow twists and $W_{G_2}$-twists of the following modules:
	\begin{itemize}
		\item the irreducible highest-weight modules $L_{G_{2}}(-2,\mu_i)$, $i=1,\dots,20$,
		\item the irreducible semirelaxed $\mathcal{S}_{G_2}^{[\lambda]}(-2,\pi_1(\mu_i))$, $i=4,5,9,10$, $\lambda\in\mu_i+\C\alpha_1$, $[\lambda]\neq[\mu],[s_1\cdot\mu_i]$,
		\item the irreducible semirelaxed $\mathcal{S}_{G_2}^{[\lambda]}(-2,\pi_2(\mu_i))$, $i=6,7,11,\dots,14$, $\lambda\in\mu_i+\C\alpha_2$, $[\lambda]\neq[\mu],[s_2\cdot\mu_i]$.
	\end{itemize}
\end{cor}

Note that $L_{-2}(G_2)$ does not admit fully relaxed modules with finite-dimensional weight spaces. It admits (fully) relaxed modules with infinite-dimensional weight spaces though but they cannot be classify using the algorithm in \cite{KR22}.}

\section{Associated variety of $L_{-2}(G_{2})$}
\label{Sec:G2_assoc}

We continue to write $\g$ for the simple Lie algebra $G_{2}$, and 
we keep the related notations of the previous sections, in particular 
about the numbering of positive roots in $G_2$. 
First, using the computations of the previous section, we establish 
the following result. 

\begin{prop}
\label{Pro:quasi-lisse}
The associated variety of $L_{-2}(G_2)$ 
is contained in the nilpotent cone of $G_2$. 
In other words, the simple vertex algebra $L_{-2}(G_2)$ is quasi-lisse. 
\end{prop}

\begin{proof}
We follow the strategy adopted in \cite{AraMor2}. 
Let $M$ be the $G_{2}$-module generated by 
$v''_{\sing}$ under the adjoint action, 
and $I_{M}$ the ideal of $S(G_{2})$ generated by $M$. 
Then $R_{L_{-2}(G_2)}\cong S(G_{2})/I_{M}$. 
Set 
$$I_{M}^{\h} :=\Psi(I_{M}\cap S(G_{2})^{\h}),$$
where $\Psi$ is the Chevalley projection map \eqref{eq:Chevalley}.  
We extract seven linearly independent polynomials of 
$\h$ in $I_{M}^{\h}$: 
{\footnotesize 
\begin{align*}
&[e_{-\theta},[e_{-\alpha_5},[e_{-\alpha_4},v''_{\sing}]]] 
= -3 h_1 h_2 \left(h_1+h_2\right) \left(h_1+2 h_2\right) \left(h_1+3
h_2\right) \left(2 h_1+3 h_2\right), & \\ 
&[e_{-\alpha_4},[e_{-\alpha_4},[e_{-\alpha_4},[e_{-\alpha_4},v''_{\sing}]]]] 
=  -24 \left(h_1+h_2\right) \left(h_1+2 h_2\right) \left(2 h_1+3
h_2\right)^4, &\\ 
&[e_{-\alpha_5},[e_{-\alpha_4},[e_{-\alpha_4},[e_{-\alpha_3},v''_{\sing}]]]] 
= 4 \left(h_1+h_2\right) \left(h_1+2 h_2\right) \left(h_1+3 h_2\right)
\left(2 h_1+3 h_2\right)^2 \left(4 h_1+3 h_2\right),&\\
&[e_{-\alpha_5},[e_{-\alpha_5},[e_{-\alpha_3},[e_{-\alpha_3},v''_{\sing}]]]] 
= -4 \left(h_1+h_2\right){}^2 \left(h_1+3 h_2\right)^2 \left(16 h_1^2+33
h_2 h_1+18 h_2^2\right), &\\
&[e_{-\theta},[e_{-\alpha_4},[e_{-\alpha_4},[e_{-\alpha},v''_{\sing}]]]]
 = 4 h_1 \left(h_1+h_2\right) \left(h_1+2 h_2\right) \left(2 h_1+3
				h_2\right)^2 \left(4 h_1+9 h_2\right), & \\ 
&[e_{-\theta},[e_{-\alpha_5},[e_{-\alpha_3},[e_{-\alpha_1},v''_{\sing}]]]] 
= -2 h_1 \left(h_1+h_2\right) \left(h_1+2 h_2\right) \left(h_1+3
				h_2\right) \left(16 h_1^2+51 h_2 h_1+45 h_2^2\right), &\\ 
&[e_{-\theta},[e_{-\theta},[e_{-\alpha_1},[e_{-\alpha_1},v''_{\sing}]]]]
= -4 h_1^2 \left(h_1+2 h_2\right)^2 \left(16 h_1^2+63 h_2 h_1+63
				h_2^2\right), & 
\end{align*}
\noindent }
\hspace{-0.11cm}where the equalities are modulo $\n_{-}S(G_{2})+S(G_{2})\n_{+}$.  
The only semisimple element on which 
these seven polynomials vanish 
is $0$. 
Because the associated 
variety is invariant under the adjoint group, 
this implies that the associated variety of $L_{-2}(G_2)$ 
has no (non-zero) semisimple element, 
and so is contained in the nilpotent cone of $G_2$.  
The proposition follows. 
\end{proof}

Set 
\begin{align*}
f =f_\text{sreg}= e_{-\alpha_2}+e_{-\alpha_4},
\quad x =\frac{h}{2}=h_{1}+2h_{2}
\end{align*}
so that $(e,h,f)$ is a subregular $\mf{sl}_2$-triple in $G_2$. 
It defines a grading with respect to $f_{\text{sreg}}$.
The nilpotent element $f_{\text{sreg}}$ is even and we have
\begin{align*}
\g=\g_{-2}\oplus\g_{-1}\oplus\g_{0}\oplus\g_{1}\oplus\g_{2},
\end{align*}
where
\begin{align*}
&\g_{2}=\mathbb{C}e_{\theta} , &&
\g_{-2}=\mathbb{C}e_{-\theta} ,\\
&\g_{1}=\mathbb{C}e_{\alpha_2}\oplus \mathbb{C} e_{\alpha_4} 
\oplus \mathbb{C}e_{\alpha_3} \oplus \mathbb{C} e_{\alpha_5}, &&
\g_{-1}=\mathbb{C}e_{-\alpha_2}\oplus \mathbb{C} e_{-\alpha_4} 
\oplus \mathbb{C}e_{-\alpha_3} \oplus \mathbb{C} e_{-\alpha_5},\\
&\g_{0}=\mathbb{C}h_{1}\oplus \mathbb{C}h_{2} \oplus \mathbb{C}e_{\alpha_1} \oplus 
		\mathbb{C}e_{-\alpha_1} . &&
\end{align*}
The centralizer of $f$ in $G_{2}$ is a four-dimensional vector space 
$\g^{f}=\g^{f}_{-2}\oplus \g^{f}_{-1}$, with: 
\begin{align*}
\g^{f}_{-1}
=\mathbb{C}(e_{-\alpha_2}+e_{-\alpha_4})
\oplus \mathbb{C}e_{-\alpha_2}\oplus \mathbb{C}(e_{-\alpha_3}
-3e_{-\alpha_5}), \quad 
\g^{f}_{-2}=\mathbb{C} e_{-\theta}. 
\end{align*}
Consider the $\W$-algebra 
$\W^{k}(G_{2},f_{\text{sreg}})$  
associated with $G_2$ and $f_{\text{sreg}}$ (see Section \ref{Sub:W-algebras}). 
Since $\dim \g^{f}=4$, 
we know that the $\W$-algebra 
$\W^{k}(G_{2},f_{\text{sreg}})$ 
is strongly generated by the fields 
$J^{\{f_{\text{sreg}}\}},J^{\{e_{-\alpha_2}\}}$,  
$J^{\{e_{-\alpha_3}-3e_{-\alpha_5}\}}$ 
and $J^{\{e_{-\theta}\}}$ defined bellow. 
The OPEs between 
the generators have been 
computed in \cite{Fasquel-OPE}: 
{\footnotesize 
\begin{align*}
J^{\{f_{\text{sreg}}\}}
&=J^{f_{\text{sreg}}}-\frac{1}{3}:J^{h_{1}}J^{h_{1}}:-:J^{h_{1}}J^{h_{2}}:-:J^{h_{2}}J^{h_{2}}:-\frac{1}{3}:J^{e_{-\alpha_1}}J^{e_{-\alpha_1}}:-\left(\frac{7}{3}+k\right)\partial 
J^{h_{1}}-\left(5+2k\right)\partial J^{h_{2}}\\
J^{\{e_{-\alpha_2}\}}&=J^{e_{-\alpha_2}}-\frac{1}{4}:J^{h_{2}}J^{h_{2}}:-
\frac{1}{12}:J^{e_{\alpha_1}}J^{e_{\alpha_1}}:-\frac{1}{4}\left(5+2k\right)\partial J^{h_{2}}\\
J^{\{e_{-\alpha_3}-3e_{-\alpha_5}\}}
& =J^{e_{-\alpha_3}-3e_{-\alpha_5}}
-\frac{2}{3}:J^{h_{1}}J^{e_{\alpha_{1}}}:
+:J^{h_{1}}J^{e_{-\alpha_{1}}}:-:J^{h_{2}}J^{e_{\alpha_1}}:\\
&\qquad +:J^{h_{2}}J^{e_{-\alpha_1}}:-\left(\frac{7}{3}+k\right)
\partial J^{e_{\alpha_1}}+\left(3+k\right)\partial J^{e_{-\alpha_{1}}}\\
6J^{\{e_{-\theta}\}} &=6J^{e_{-\theta}}+\partial J^{e_{-\alpha_3}}
+3\partial J^{e_{-\alpha_5}}+:J^{e_{-\alpha_2}}J^{e_{\alpha_1}}:
-2:J^{e_{-\alpha_2}}J^{e_{-\alpha_2}}:+2:J^{e_{-\alpha_2}}J^{h_{1}}:\\
&\qquad+3:J^{e_{-\alpha_3}}J^{h_{2}}:-:J^{e_{-\alpha_4}}J^{e_{\alpha_1}}:
+3:J^{e_{-\alpha_5}}J^{h_{2}}:-\frac{1}{3}\partial^{2}J^{e_{\alpha_1}}-\frac{1}{3}:J^{h_{1}}\partial J^{e_{\alpha_1}}:\\
&\qquad
-:J^{h_{2}}\partial J^{e_{\alpha_1}}:+\frac{1}{3}:\partial J^{h_{1}}J^{e_{\alpha_1}}:
-:\partial J^{e_{-\alpha_1}}J^{h_{1}}:-:\partial J^{e_{-\alpha_1}}J^{h_{2}}:
-\frac{1}{3}:J^{h_{1}}J^{h_{1}}J^{e_{\alpha_1}}:\\
&\qquad-:J^{h_{1}}J^{h_{2}}J^{e_{\alpha_1}}:-:J^{h_{2}}J^{h_{2}}J^{e_{-\alpha_1}}:-\frac{1}{9}:J^{e_{\alpha_1}}J^{e_{\alpha_1}}J^{e_{\alpha_1}}:
-:J^{e_{-\alpha_1}}J^{h_{1}}J^{h_{2}}:\\
&\qquad-:J^{e_{-\alpha_1}}J^{h_{2}}J^{h_{2}}:
+\frac{1}{3}:J^{e_{-\alpha_1}}J^{e_{\alpha_1}}J^{e_{\alpha_1}}:
\end{align*}
}
For $k\neq -4$, we can redefine the generators as follows:
\begin{align*}
&\text{conformal weight 2}: L=-\frac{J^{\{f_{\text{sreg}}\}}}{4+k}\\
&\text{conformal weight 2}: G^{+}=-J^{\{f_{\text{sreg}}\}}+4J^{\{e_{-\alpha_2}\}}\\
&\text{conformal weight 2}:
		G^{-}=-J^{\{e_{-\alpha_3}-3e_{-\alpha_5}\}}\\
&\text{conformal weight 3}:
		F=6J^{\{e_{-\theta}\}}.
\end{align*}
Let $(-|-)$ be an invariant inner product of $G_{2}$. 
Define $\chi\in \g_{>0}^{*}$ 
by $\chi(x)=-(f_{\text{sreg}}|x)$ for $x\in \g_{>0}$. 
 Set 
\begin{align*}
\mathfrak m:= \g _1 \oplus  \g _2, 
\quad J_{\chi} :=\sum_{x\in \mathfrak m} \mathbb C[\g ^*] (x-\chi (x)).
\end{align*}
We obtain 
$$(f_{\text{sreg}}|e_{\alpha_2})=1,  
\quad  (f_{\text{sreg}}|e_{\alpha_3})=0, 
\quad  (f_{\text{sreg}}|e_{\alpha_4})=3,  
\quad (f_{\text{sreg}}|e_{\alpha_5})=0, 
\quad (f_{\text{sreg}}|e_{\theta})=0,$$
and
\begin{align*}
& v''= 
		9(e_{-\alpha_2}+e_{-\alpha_4})-12e_{-\alpha_2} 
		+e_{\alpha_1}^{2}-3e_{\alpha_1}e_{-\alpha_1} 
		-3h_{1}^{2}-9h_{1}h_{2}-6h_{2}^{2} \quad \text{mod}\, J_{\chi}
\end{align*}
Moreover,
\begin{align*}
& 9J^{\{f_{\text{sreg}}\}}
		-12J^{\{e_{-\alpha_2}\}} =9J^{e_{-\alpha_2}+e_{-\alpha_4}}-12J^{e_{-\alpha_2}}
		+:J^{e_{\alpha_1}}J^{e_{\alpha_1}}:
		-3:J^{e_{\alpha_1}}J^{e_{-\alpha_1}}: & \\
& \qquad \quad -3:J^{h_{1}}J^{h_{1}}:-9:J^{h_{1}}J^{h_{2}}:-6:J^{h_{2}}J^{h_{2}}:-(21+9k)\partial J^{h_{1}}-6(5+2k)\partial J^{h_{2}}.&\end{align*}
Next, consider the term $e_{-\alpha_1}(0)v_\sing$ 
which preserves the conformal weight. 
It is easy to select the possible nonzero terms inside 
$v_{\sing}$ that contribute for the evaluation 
of $e_{-\alpha_1}(0)v_{\sing}$, namely
{\footnotesize 
\begin{align*}
v^{\text{nonzero}}
&= e_{\alpha_4}(-1)^{3}e_{\alpha_2}(-1)e_{\alpha_1}(-1)^{2}\mathbf{1}
-e_{\alpha_4}(-1)^{3}e_{\alpha_3}(-1)e_{\alpha_1}(-1)h_{2}(-1)\mathbf{1} -e_{\alpha_4}(-1)^{4}e_{\alpha_1}(-1)e_{-\alpha_1}(-1)\mathbf{1}
&\\
&\qquad +e_{\alpha_5}(-1)^{4}e_{\alpha_2}(-1)e_{-\alpha_2}(-1)\mathbf{1}
-e_{\alpha_5}(-1)^{4}e_{\alpha_3}(-1)e_{-\alpha_3}(-1)\mathbf{1}
-e_{\alpha_4}(-1)^{5}e_{-\alpha_4}(-1)\mathbf{1}&\\
&\qquad-e_{\alpha_4}(-1)^{4}h_{1}(-1)^{2}\mathbf{1}
-3e_{\alpha_4}(-1)^{4}h_{1}(-1)h_{2}(-1)\mathbf{1}-2 e_{\alpha_4}(-1)^{4}h_{2}(-1)^{2}\mathbf{1}&\\
&\qquad
+10e_{\alpha_5}(-1)e_{\alpha_4}(-1)^{2}e_{\alpha_2}(-1)e_{\alpha_1}(-1)h_{1}(-1)\mathbf{1}
+15e_{\alpha_5}(-1)e_{\alpha_4}(-1)^{2}e_{\alpha_2}(-1)e_{\alpha_1}(-1)h_{2}(-1)\mathbf{1}&\\
&\qquad+8e_{\alpha_5}(-1)e_{\alpha_4}(-1)^{3}e_{\alpha_2}(-1)e_{-\alpha_3}(-1)\mathbf{1}
-7e_{\alpha_5}(-1)e_{\alpha_4}(-1)^{4}e_{-\alpha_5}(-1)\mathbf{1}&\\
&\qquad-4e_{\alpha_5}(-1)e_{\alpha_4}(-1)^{3}h_{1}(-1)e_{-\alpha_1}(-1)\mathbf{1}
-2e_{\alpha_5}(-1)e_{\alpha_4}(-1)^{3}h_{2}(-1)e_{-\alpha_1}(-1)\mathbf{1}.&
\end{align*}}
By calculation, we have
{\footnotesize 
\begin{align*}
\chi(e_{-\alpha_1}(0)v^{\text{nonzero}})
			&=-3^{3}h_{1}e_{\alpha_1}-3^{3}e_{\alpha_1}h_{1}-3^{4}e_{\alpha_1}h_{2}
			+3^{4}h_{1}e_{-\alpha_1}-3^{4}e_{-\alpha_3}&\\
			&\quad -3^{5}e_{-\alpha_3}+3^{6}e_{-\alpha_5}
			-2\cdot3^{4}\left(e_{-\alpha_1}h_{1}-h_{1}e_{-\alpha_1}\right)
			-2\cdot3^{5}e_{-\alpha_1}h_{2}+3^{5}h_{1}e_{-\alpha_1}
			&\\
			&\quad+2\cdot3^{4}\left(e_{-\alpha_1}h_{2}
			+h_{2}e_{-\alpha_1}\right)+10\cdot3^{3} e_{\alpha_1}h_{1}
			+15\cdot3^{3}e_{\alpha_1}h_{2}+8\cdot3^{4}e_{-\alpha_3} &\\
			&\quad-7\cdot3^{5}e_{-\alpha_5}-4\cdot3^{4}h_{1}e_{-\alpha_1}-2\cdot3^{4}h_{2}e_{-\alpha_1}.
	\end{align*}}
Hence 
{\footnotesize 
\begin{align*}
(e_{-\alpha_1}(0)v_{\sing})'' & =  
3^{3} \cdot12\left( 
-(e_{-\alpha_3}-3e_{-\alpha_5}) 
+\frac{2}{3}h_{1}e_{\alpha_{1}}-h_{1}e_{-\alpha_1}
+h_{2}e_{\alpha_1}-h_{2}e_{-\alpha_1}\right) \text{mod}\, J_{\chi}.&
\end{align*}
}

The following result is known, see for instance 
\cite{Ara15,DeSole-Kac}. 
\begin{lem} 
\label{Lem:RV}
Denoting by $M$ the connected nilpotent subgroup 
with Lie algebra $\mf{m}$ of the 
adjoint group of $G_2$, we have: 
\[
R_{\W^k(\g,f)} \cong (S(\g ) /J_{\chi})^M.
\]
\end{lem}

\begin{thm} 
\label{th:WG2}
We have 
\begin{align*}
H^0_{DS,f_{\text{\em sreg}}}(L_{-2}(G_{2}))\cong \W_{-2}(G_{2},f_{\text{\em sreg}})
\end{align*}
\end{thm}

\begin{proof}
Set for simplicity $f:=f_{\text{sreg}}$. 
Writing $I_{G_2}=\< v_{\sing}\>$, 
we get the short exact sequence 
\begin{equation}
0 \longrightarrow I_{G_2} \longrightarrow V^{-2}(G_2) \longrightarrow L_{-2}(G_2) \longrightarrow 0.
\end{equation}
Applying the quantum Drinfeld--Sokolov reduction 
$H^0_{DS, f}(-)$ to the above sequence, 
we obtain the short exact sequence
\begin{align*}
0\longrightarrow H^0_{DS,f}(I_{G_2}) \longrightarrow 
\W^{-2}(G_2,f) \longrightarrow H^0_{DS,f}(L_{-2}(G_2)) \longrightarrow 0
\end{align*}
due to the exactness of the quantum Drinfeld--Sokolov reduction functor.

Suppose that $v_{\sing}$ maps to $\tilde v$ in $\W^{-2}(G_2,f)$.
One can easily verify that the conformal weight of $\tilde v$ equals $2$. 
By Lemma \ref{Lem:RV}, its image in $R_{\W^{-2}(G_2,f)}$ is the image of the vector 
$(e_{-\alpha_1}(0)v_{\sing})''$ in $(S(\g )/J_{\chi})^M$. 
It is clear that $v_{\sing}$ maps to $-12L-3G^{+}$ in $\W^{-2}(G_2,f)$, 
and similarly $e_{-\alpha_1}(0)v_{\sing}$ maps to $324 G^{-}$ in $\W^{-2}(G_2,f)$.
From the OPEs of the four strong generators $L,G^{\pm}, F$ 
(see \cite{Fasquel-OPE}), we see  
that $-12L-3G^+$ does not generate the maximal ideal in $\W^{-2}(G_2, f)$, 
but the element $G^{-}$ does.
Thus, $H^0_{DS,f}(I_{G_2})$ is the maximal ideal in $\W^{-2}(G_2,f)$, 
and hence \hbox{$H^0_{DS,f}(L_{-2}(G_2))\cong \C$}. 
\end{proof}

\begin{rem}
From the above proof, we recover that 
$\W_{-2}(G_2,f) \cong\C$ from \cite[Corollary~4.2]{Fasquel-OPE}. 
\end{rem}

We are now in a position to prove Theorem \ref{Th:mainA}. 

\begin{proof}[Proof of Theorem \ref{Th:mainA}] 
As in the previous proof, 
set for simplicity $f:=f_{\text{sreg}}$. 
Write $\g$ for the simple exceptional 
Lie algebra $G_2$, and $G$ for its adjoint group. 
Let $\mathbb{O}_f=G.f$ be the adjoint orbit of $f$. 
We have to show that $X_{L_{-2}(G_2)}=\overline{\mathbb{O}_f}$. 

On the one hand, by \cite{Ara15}, 
the associated variety of 
$H^0_{DS,f}(L_{-2}(G_2))$ 
is the intersection of $X_{L_{-2}(G_2)}$ 
with the Slodowy slice $\mathscr{S}_f:=f+\g^{e}$,  
whence
$$\{f\}=X_{H^0_{DS,f}(L_{-2}(G_2))} 
= X_{L_{-2}(G_2)} \cap \mathscr{S}_f$$ 
using $H^0_{DS,f}(L_{-2}(G_2))\cong \W_{-2}(G_2,f) \cong \C$ 
from Theorem~\ref{th:WG2}. 
As a consequence, the associated variety 
$X_{L_{-2}(G_2)}$ contains $f$ and so  $\overline{\mathbb{O}_f}$, 
because $X_{L_{-2}(G_2)}$ is closed and $G$-invariant. 

On the other hand, the associated variety $X_{L_{-2}(G_2)}$
is included in the nilpotent cone $\mc{N}$ of $G_2$ 
by Proposition \ref{Pro:quasi-lisse}. 
We conclude  that 
$$\overline{\mathbb{O}_f} \subset X_{L_{-2}(G_2)} \subset \mc{N},$$ 
and so $X_{L_{-2}(G_2)} $ 
is the closure of the regular or the subregular nilpotent orbit of $G_2$. 
The intersection between the nilpotent cone 
and the Slodowy slice $\mathscr{S}_f$ 
is two-dimensional whereas  $X_{L_{-2}(G_2)} \cap \mathscr{S}_f=\{f\}$, 
the only possibility is 
$X_{L_{-2}(G_2)} =\overline{\mathbb{O}_f}.$

Theorem \ref{th:WG2} also allows to 
prove that $X_{L_{-2}(G_2)} \subset \mc{N}$ 
without having recourse to Proposition~\ref{Pro:quasi-lisse}. 
The argument first appears in \cite[Proposition 6.3.1]{Fasquel-thesis}.
We reproduce it for completeness of the paper. 
Suppose that there exists a non nilpotent element 
$x \in X_{L_{-2}(G_2)}$.  
Denote by $x = x_n + x_s$ 
its Jordan
decomposition with $x_n$ nilpotent and 
$x_s$ a nonzero semisimple element. 
The 
$G$-invariant closed cone $C(x):=G.\C^* x$ generated by $x$ 
is included in the associated variety.  
But according to  \cite[Theorem 2.9]{CharbMor10}, 
$C(x)$ contains the induced nilpotent orbit 
${\rm Ind}_{\g^{x_s}}^\g (\mathbb{O}_{x_n})$  
from the
adjoint orbit of $x_n$ in $\g^{x_s}$. 
The only induced nilpotent orbits in $G_2$ 
are the regular and the subregular
orbits, 
so $C(x)$ strictly contains the subregular nilpotent orbit. 
The variety $C(x)$ is
$G$-invariant, reduced and irreducible. 
Thus by \cite[Corollary 1.3.8]{Ginz},  
\begin{align*}
0 =\dim(X_{L_{-2}(G_2)} \cap \mathscr{S}_f) \geqslant \dim(C(x) \cap \mathscr{S}_f) 
= \dim C(x) - \dim \mathbb{O}_f > 0, 
\end{align*}
whence a contradiction. 
\end{proof}

\section{The representation theory of $L_{-2}(B_3)$}
\label{Sec:B3_rep}
In this section, we study the representations of the simple affine 
vertex algebra $L_{-2}(B_3)$. 
Let us consider the simple exceptional Lie algebra of type $B_3$ 
with simple roots $\beta_1,\beta_2,\beta_3$ 
and Dynkin diagram 
$$\begin{Dynkin}
\Dbloc{\Dcirc\Deast\Dtext{t}{\beta_1}}
\Dbloc{\Dcirc\Dwest\Ddoubleeast\Dtext{t}{\beta_2}}
\Drightarrow\Dbloc{\Dcirc\Ddoublewest\Dtext{t}{\beta_3}}
\end{Dynkin}$$
and 
the simple Lie algebra $D_4$ 
with simple roots $\gamma_1,\gamma_2,\gamma_3,\gamma_4$ 
and Dynkin diagram 
$$
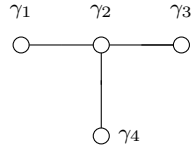
\begin{Dynkin}
	\Dbloc{\Dcirc\Deast\Dtext{t}{\gamma_1}}
	\Dbloc{\Dcirc\Dwest\Dsouth\Deast\Dtext{t}{\gamma_2}}
	\Dbloc{\Dcirc\Dwest\Dwest\Dtext{t}{\gamma_3}}\\
	\Dspace\Dbloc{\Dcirc\put(10,12){\line(0,1){12}}\Dtext{r}{\gamma_4}}
\end{Dynkin}
$$

We describe below the explicit embeddings 
$\iota_2 \colon G_{2} \hookrightarrow  B_3$  and $\iota_3 \colon B_{3} \hookrightarrow  D_4$ 
induced by the automorphisms of the Dynkin diagrams. 
First, one can express a Chevalley basis of $G_{2}$ 
in terms of that for $B_{3}$, 
which gives the embedding $\iota_2 \colon G_{2} \hookrightarrow  B_3$:  
\begin{align*} 
&e_{\alpha_1}  =e_{\beta_{1}}+e_{\beta_{3}},   
&&e_{-\alpha_1}=e_{-\beta_{1}}+e_{-\beta_{3}}, & \\
&e_{\alpha_2}=e_{\beta_{2}}, 
&& e_{-\alpha_2}=e_{-\beta_{2}} &\\
&e_{\alpha_3}=e_{\beta_{1}+\beta_{2}}-e_{\beta_{2}+\beta_{3}}, 
&&e_{-\alpha_3}=e_{-(\beta_{1}+\beta_{2})}-e_{-(\beta_{2}+\beta_{3})},&\\
&e_{\alpha_4}=-e_{\beta_{2}+2\beta_{3}}-e_{\beta_{1}+\beta_{2}+\beta_{3}}, 
&&e_{-\alpha_4}=-e_{-(\beta_{2}+2\beta_{3})}-e_{-(\beta_{1}+\beta_{2}+\beta_{3})},&\\
&e_{\alpha_5}=-e_{\beta_{1}+\beta_{2}+2\beta_{3}}, 
&& e_{-\alpha_5}=-e_{-(\beta_{1}+\beta_{2}+2\beta_{3})};&\\
&e_{\alpha_6}=-e_{\beta_{1}+2\beta_{2}+2\beta_{3}},
&&e_{-\alpha_6}=-e_{-(\beta_{1}+2\beta_{2}+2\beta_{3})},&\\
&h_1=h_{\alpha_1}=h_{\beta_{1}}+h_{\beta_{3}},
&&h_2=h_{\alpha_2}=h_{\beta_{2}}.&
\end{align*}

\noindent
Similarly, let us describe the Chevalley basis of $B_{3}$ in terms of that for $D_{4}$ in order to get 
the embedding $\iota_3\colon B_3 \hookrightarrow D_4$:  
\begin{align*}
&e_{\beta_{1}}=e_{\gamma_{1}}, 
&& e_{-\beta_{1}}=e_{-\gamma_{1}},&\\
&e_{\beta_{2}} =e_{\gamma_{2}}, 
&& e_{-\beta_{2}}=e_{-\gamma_{2}},&\\
& e_{\beta_{3}} =e_{\gamma_{3}}+e_{\gamma_{4}},
&& e_{-\beta_{3}}=e_{-\gamma_{3}}+e_{-\gamma_{4}}&\\
&e_{\beta_{1}+\beta_{2}}=e_{\gamma_{1}+\gamma_{2}}, 
&&e_{-(\beta_{1}+\beta_{2})}=e_{-(\gamma_{1}+\gamma_{2})},& \\
& e_{\beta_{2}+\beta_{3}}=e_{\gamma_{2}+\gamma_{3}}+e_{\gamma_{2}+\gamma_{4}},  
&& e_{-(\beta_{2}+\beta_{3})}=e_{-(\gamma_{2}+\gamma_{3})}+e_{-(\gamma_{2}+\gamma_{4})},&\\
& e_{\beta_{1}+\beta_{2}+\beta_{3}}
=e_{\gamma_{1}+\gamma_{2}+\gamma_{3}}+e_{\gamma_{1}+\gamma_{2}+\gamma_{4}},
&& e_{-(\beta_{1}+\beta_{2}+\beta_{3})}=e_{-(\gamma_{1}+\gamma_{2}+\gamma_{3})}+e_{-(\gamma_{1}+\gamma_{2}+\gamma_{4})},&\\
&e_{\beta_{2}+2\beta_{3}}=e_{\gamma_{2}+\gamma_{3}+\gamma_{4}},  
&&e_{-(\beta_{2}+2\beta_{3})}=e_{-(\gamma_{2}+\gamma_{3}+\gamma_{4})},\\ 
&e_{\beta_{1}+\beta_{2}+2\beta_{3}}=e_{\gamma_{1}+\gamma_{2}+\gamma_{3}+\gamma_{4}},
&&e_{-(\beta_{1}+\beta_{2}+2\beta_{3})}=e_{-(\gamma_{1}+\gamma_{2}+\gamma_{3}+\gamma_{4})},&\\
&e_{\beta_{1}+2\beta_{2}+2\beta_{3}}=e_{\gamma_{1}+2\gamma_{2}+\gamma_{3}+\gamma_{4}},
&& e_{-(\beta_{1}+2\beta_{2}+2\beta_{3})}=e_{-(\gamma_{1}+2\gamma_{2}+\gamma_{3}+\gamma_{4})}.&
\end{align*}
\begin{align*}
h_{\beta_{1}} = h_{\gamma_1}, \quad 
h_{\beta_{2}} = h_{\gamma_2}, \quad 
h_{\beta_{3}} = h_{\gamma_3}+h_{\gamma_4}. 
\end{align*}

We can compose these linear maps so that 
\begin{equation}
G_2  \stackrel{\iota_2}{\longhookrightarrow} B_3  
\stackrel{\iota_3}{\longhookrightarrow}  D_4 
\longhookrightarrow V^{-2}(D_4) 
\longtwoheadrightarrow 
L_{-2}(D_4).
 \end{equation} 
In \cite{AP}, Adamovi\'{c} and Per\v{s}e proved that 
the vertex subalgebra generated by $G_2$ (resp.~$B_3$) 
in $L_{-2}(D_4)$ is isomorphic to the irreducible affine vertex algebra 
$L_{-2}(G_2)$ (resp.~$L_{-2}(B_3)$). 
Consider the vertex algebra homomorphism 
$\hat{\iota}_2 \colon V^{-2}(G_2) \rightarrow V^{-2}(B_3)$ induced from $\iota_2$.
A direct consequence of \cite{AP} 
is that the vertex algebra homomorphism 
$\bar{\iota}_2 \colon L_{-2}(G_2) \rightarrow L_{-2}(B_3)$ 
is well defined and satisfies the following commutative diagram
\[
\begin{tikzcd}
    V^{-2}(G_2) \arrow{r}{\hat{\iota}_2} \arrow[swap]{d}{\pi_{G_2}} 
    & V^{-2}(B_3) \arrow{d}{\pi_{B_3}} \\
    L_{-2}(G_2) \arrow[swap]{r}{\bar{\iota}_2} & L_{-2}(B_3)
\end{tikzcd}
\]
where $\pi_{G_2}$ and $\pi_{B_3}$ are the natural projection maps.

Let $N^{B_3}_{-2}$ be the maximal ideal in $V^{-2}(B_3)$ and  
let 
$$v^{G_2}_\sing := v_\sing$$ 
be the singular vector of $V^{-2}(G_2)$ 
as in Theorem \ref{Th:singular_G2}. 
It is clear from the commutative diagram that 
$\hat{\iota}_2(\< v^{G_2}_{\sing} \> )\subset  N^{B_3}_{-2}$, because 
$ \bar{\iota}_2\circ \pi_{G_2} (\< v^{G_2}_{\sing} \> )=0$.
Hence the vector 
$$w : =\hat{\iota}_2 (v^{G_2}_{\sing})$$ 
is contained in $N^{B_3}_{-2}$ with conformal weight $6$. 
Fix $\h_{B_3}={\rm span}_\C\{h_{\beta_1},h_{\beta_2},h_{\beta_3}\}$ 
a Cartan subalgebra of $B_3$. 
Since the embedding $\hat{\iota}_2$ does not preserve the 
$\mathfrak h_{B_3}$-weight, 
we decompose $w$ into a sum of $\mathfrak h_{B_3}$-weight vectors
\begin{align}
\label{eq:w_mu}
w=\sum_{\mu \in \mathfrak h^*_{B_3}}w_{\mu}
\end{align}
where $h.w_\mu=\mu(h)w_\mu$ for all $h\in\mathfrak h_{B_3}$. 
{In particular, identifying $\h_{B_3}$ with its dual 
	using $(-|-)$, 
	the fundamental weights of $B_3$ are 
	$\varpi_1=\beta_1+\beta_2+\beta_3$, 
	$\varpi_2=\beta_1+2\beta_2+2\beta_3$, 
	$\varpi_3=\frac{1}{2}\beta_1+\beta_2+\frac{3}{2} \beta_3$.  }

It is known by \cite{AraMor2} that there is a singular vector 
$v_\sing^{B_3}$ of conformal weight two in $V^{-2}(B_{3})$ given by: 
\begin{align*}
& v_\sing^{B_3}= &\\
& \quad e_{\beta_{1}+2\beta_{2}+2\beta_{3}}(-1)e_{\beta_{3}}(-1) {\bf 1} 
-e_{\beta_{1}+\beta_{2}+2\beta_{3}}(-1)  e_{\beta_{2}}(-1) {\bf 1}  
+e_{\beta_{1}+\beta_{2}+\beta_{3}}(-1)e_{\beta_{2}+2\beta_{3}}(-1){\bf 1} .
\end{align*}
We denote by $I^{B_3}$ the left-ideal generated by $v_\sing^{B_3}$ in $V^{-2}(B_3)$, 
and consider the quotient vertex algebra 
$$\mathcal{V}_{-2}(B_{3})=V^{-2}(B_{3})/I^{B_3}.$$ 
The $U(B_{3})$-submodule $R^{B_{3}}$ 
generated by the vector ${v'}_{\sing}^{B_3}$ 
under the adjoint
action is isomorphic to $L_{B_3}(2\varpi_{3})$. 
By using the same method as for $G_2$, 
we can determine a basis of the space of polynomials 
$\mathscr{P}_{v_\sing^{B_3}}$, defined as in~\eqref{eq:P_0} 
with respect to $R^{B_{3}}$.  
\begin{lem}  
\label{lem2.1}
We have $\mathscr{P}_{v_\sing^{B_3}}
= {\rm span}_{\mathbb{C}}\{p_{1}^{B_{3}},p_{2}^{B_{3}},p^{B_{3}}_{3}\}$, 
where
\begin{align*}
&{p^{B_{3}}_{1}(h)}=(2h_{2}+h_{3})(h_{1}+h_{2}+h_{3})+2(h_{2}+h_{3}),& \\
&p^{B_{3}}_{2}(h)=(h_{2}+h_{3})(2h_{1}+2h_{2}+h_{3}+2),&\\
&p^{B_{3}}_{3}(h)=h_{3}(h_{1}+2h_2+h_{3}+2). 
\end{align*}
\end{lem}
One gets the complete list of irreducible 
$A(\mathcal{V}_{-2}(B_{3}))$-modules in the category 
$\mathcal{O}$ by solving the polynomial equations
\begin{align*}
p^{B_{3}}_{1}(h)=p^{B_{3}}_{2}(h)=p^{B_{3}}_{3}(h)=0. 
\end{align*}

Using Zhu's correspondence, we obtain the following results.
	
\begin{thm}
The complete list of irreducible 
$\mathcal{V}_{-2}(B_{3})$-modules in the category $\mathcal{O}$ 
is given by the following set: 
\begin{align*} 
\{L_{B_{3}}(-2,\mu_{i}(t)) \colon i=1,2,3,\; t\in \mathbb{C}\}
\end{align*}
where,
\begin{align*}
\mu_{1}(t)=t\varpi_{1}, \quad 
\mu_{2}(t)=(-1-t)\varpi_{1}+t\varpi_{2}, \quad 
\mu_{3}(t)=t\varpi_{2}-2(1+t)\varpi_{3}. 
\end{align*}
\end{thm}
	
\begin{cor}
The complete list of irreducible ordinary modules for 
$\mathcal{V}_{-2}(B_{3})$ is given by the following set:
\begin{align*}
\{L_{B_{3}}(-2,k\varpi_{1}) \colon k\in\mathbb{Z}_{\geqslant 0}\}. 
\end{align*}
\end{cor}

The vertex algebra $\mathcal{V}_{-2}(B_{3})$ is not simple 
and the following lemma gives a description 
of the structure of the 
quotient vertex algebra, see \cite[Corollary 7.6]{AKMPP}\footnote{There is a typo in 
\cite[Corollary 7.6]{AKMPP}: 
$L_{B_3}(-2,-4\Lambda_{0}+2\Lambda_{1})$ should be 
$L_{B_3}(-2,-6\Lambda_{0}+4\Lambda_{1})$.}.
\begin{lem} 
\label{lem2.2} 
The vertex algebra $\mathcal{V}_{-2}(B_{3})$ contains a unique ideal 
$I\cong L_{B_3}(-2,-6\Lambda_{0}+4\Lambda_{1})$, 
where $\Lambda_0,\Lambda_1,\Lambda_2,\Lambda_3$ 
are the fundamental weights of $\hat{B}_3$. 
\end{lem}

The key observation is that the unique singular vector in 
$\mathcal{V}_{-2}(B_{3})$ comes from a subsingular vector in $V^{-2}(B_3)$.

\begin{lem} \label{lem2.3}
In the decomposition \eqref{eq:w_mu}, 
assume that $w_{4\varpi_1} \neq 0$ 
and $w_{4\varpi_1} \notin I^{B_3}$, 
then the maximal ideal $N^{B_3}_{-2}$ 
is generated by $w_{4 \varpi_1}$ and $v_\sing^{B_3}$.
\end{lem}

\begin{proof}
Since $N^{B_3}_{-2} $ is homogeneous with respect to $\mathfrak h_{B_3}^*$, 
we deduce that $w_{4\varpi_1} \in N^{B_3}_{-2}$. 
 Let $v_\text{sub}$ be a homogeneous subsingular vector in $V^{-2}(B_3)$ 
which maps {through the quotient map $V^{-2}(B_3) \twoheadrightarrow 
\mathcal V_{-2}(B_3)$} 
 to the unique singular vector in $\mathcal V_{-2}(B_3)$. 
So $v_\text{sub}$ has $\h_{B_3}$-weight $4\varpi_1$ 
and conformal weight $6$. 
By Lemma~\ref{lem2.2}, we have $N^{B_3}_{-2}= \< v_\text{sub},  v_\sing^{B_3}\>.$  
For any ideal $J$, we denote by $J_6$ the conformal weight 6 subspace of $J$. 
It is clear that for any 
$u=\sum_{\mu \in \h_{B_3}^{\ast}} u_{\mu}$ in $\< v_\text{sub} \>_6  \backslash I^{B_3} $, 
 $u_{\mu} \neq 0$  implies that 
$\mu \leqslant 4\varpi_1$, where the equality holds if and only if
$u_{\mu}=u_{4\varpi_1}=c v_\text{sub} \mod I^{B_3}$ for some constant $c\neq 0$. 
Since $w_{4\varpi_1} \in N^{B_3}_{-2}\backslash I^{B_3}$, $w_{4\varpi_1}=c v_\text{sub} \mod I^{B_3}$ for $c\neq 0$. 
Therefore $v_\text{sub} \in \<w_{4\varpi_1} ,I^{B_3}\>$.
\end{proof}

Under the adjoint
action of $B_{3}$, the submodule of $U(B_{3})$ 
generated by vector $w_{4\varpi_{1}}$ is isomorphic to $L_{B_{3}}(4\varpi_{1})$, 
the zero-weight space of $L_{B_{3}}(4\varpi_{1})$ has dimension six. 
Let 
$$\tilde{L}_{-2}(B_3)  
:= V^{-2}(B_3)/ \<w_{4\varpi_1} ,I^{B_3}  \>.$$
Hence the irreducible highest-weight modules of $A(\tilde L_{-2}(B_3))$ are determined by polynomials in Lemma \ref{lem2.1} and \ref{lem2.4} below.

\begin{lem} 
\label{lem2.4} 
Let $\mathscr{P}_{w_{4\varpi_1} }$ be the polynomial 
set \eqref{eq:P_0} relatively to $w_{4\varpi_1}$ 
defined by the decomposition \eqref{eq:w_mu}. 
Then  $\mathscr{P}_{w_{4\varpi_1} }={\rm span}_{\C}\{p^{B_{3}}_{4}, p^{B_{3}}_{5},\cdots,p^{B_{3}}_{9}\}$. 
\end{lem}

\begin{proof}
By direct calculation we show that 
the following six polynomials are linearly independent, modulo 
$\mathfrak{n}_{-}U(B_{3})+U(B_{3})\mathfrak{n}_{+}$: 

{\footnotesize 
\begin{align*}
&p^{B_{3}}_{4}=\left(e_{-\beta_{1}-\beta_{2}-\beta_{3}}e_{-\beta_{1}-\beta_{2}
-\beta_{3}}e_{-\beta_{1}-\beta_{2}-\beta_{3}}e_{-\beta_{1}-\beta_{2}-\beta_{3}}\right)_{L}\left(w_{4\varpi_1}\right) ,&\\
&p^{B_{3}}_{5}=\left(e_{-\beta_{1}-\beta_{2}}e_{-\beta_{1}-\beta_{2}
-\beta_{3}}e_{-\beta_{1}-\beta_{2}-\beta_{3}}e_{-\beta_{1}-\beta_{2}-2\beta_{3}}\right)_{L}\left(w_{4\varpi_1}\right),&\\
&p^{B_{3}}_{6}=\left(e_{-\beta_{1}-\beta_{2}}e_{-\beta_{1}-\beta_{2}}e_{-\beta_{1}-\beta_{2}
-2\beta_{3}}e_{-\beta_{1}-\beta_{2}-2\beta_{3}}\right)_{L}\left(w_{4\varpi_1}\right),&\\
&p^{B_{3}}_{7}=\left(e_{-\beta_{1}}e_{-\beta_{1}-\beta_{2}-\beta_{3}}e_{-\beta_{1}-\beta_{2}
-\beta_{3}}e_{-\beta_{1}-2\beta_{2}-2\beta_{3}}\right)_{L}\left(w_{4\varpi_1}\right),&\\
&p^{B_{3}}_{8}=\left(e_{-\beta_{1}}e_{-\beta_{1}-\beta_{2}}e_{-\beta_{1}-\beta_{2}-2\beta_{3}}e_{-\beta_{1}
-2\beta_{2}-2\beta_{3}}\right)_{L}\left(w_{4\varpi_1}\right),&\\
&p^{B_{3}}_{9}=\left(e_{-\beta_{1}}e_{-\beta_{1}}e_{-\beta_{1}-2\beta_{2}-2\beta_{3}}e_{-\beta_{1}
-2\beta_{2}-2\beta_{3}}\right)_{L}\left(w_{4\varpi_1}\right).&
\end{align*}
}

\noindent 
The explicit form of these polynomials can be found in Appendix \ref{App:pol_B3}.
\end{proof}

\begin{prop} 
\label{prop2.5} 
The complete list of irreducible $A(\tilde L_{-2}(B_3))$-modules 
in the category $\mathcal{O}$ is given by the set
$\{L_{B_3}(\mu_{i}) \colon i=1,2,\ldots,13\},$
where the $\mu_i$'s are given by Table \ref{Tab:mu_B3}. 

\medskip

\begin{table}[htb]
\begin{center}
\begin{tabular}{|l|l||l|l|}
\hline
$\mu_{1}$ & $0$ & $\mu_{8}$ & $-\frac{5}{2}\varpi_{2}+3\varpi_{3}$\\[0.2em]
$\mu_{2}$ & $\varpi_{1}$ & $\mu_{9}$ & $-\frac{3}{2}\varpi_{2}+\varpi_{3}$\\[0.2em]
$\mu_{3}$ & $-2\varpi_{1}$ & $\mu_{10}$ & $-\frac{1}{2}\varpi_{1}-\frac{1}{2}\varpi_{2}$\\[0.2em]
$\mu_{4}$ & $-3\varpi_{1}$ & $\mu_{11}$ & $-\frac{3}{2}\varpi_{1}$\\[0.2em]
$\mu_{5}$ & $-\varpi_{2}$ & $\mu_{12}$ & $-\frac{1}{2}\varpi_{1}$\\[0.2em]
$\mu_{6}$ & $-2\varpi_{3}$ & $\mu_{13}$ & $-\frac{3}{2}\varpi_{1}+\frac{1}{2}\varpi_{2}$\\[0.2em]
$\mu_{7}$ & $\varpi_{1}-2\varpi_{2}$ && \\
\hline
\end{tabular}\\[0.2em]
\caption{The weights $\mu_i$ for $B_3$}
\label{Tab:mu_B3}
\end{center}
\end{table}

\end{prop}
\begin{proof}
The assertion can be established through a straightforward 
computation involving the polynomials in Lemma~\ref{lem2.1} and~\ref{lem2.4}.
\end{proof}

\begin{thm}
We have $\tilde L_{-2}(B_3) \cong L_{-2}(B_3)$.
\end{thm}
\begin{proof}
According to Proposition \ref{prop2.5}, the 
set of the solutions of the equations in 
$\mathscr{P}_{v_\sing^{B_3}} \cup \mathscr{P}_{w_{4\varpi_1} }$ 
is distinct from the solution set of $\mathscr{P}_{v_\sing^{B_3}}$. 
Hence $w_{4\varpi_1}$ is nonzero and not contained in $I^{B_3}$. 
We complete the proof due to Lemma~\ref{lem2.3}.
\end{proof}

Using Zhu's correspondence, 
we have achieved the proof of Theorem \ref{Th:mainC}. 


{
Moreover, as for $L_{-2}(G_2)$, one can construct the relaxed modules of $L_{-2}(B_3)$ based on the classification of highest-weight modules (Theorem \ref{Th:mainC}) and the parabolic subalgebras of $B_3$ whose Levi is of $AC$ type. 
There are six of them:
\begin{itemize}
	\item the Borel subalgebra $\mathfrak{b}$ with Levi subalgebra $\h_{B_3}$,
	\item the subalgebra $\mathfrak{p}_1=\mathfrak{b}\oplus\C e_{-\beta_1}$ with Levi $\mathfrak{l}_1\simeq\sll_2\oplus\gl_1^{\oplus2}$,
	\item  the subalgebra $\mathfrak{p}_2=\mathfrak{b}\oplus\C e_{-\beta_2}$ with Levi $\mathfrak{l}_2\simeq\sll_2\oplus\gl_1^{\oplus2}$,
	\item  the subalgebra $\mathfrak{p}_3=\mathfrak{b}\oplus\C e_{-\beta_3}$ with Levi $\mathfrak{l}_3\simeq\sll_2\oplus\gl_1^{\oplus2}$,
	\item the subalgebra $\mathfrak{p}_{12}=\mathfrak{b}\oplus\Span_\C\{e_{-\beta_1},e_{-\beta_2}, e_{-\beta_1-\beta_2}\}$ with Levi $\mathfrak{l}_{12}\simeq\sll_3\oplus\gl_1$,
	\item  the subalgebra $\mathfrak{p}_{23}=\mathfrak{b}\oplus\oplus\Span_\C\{e_{-\beta_2},e_{-\beta_3}, e_{-\beta_2-\beta_3}, e_{-\beta_2-2\beta_3}\}$ with Levi $\mathfrak{l}_{23}\simeq\spp_4\oplus\gl_1$.
\end{itemize}

The choice of the Borel corresponds to the highest-weight $A(L_{-2}(B_3))$-modules appearing in Proposition \ref{prop2.5} and their twists under the action of the elements of the Weyl group $W_{B_3}\simeq \mathfrak{S}_3\ltimes\Z_2^3$.
The parabolic subalgebra $\mathfrak{p}_i$ ($i=1,2,3$) gives one-parameters families of semidense modules obtained by the localization of one of the simple negative root vector $e_{-\beta_i}$. We have four families corresponding to $\mathfrak{p}_1$:
\begin{equation}\label{eq:semirelaxedB3_1}
	\mathcal{S}^{[\lambda]}_{B_3}({\pi_1(\mu_3)}),\,
	\mathcal{S}^{[\lambda]}_{B_3}({\pi_1(\mu_4)}),\,
	\mathcal{S}^{[\lambda]}_{B_3}({\pi_1(\mu_{11})})\simeq\mathcal{S}^{[\lambda']}_{B_3}({\pi_1(\mu_{10})}),\,
	\mathcal{S}^{[\lambda]}_{B_3}({\pi_1(\mu_{12})})\simeq\mathcal{S}^{[\lambda']}_{B_3}({\pi_1(\mu_{13})}),
\end{equation}
four corresponding to $\mathfrak{p}_2$:
\begin{equation}
	\mathcal{S}^{[\lambda]}_{B_3}({\pi_2(\mu_5)}),\,
	\mathcal{S}^{[\lambda]}_{B_3}({\pi_2(\mu_7)}),\,
	\mathcal{S}^{[\lambda]}_{B_3}({\pi_2(\mu_8)})\simeq\mathcal{S}^{[\lambda']}_{B_3}({\pi_2(\mu_{13})}),\,
	\mathcal{S}^{[\lambda]}_{B_3}({\pi_2(\mu_{9})})\simeq\mathcal{S}^{[\lambda']}_{B_3}({\pi_2(\mu_{10})}),
\end{equation}
and one coming from the choice of $\mathfrak{p}_3$:
\begin{equation}
	\mathcal{S}^{[\lambda]}_{B_3}({\pi_3(\mu_6)}).
\end{equation}

Finally, the choice of parabolics $\mathfrak{p}_{12}$ and $\mathfrak{p}_{23}$ provides two-parameters families of semidense $A(L_{-2}(B_3))$-modules. 
There are four corresponding to $\mathfrak{p}_{12}$:
\begin{equation}
	\begin{gathered}
		\mathcal{S}^{[\lambda]}_{B_3}({\pi_{12}(\mu_3)})\simeq\mathcal{S}^{[\lambda']}_{B_3}({\pi_{12}(\mu_{5})}),\,
		\mathcal{S}^{[\lambda]}_{B_3}({\pi_{12}(\mu_4)})\simeq\mathcal{S}^{[\lambda']}_{B_3}({\pi_{12}(\mu_{7})}),\,\\
		\mathcal{S}^{[\lambda]}_{B_3}({\pi_{12}(\mu_8)})\simeq\mathcal{S}^{[\lambda']}_{B_3}({\pi_{12}(\mu_{12})})\simeq\mathcal{S}^{[\lambda'']}_{B_3}({\pi_{12}(\mu_{13})}),\,\\
		\mathcal{S}^{[\lambda]}_{B_3}({\pi_{12}(\mu_{9})})\simeq\mathcal{S}^{[\lambda']}_{B_3}({\pi_{12}(\mu_{10})})\simeq\mathcal{S}^{[\lambda'']}_{B_3}({\pi_{12}(\mu_{11})}),
	\end{gathered}
\end{equation}
and two for the choice of $\mathfrak{p}_{23}$:
\begin{equation}\label{eq:semirelaxedB3_last}
	\mathcal{S}^{[\lambda]}_{B_3}({\pi_{23}(\mu_8)})\simeq\mathcal{S}^{[\lambda']}_{B_3}({\pi_{23}(\mu_{13})}),\,
	\mathcal{S}^{[\lambda]}_{B_3}({\pi_{23}(\mu_{9})})\simeq\mathcal{S}^{[\lambda']}_{B_3}({\pi_{23}(\mu_{10})}).
\end{equation}

The semidense modules defined previously are generically irreducible.
Using Zhu's correspondence, we deduce the following classification of relaxed $L_{-2}(B_3)$-modules.
\begin{cor}\label{cor:relaxedB3}
	The irreducible 
	relaxed $L_{-2}(B_3)$-modules with finite-dimensional weight spaces are obtained as spectral flow twists and $W_{B_3}$-twists of the following modules:
	\begin{itemize}
		\item the irreducible highest-weight modules $L_{B_3}(-2,\mu_i)$, $i=1,\dots,13$,
		\item the irreducible semirelaxed $\mathcal{S}_{G_2}^{[\lambda]}(-2,\pi(\mu_i))$ corresponding to the image of the irreducible semidense $A(L_{-2}(B_3))$-modules in \eqref{eq:semirelaxedB3_1}--\eqref{eq:semirelaxedB3_last}.
	\end{itemize}
\end{cor}
}

\subsection{Decomposition of non-ordinary modules}
In this paragraph,  
we obtain a nontrivial decomposition theorem 
of non-ordinary modules 
coming from the spectral flows 
of the ordinary modules.
These modules are still graded by $L_0$, 
however, they are not necessarily bounded below.

According to \cite{AP}, we have 
\begin{equation} 
\label{12}
L_{D_4}(-2,0) =L_{B_3}(-2,0) \oplus L_{B_3}(-2, \varpi_1 ).
\end{equation}
Consider the standard 
representation 
$L:=L_{B_3}(\varpi_1) \cong\C^7$ 
of $B_3$ with canonical basis $\epsilon_i$, $i \in\{1,\ldots,7\}$. 
We schematize the representation by the following graph. 
\begin{align*}
L_{\beta_1+\beta_2+\beta_3} \ni \epsilon_2 \stackrel{f_1}{\longrightarrow} 
\epsilon_3 \stackrel{f_2}{\longrightarrow} \epsilon_4 
\stackrel{f_3}{\longrightarrow} \epsilon_1 \stackrel{f_3}{\longrightarrow} 
\epsilon_7 \stackrel{f_2}{\longrightarrow} \epsilon_6 
\stackrel{f_1}{\longrightarrow} \epsilon_5 \in L_{-\beta_1-\beta_2-\beta_3},
\end{align*}
where $f_j:=e_{-\beta_j}$ is the negative $\beta_j$-root vector. 

Assume for a while that $\g$ is an arbitrary simple Lie algebra  
as in Section~\ref{Sec:Pre}. 
For an arbitrary $\extg$-module $M$, 
one obtains a new $\extg$-module structure on $M$ by twisting the action by a certain 
automorphism $\sigma$ of $\extg$ 
as follows:
\[
x(n) \sigma^{\ast} (v)=\sigma^{\ast} (\sigma^{-1} (x(n)) v), 
\quad\text{for any}\; x\in \g, n \in \Z \text{ and } v\in M. 
\]
To distinguish the two module structures, we will denote the new module by 
$\sigma^{\ast}(M)$. 
Among the automorphisms of $\extg$, the spectral flows are of particular interest. 
We refer \cite{Li} (or to \cite[Appendix A]{Ridout} and references therein) for 
precise definitions and motivations. 

In the following, we consider spectral flow automorphisms which correspond to translations of the extended Weyl group of $\extg$.
More concretely, each simple coroots $\alpha_i^\vee$ 
of $\g$ defines a transformation  
$\tau_i$ that acts on the generators of $\extg$ as follows:
\begin{align*}
&\tau_i(e_\alpha(n))=  e_\alpha(n-\langle \alpha, \alpha_i^\vee\rangle), 
\qquad
\tau_i(h_j(n))= h_j(n) - (\alpha_j^\vee|\alpha_i^\vee) \delta_{n,0}K,\\
&\tau_i(K) = K, 
\qquad \tau_i(L_0) = L_0 - h_i(0)+\frac{(\alpha_i^\vee|\alpha_i^\vee)}{2} K,
\end{align*}
with $n\in\Z$, $\alpha\in\Delta$ and $L_0=-D$. 
The powers of $\tau_i$ acts as follows: 
\begin{align*}
&\tau_i^{s}(e_\alpha(n))=  e_\alpha(n-s\langle \alpha,\alpha_i^\vee\rangle),
\qquad 
\tau_i^{s}(h_j(n))= h_j(n) -  s (\alpha_j^\vee | \alpha_i^\vee)\delta_{n,0} K,\\
&\tau_i^{s}(K) = K, 
\qquad 
\tau_i^{s}(L_0) = L_0 - s\, h_i (0)+ \frac{s^2}{2}(\alpha_i^\vee |\alpha_i^\vee) K.
\end{align*}

Return to the case of $\g=D_4$ 
and 
consider the spectral flow automorphism along the direction 
$\Lambda_1$, which is defined by: 
$\sigma^{-1}:=\tau_1^{1} \tau_2^{1} \tau_3^{1/2} \tau_4^{1/2}$. 
It is direct to check that $\sigma$ is determined by the following maps
\begin{align*}
& e_{\gamma_1}(n)  \longmapsto e_{\gamma_1}(n+1), \quad    
 e_{-\gamma_1}(n) \longmapsto e_{-\gamma_1}(n-1),  
 \quad  e_{\pm \gamma_i}(n) \longmapsto  e_{\pm \gamma_i}(n), &\\
& K  \longmapsto  K, \quad 
h^{D_4}_1({0})  \longmapsto  h^{D_4}_1({0})+K, 
\quad  h^{D_4}_i(0) \longmapsto  h^{D_4}_i(0), \text{ for } i=2,3,4.& 
\end{align*}
In particular $\sigma^{-1}(e_{-\theta}(n))=e_{-\theta}(n-1)$. 

Applying the spectral flow to both sides of (\ref{12}), 
one obtain the decomposition of Theorem~\ref{Th:mainD}.
\begin{proof}[Proof of Theorem~\ref{Th:mainD}]
It is clear that the spectral flow $\sigma$ preserves $\hat{B}_3$, 
and hence the highest-weight module structures for $L_{-2}(D_4)$ and 
for $L_{-2}(B_3)$.
It suffices to show that $\sigma^{\ast} L_{B_3}(-2,\varpi_1)=L_{B_3}(-2,-3\varpi_1)$, 
which follows from the following lemma.

\begin{lem}
\label{Lem:flows}
Let ${\bf 1}_{\varpi_1}$ be the highest-weight vector of $L_{B_3}(-2,\varpi_1)$. 
Then $\sigma^{\ast}(e_{-\beta_1}(0) e_{-\theta}(0) {\bf 1}_{\varpi_1})$ 
is the highest-weight vector of 
$\sigma^{\ast} L_{B_3}(-2,\varpi_1)=L_{B_3}(-2,-3\varpi_1)$.
\end{lem}

\begin{proof}
From the realization of the standard representation $L_{B_3}(\varpi_1)$, 
we deduce that $v=e_{-\beta_1}(0) e_{-\theta}(0) {\bf 1}_{\varpi_1}\neq 0$ 
is the lowest weight vector in $L_{B_3}(\varpi_1)$ of weight $-\varpi_1$. 
By calculating the conformal weight, we have 
\[
\sigma^{-1}(e_{\beta_1}(0)) v
=e_{\beta_1}(1) e_{-\beta_1}(0) e_{-\theta}(0) {\bf 1}_{\varpi_1}=0. 
\]
For $i=2,3$, we have
\[
\sigma^{-1}(e_{\beta_i}(0)) v
=e_{\beta_i}(0) e_{-\beta_1}(0) e_{-\theta}(0) {\bf 1}_{\varpi_1}=0,
\]
where the last equality is due to the fact that 
$-\varpi_1+\beta_i =\beta_1-\beta_i$ 
is not a weight in $L_{B_3}(\varpi_1)$.
Similarly, we have 
\[
\sigma^{-1}(e_{-\theta}(1)) v=e_{-\theta}(0) v=0,
\]
as $-\varpi_1-\theta$ is not a weight in $L_{B_3}(\varpi_1)$.
Therefore $v$ is a singular vector in 
$\sigma^{\ast}(e_{-\beta_1}(0) e_{-\theta}(0) {\bf 1}_{\varpi_1})$ 
of highest-weight $-2\Lambda_0-3\varpi_1$.
\end{proof}
Lemma \ref{Lem:flows} concludes the proof of Theorem~\ref{Th:mainD}.
 \end{proof}

\appendix
\section{Singular vector for $V^{-2}(G_{2})$}
\label{App:singular_G2}

We give in this appendix an explicit description 
of a singular vector $v_{\sing}$ in the
affine vertex algebra $V^{-2}(G_{2})$ 
with weight $-2\Lambda_{0}+4\varpi_{1}-6\delta$ in $V^{-2}(G_{2})$ (385 terms) 
as obtained in Theorem \ref{Th:singular_G2}. The singular vector can be detected and verified in Mathematica using the OPE package by Kris Thielemans.
We also obtain from this the 
image $v'_{\sing}:=F([v_{\sing}])$ of 
$v_{\sing}$ is the Zhu algebra, where $F$ is the isomorphism \eqref{eq:Zhu}.

\smallskip

{\tiny 
\begin{align*}
&v_{\sing}=-60e_{\theta}(-3)e_{\alpha_5}(-2)e_{\alpha_4}(-1)\mathbf{1}
+12e_{\theta}(-3)e_{\alpha_5}(-1)e_{\alpha_4}(-2)\mathbf{1} 
+60e_{\theta}(-2)e_{\alpha_5}(-3)e_{\alpha_4}(-1)\mathbf{1}
-120e_{\theta}(-2)e_{\alpha_5}(-2)e_{\alpha_4}(-2)\mathbf{1}\\
&+36e_{\theta}(-2)e_{\alpha_5}(-1)e_{\alpha_4}(-3)\mathbf{1}
+84e_{\theta}(-3)e_{\alpha_5}(-1)e_{\alpha_4}(-2)\mathbf{1}
-84e_{\theta}(-1)e_{\alpha_5}(-2)e_{\alpha_4}(-3)\mathbf{1}
+8e_{\alpha_4}(-3)e_{\alpha_4}(-1)^{3}\mathbf{1}\\
&-12e_{\alpha_4}(-2)^{2}e_{\alpha_4}(-1)^{2}\mathbf{1}
-12e_{\alpha_5}(-3)e_{\alpha_4}(-1)^{2}e_{\alpha_3}(-1)\mathbf{1}
+48e_{\alpha_5}(-3)e_{\alpha_5}(-1)e_{\alpha_3}(-1)^{2}\mathbf{1}
-36e_{\alpha_5}(-3)e_{\alpha_5}(-1)e_{\alpha_4}(-1)e_{\alpha_2}(-1)\mathbf{1}\\
&+60e_{\alpha_5}(-2)e_{\alpha_4}(-2)e_{\alpha_4}(-1)e_{\alpha_3}(-1)\mathbf{1}
-24e_{\alpha_5}(-2)e_{\alpha_4}(-1)^{2}e_{\alpha_3}(-1)\mathbf{1}
-60e_{\alpha_5}(-2)^{2}e_{\alpha_3}(-1)^{2}\mathbf{1}
+60e_{\alpha_5}(-2)^{2}e_{\alpha_4}(-1)e_{\alpha_2}(-1)\mathbf{1}\\
&-24e_{\alpha_5}(-2)e_{\alpha_5}(-1)e_{\alpha_3}(-2)e_{\alpha_3}(-1)\mathbf{1}
-60e_{\alpha_5}(-2)e_{\alpha_5}(-1)e_{\alpha_4}(-2)e_{\alpha_2}(-1)\mathbf{1}
+48e_{\alpha_5}(-2)e_{\alpha_5}(-1)e_{\alpha_4}(-1)e_{\alpha_2}(-2)\mathbf{1}\\
&-28e_{\alpha_5}(-1)e_{\alpha_4}(-3)e_{\alpha_4}(-1)e_{\alpha_3}(-1)\mathbf{1} 
-12e_{\alpha_5}(-1)e_{\alpha_4}(-2)e_{\alpha_4}(-2)e_{\alpha_3}(-1)\mathbf{1} 
+48e_{\alpha_5}(-1)e_{\alpha_4}(-2)e_{\alpha_4}(-1)e_{\alpha_3}(-2)\mathbf{1}\\
&-8 e_{\alpha_5}(-1)e_{\alpha_4}(-1)e_{\alpha_4}(-1)e_{\alpha_3}(-3)\mathbf{1}
+24 e_{\alpha_5}(-1)e_{\alpha_5}(-1)e_{\alpha_3}(-3)e_{\alpha_3}(-1)\mathbf{1}
-24 e_{\alpha_5}(-1)e_{\alpha_5}(-1)e_{\alpha_3}(-2)e_{\alpha_3}(-2)\mathbf{1}\\
&+36 e_{\alpha_5}(-1)e_{\alpha_5}(-1)e_{\alpha_4}(-3)e_{\alpha_2}(-1)\mathbf{1}
-48 e_{\alpha_5}(-1)e_{\alpha_5}(-1)e_{\alpha_4}(-2)e_{\alpha_2}(-2)\mathbf{1}
-12 e_{\theta}(-3)e_{\alpha_4}(-1)e_{\alpha_4}(-1)e_{\alpha_1}(-1)\mathbf{1}\\
&+48 e_{\theta}(-3)e_{\alpha_5}(-1)e_{\alpha_3}(-1)e_{\alpha_1}(-1)\mathbf{1}
+48 e_{\theta}(-3)e_{\alpha_5}(-1)e_{\alpha_4}(-1)h_{1}(-1)\mathbf{1}
+48 e_{\theta}(-3)e_{\alpha_5}(-1)e_{\alpha_4}(-1)h_2(-1)\mathbf{1}\\
&-36 e_{\theta}(-3)e_{\alpha_5}(-1)^{2}e_{-\alpha_1}(-1)\mathbf{1}
+48 e_{\theta}(-3)e_{\theta}(-1)e_{\alpha_1}(-1)^{2}\mathbf{1}
+24 e_{\theta}(-3)e_{\theta}(-1)e_{\alpha_4}(-1)e_{-\alpha_2}(-1)\mathbf{1}\\
&+48 e_{\theta}(-3)e_{\theta}(-1)e_{\alpha_5}(-1)e_{-\alpha_3}(-1)\mathbf{1}
+60 e_{\theta}(-2)e_{\alpha_4}(-2)e_{\alpha_4}(-1)e_{\alpha_1}(-1)\mathbf{1}
-12 e_{\theta}(-2)e_{\alpha_4}(-1)^{2}e_{\alpha_1}(-2)\mathbf{1}\\
&-120 e_{\theta}(-2)e_{\alpha_5}(-2)e_{\alpha_3}(-1)e_{\alpha_1}(-1)\mathbf{1}
+60 e_{\theta}(-2)e_{\alpha_5}(-2)e_{\alpha_4}(-1)h_2(-1)\mathbf{1}
+72 e_{\theta}(-2)e_{\alpha_5}(-2)e_{\alpha_5}(-1)e_{-\alpha_1}(-1)\mathbf{1}\\
&-36 e_{\theta}(-2)e_{\alpha_5}(-1)e_{\alpha_3}(-2)e_{\alpha_1}(-1)\mathbf{1}
+48 e_{\theta}(-2)e_{\alpha_5}(-1)e_{\alpha_3}(-1)e_{\alpha_1}(-2)\mathbf{1}
+72 e_{\theta}(-2)e_{\alpha_5}(-1)e_{\alpha_4}(-2)h_1(-1)\mathbf{1}\\
&+36 e_{\theta}(-2)e_{\alpha_5}(-1)e_{\alpha_4}(-2)h_2(-1)\mathbf{1}
-16 e_{\theta}(-2)e_{\alpha_5}(-1)e_{\alpha_4}(-1)h_1(-2)\mathbf{1}
-36 e_{\theta}(-2)e_{\alpha_5}(-1)e_{\alpha_4}(-1)h_2(-2)\mathbf{1}\\
&-48 e_{\theta}(-2)e_{\alpha_5}(-1)^{2}e_{-\alpha_1}(-2)\mathbf{1}
-60 e_{\theta}(-2)^{2}e_{\alpha_1}(-1)^{2}\mathbf{1}
-60 e_{\theta}(-2)^{2}e_{\alpha_4}(-1)e_{-\alpha_2}(-1)\mathbf{1}
-72 e_{\theta}(-2)^{2}e_{\alpha_5}(-1)e_{-\alpha_3}(-1)\mathbf{1}\\
&-72 e_{\theta}(-2)e_{\theta}(-1)e_{\alpha_1}(-2)e_{\alpha_1}(-1)\mathbf{1}
+156 e_{\theta}(-2)e_{\theta}(-1)e_{\alpha_4}(-2)e_{-\alpha_2}(-1)\mathbf{1}
-48 e_{\theta}(-2)e_{\theta}(-1)e_{\alpha_4}(-1)e_{-\alpha_2}(-2)\mathbf{1}\\
&+72 e_{\theta}(-2)e_{\theta}(-1)e_{\alpha_5}(-2)e_{-\alpha_3}(-1)\mathbf{1}
+48 e_{\theta}(-2)e_{\theta}(-1)e_{\alpha_5}(-1)e_{-\alpha_3}(-2)\mathbf{1}
-36 e_{\theta}(-1)e_{\alpha_4}(-3)e_{\alpha_4}(-1)e_{\alpha_1}(-1)\mathbf{1}\\
&-12 e_{\theta}(-1)e_{\alpha_4}(-2)^{2}e_{\alpha_1}(-1)\mathbf{1}
+60 e_{\theta}(-1)e_{\alpha_4}(-2)e_{\alpha_4}(-1)e_{\alpha_1}(-2)\mathbf{1}
-12 e_{\theta}(-1)e_{\alpha_4}(-1)^{2}e_{\alpha_1}(-3)\mathbf{1}\\
&+48 e_{\theta}(-1)e_{\alpha_5}(-3)e_{\alpha_3}(-1)e_{\alpha_1}(-1)\mathbf{1}
-48 e_{\theta}(-1)e_{\alpha_5}(-3)e_{\alpha_4}(-1)h_1(-1)\mathbf{1}
-84 e_{\theta}(-1)e_{\alpha_5}(-3)e_{\alpha_4}(-1)h_2(-1)\mathbf{1}\\
&+36 e_{\theta}(-1)e_{\alpha_5}(-3)e_{\alpha_5}(-1)e_{-\alpha_1}(-1)\mathbf{1}
+60 e_{\theta}(-1)e_{\alpha_5}(-2)e_{\alpha_3}(-2)e_{\alpha_1}(-1)\mathbf{1}
-120 e_{\theta}(-1)e_{\alpha_5}(-2)e_{\alpha_3}(-1)e_{\alpha_1}(-2)\mathbf{1}\\
&+24 e_{\theta}(-1)e_{\alpha_5}(-2)e_{\alpha_4}(-2)h_1(-1)\mathbf{1}
+16 e_{\theta}(-1)e_{\alpha_5}(-2)e_{\alpha_4}(-1)h_1(-2)\mathbf{1}
+84 e_{\theta}(-1)e_{\alpha_5}(-2)e_{\alpha_4}(-1)h_2(-2)\mathbf{1}\\
&-72 e_{\theta}(-1)e_{\alpha_5}(-2)^{2}e_{-\alpha_1}(-1)\mathbf{1}
+48 e_{\theta}(-1)e_{\alpha_5}(-2)e_{\alpha_5}(-1)e_{-\alpha_1}(-2)\mathbf{1}
+48 e_{\theta}(-1)e_{\alpha_5}(-1)e_{\alpha_3}(-3)e_{\alpha_1}(-1)\mathbf{1}\\
&-96 e_{\theta}(-1)e_{\alpha_5}(-1)e_{\alpha_3}(-2)e_{\alpha_1}(-2)\mathbf{1}
+48 e_{\theta}(-1)e_{\alpha_5}(-1)e_{\alpha_3}(-1)e_{\alpha_1}(-3)\mathbf{1}
+36 e_{\theta}(-1)e_{\alpha_5}(-1)e_{\alpha_4}(-3)h_1(-1)\mathbf{1}\\
&+84 e_{\theta}(-1)e_{\alpha_5}(-1)e_{\alpha_4}(-3)h_2(-1)\mathbf{1}
-32 e_{\theta}(-1)e_{\alpha_5}(-1)e_{\alpha_4}(-2)h_1(-2)\mathbf{1}
-84 e_{\theta}(-1)e_{\alpha_5}(-1)e_{\alpha_4}(-2)h_2(-2)\mathbf{1}\\
&+48 e_{\theta}(-1)^{2}e_{\alpha_1}(-3)e_{\alpha_1}(-1)\mathbf{1}
-60 e_{\theta}(-1)^{2}e_{\alpha_1}(-2)^{2}\mathbf{1}
-42 e_{\theta}(-1)^{2}e_{\alpha_4}(-3)e_{-\alpha_2}(-1)\mathbf{1}\\
&+42 e_{\theta}(-1)^{2}e_{\alpha_4}(-2)e_{-\alpha_2}(-2)\mathbf{1}
-42 e_{\theta}(-1)^{2}e_{\alpha_5}(-3)e_{-\alpha_3}(-1)\mathbf{1}
-48 e_{\theta}(-1)^{2}e_{\alpha_5}(-2)e_{-\alpha_3}(-2)\mathbf{1}\\
&+6 e_{\alpha_4}(-2)e_{\alpha_4}(-1)^{3}h_1(-1)\mathbf{1}
+8 e_{\alpha_4}(-2)e_{\alpha_4}(-1)^{3}h_2(-1)\mathbf{1}
+2 e_{\alpha_4}(-1)^{3}e_{\alpha_3}(-2)e_{\alpha_1}(-1)\mathbf{1}
-e_{\alpha_4}(-1)^{4}h_1(-2)\mathbf{1}\\
&-2 e_{\alpha_4}(-1)^{4}h_2(-2)\mathbf{1}
-12 e_{\alpha_5}(-2)e_{\alpha_4}(-1)^{2}e_{\alpha_2}(-1)e_{\alpha_1}(-1)\mathbf{1}
-12 e_{\alpha_5}(-2)e_{\alpha_4}(-1)^{2}e_{\alpha_3}(-1)h_1(-1)\mathbf{1}\\
&-12 e_{\alpha_5}(-2)e_{\alpha_4}(-1)^{2}e_{\alpha_3}(-1)h_2(-1)\mathbf{1}
+6 e_{\alpha_5}(-2)e_{\alpha_4}(-1)^{3}e_{-\alpha_1}(-1)\mathbf{1}
+48 e_{\alpha_5}(-2)e_{\alpha_5}(-1)e_{\alpha_3}(-1)e_{\alpha_2}(-1)e_{\alpha_1}(-1)\mathbf{1}\\
&+48 e_{\alpha_5}(-2)e_{\alpha_5}(-1)e_{\alpha_3}(-1)^{2}h_1(-1)\mathbf{1}
+48 e_{\alpha_5}(-2)e_{\alpha_5}(-1)e_{\alpha_3}(-1)^{2}h_2(-1)\mathbf{1}
-24 e_{\alpha_5}(-2)e_{\alpha_5}(-1)e_{\alpha_4}(-1)e_{\alpha_2}(-1)h_1(-1)\mathbf{1}\\
&-36 e_{\alpha_5}(-2)e_{\alpha_5}(-1)e_{\alpha_4}(-1)e_{\alpha_2}(-1)h_2(-1)\mathbf{1}
-24 e_{\alpha_5}(-2)e_{\alpha_5}(-1)e_{\alpha_4}(-1)e_{\alpha_3}(-1)e_{-\alpha_1}(-1)\mathbf{1}\\
&+18 e_{\alpha_5}(-2)e_{\alpha_5}(-1)^{2}e_{\alpha_2}(-1)e_{-\alpha_1}(-1)\mathbf{1}
-12 e_{\alpha_5}(-1)e_{\alpha_4}(-2)e_{\alpha_4}(-1)e_{\alpha_2}(-1)e_{\alpha_1}(-1)\mathbf{1} \\
&-24 e_{\alpha_5}(-1)e_{\alpha_4}(-2)e_{\alpha_4}(-1)e_{\alpha_3}(-1)h_1(-1)\mathbf{1}
-28 e_{\alpha_5}(-1)e_{\alpha_4}(-2)e_{\alpha_4}(-1)e_{\alpha_3}(-1)h_2(-1)\mathbf{1}\\
&-2 e_{\alpha_5}(-1)e_{\alpha_4}(-2)e_{\alpha_4}(-1)^{2}e_{-\alpha_1}(-1)\mathbf{1}
-4 e_{\alpha_5}(-1)e_{\alpha_4}(-1)e_{\alpha_3}(-2)e_{\alpha_3}(-1)e_{\alpha_1}(-1)\mathbf{1}
-6 e_{\alpha_5}(-1)e_{\alpha_4}(-1)^{2}e_{\alpha_2}(-2)e_{\alpha_1}(-1)\mathbf{1}\\
&+10 e_{\alpha_5}(-1)e_{\alpha_4}(-1)^{2}e_{\alpha_2}(-1)e_{\alpha_1}(-2)\mathbf{1}
-2 e_{\alpha_5}(-1)e_{\alpha_4}(-1)^{2}e_{\alpha_3}(-2)h_1(-1)\mathbf{1}
-8 e_{\alpha_5}(-1)e_{\alpha_4}(-1)^{2}e_{\alpha_3}(-2)h_2(-1)\mathbf{1}\\
&+8 e_{\alpha_5}(-1)e_{\alpha_4}(-1)^{2}e_{\alpha_3}(-1)h_1(-2)\mathbf{1}
+12 e_{\alpha_5}(-1)e_{\alpha_4}(-1)^{2}e_{\alpha_3}(-1)h_2(-2)\mathbf{1}
-4 e_{\alpha_5}(-1)e_{\alpha_4}(-1)^{3}e_{-\alpha_1}(-2)\mathbf{1}\\
&+18 e_{\alpha_5}(-1)^{2}e_{\alpha_3}(-2)e_{\alpha_2}(-1)e_{\alpha_1}(-1)\mathbf{1}
+18 e_{\alpha_5}(-1)^{2}e_{\alpha_3}(-2)e_{\alpha_3}(-1)h_1(-1)\mathbf{1}
+24 e_{\alpha_5}(-1)^{2}e_{\alpha_3}(-2)e_{\alpha_3}(-1)h_2(-1)\mathbf{1}\\
&+6 e_{\alpha_5}(-1)^{2}e_{\alpha_3}(-1)e_{\alpha_2}(-2)e_{\alpha_1}(-1)\mathbf{1}
-30 e_{\alpha_5}(-1)^{2}e_{\alpha_3}(-1)e_{\alpha_2}(-1)e_{\alpha_1}(-2)\mathbf{1}
-16 e_{\alpha_5}(-1)^{2}e_{\alpha_3}(-1)^{2}h_1(-2)\mathbf{1}\\
&-18 e_{\alpha_5}(-1)^{2}e_{\alpha_3}(-1)^{2}h_2(-2)\mathbf{1}
+18 e_{\alpha_5}(-1)^{2}e_{\alpha_4}(-2)e_{\alpha_2}(-1)h_1(-1)\mathbf{1}
+36 e_{\alpha_5}(-1)^{2}e_{\alpha_4}(-2)e_{\alpha_2}(-1)h_2(-1)\mathbf{1}\\
&+18 e_{\alpha_5}(-1)^{2}e_{\alpha_4}(-2)e_{\alpha_3}(-1)e_{-\alpha_1}(-1)\mathbf{1}
-12 e_{\alpha_5}(-1)^{2}e_{\alpha_4}(-1)e_{\alpha_2}(-2)h_1(-1)\mathbf{1}
+3 e_{\alpha_5}(-1)^{2}e_{\alpha_4}(-1)e_{\alpha_2}(-1)h_1(-2)\mathbf{1}\\
&-12 e_{\alpha_5}(-1)^{2}e_{\alpha_4}(-1)e_{\alpha_3}(-2)e_{-\alpha_1}(-1)\mathbf{1}
+13 e_{\alpha_5}(-1)^{2}e_{\alpha_4}(-1)e_{\alpha_3}(-1)e_{-\alpha_1}(-2)\mathbf{1}
+36 e_{\alpha_5}(-1)^{3}e_{\alpha_2}(-2)e_{-\alpha_1}(-1)\mathbf{1}\\
&-9 e_{\alpha_5}(-1)^{3}e_{\alpha_2}(-1)e_{-\alpha_1}(-2)\mathbf{1}
-12 e_{\theta}(-2)e_{\alpha_4}(-1)^{2}e_{\alpha_1}(-1)h_1(-1)\mathbf{1} 
-24 e_{\theta}(-2)e_{\alpha_4}(-1)^{2}e_{\alpha_1}(-1)h_2(-1)\mathbf{1}\\
&-12 e_{\theta}(-2)e_{\alpha_4}(-1)^{2}e_{\alpha_3}(-1)e_{-\alpha_2}(-1)\mathbf{1}
-6 e_{\theta}(-2)e_{\alpha_4}(-1)^{3}e_{-\alpha_3}(-1)\mathbf{1}
+42 e_{\theta}(-2)e_{\alpha_5}(-1)e_{\alpha_2}(-1)e_{\alpha_1}(-1)^{2}\mathbf{1}\\
%
&+48 e_{\theta}(-2)e_{\alpha_5}(-1)e_{\alpha_3}(-1)e_{\alpha_1}(-1)h_1(-1)\mathbf{1}
+54 e_{\theta}(-2)e_{\alpha_5}(-1)e_{\alpha_3}(-1)e_{\alpha_1}(-1)h_2(-1)\mathbf{1}\\
&+6 e_{\theta}(-2)e_{\alpha_5}(-1)e_{\alpha_3}(-1)^{2}e_{-\alpha_2}(-1)\mathbf{1}
-28 e_{\theta}(-2)e_{\alpha_5}(-1)e_{\alpha_4}(-1)e_{\alpha_1}(-1)e_{-\alpha_1}(-1)\mathbf{1}
+72 e_{\theta}(-2)e_{\alpha_5}(-1)e_{\alpha_4}(-1)e_{\alpha_2}(-1)e_{-\alpha_2}(-1)\mathbf{1}\\
&-4 e_{\theta}(-2)e_{\alpha_5}(-1)e_{\alpha_4}(-1)e_{\alpha_3}(-1)e_{-\alpha_3}(-1)\mathbf{1}
-40 e_{\theta}(-2)e_{\alpha_5}(-1)e_{\alpha_4}(-1)^{2}e_{-\alpha_4}(-1)\mathbf{1}
-16 e_{\theta}(-2)e_{\alpha_5}(-1)e_{\alpha_4}(-1)h_1(-1)^{2}\mathbf{1}\\
&-60 e_{\theta}(-2)e_{\alpha_5}(-1)e_{\alpha_4}(-1)h_1(-1)h_2(-1)\mathbf{1}
-36 e_{\theta}(-2)e_{\alpha_5}(-1)e_{\alpha_4}(-1)h_2(-1)^{2}\mathbf{1}
+90 e_{\theta}(-2)e_{\alpha_5}(-1)^{2}e_{\alpha_2}(-1)e_{-\alpha_3}(-1)\mathbf{1}&\\
&+48 e_{\theta}(-2)e_{\alpha_5}(-1)^{2}e_{\alpha_3}(-1)e_{-\alpha_4}(-1)\mathbf{1}
-126 e_{\theta}(-2)e_{\alpha_5}(-1)^{2}e_{\alpha_4}(-1)e_{-\alpha_5}(-1)\mathbf{1}
-48 e_{\theta}(-2)e_{\alpha_5}(-1)^{2}h_1(-1)e_{-\alpha_1}(-1)\mathbf{1}\\
&-18 e_{\theta}(-2)e_{\alpha_5}(-1)^{2}h_2(-1)e_{-\alpha_1}(-1)\mathbf{1}
+48 e_{\theta}(-2)e_{\theta}(-1)e_{\alpha_1}(-1)^{2}h_1(-1)\mathbf{1}
+96 e_{\theta}(-2)e_{\theta}(-1)e_{\alpha_1}(-1)^{2}.h_2(-1)\mathbf{1}\\
&+48 e_{\theta}(-2)e_{\theta}(-1)e_{\alpha_3}(-1)e_{\alpha_1}(-1)e_{-\alpha_2}(-1)\mathbf{1}
+24 e_{\theta}(-2)e_{\theta}(-1)e_{\alpha_4}(-1)e_{\alpha_1}(-1)e_{-\alpha_3}(-1)\mathbf{1}\\
&+24 e_{\theta}(-2)e_{\theta}(-1)e_{\alpha_4}(-1)h_1(-1)e_{-\alpha_2}(-1)\mathbf{1}		
+36 e_{\theta}(-2)e_{\theta}(-1)e_{\alpha_4}(-1)h_2(-1)e_{-\alpha_2}(-1)\mathbf{1}\\
&+48 e_{\theta}(-2)e_{\theta}(-1)e_{\alpha_5}(-1)e_{\alpha_1}(-1)e_{-\alpha_4}(-1)\mathbf{1}
-108 e_{\theta}(-2)e_{\theta}(-1)e_{\alpha_5}(-1)e_{-\alpha_2}(-1)e_{-\alpha_1}(-1)\mathbf{1}\\
&-126 e_{\theta}(-2)e_{\theta}(-1)e_{\alpha_5}(-1)e_{\alpha_4}(-1)e_{-\theta}(-1)\mathbf{1}
+48 e_{\theta}(-2)e_{\theta}(-1)e_{\alpha_5}(-1)h_1(-1)e_{-\alpha_3}(-1)\mathbf{1}\\
&+108 e_{\theta}(-2)e_{\theta}(-1)e_{\alpha_5}(-1)h_2(-1)e_{-\alpha_3}(-1)\mathbf{1}
+18 e_{\theta}(-2)e_{\theta}(-1)^{2}e_{-\alpha_3}(-1)e_{-\alpha_2}(-1)\mathbf{1}
-24 e_{\theta}(-1)e_{\alpha_4}(-2)e_{\alpha_4}(-1)e_{\alpha_1}(-1)h_1(-1)\mathbf{1}\\
&-36 e_{\theta}(-1)e_{\alpha_4}(-2)e_{\alpha_4}(-1)e_{\alpha_1}(-1)h_2(-1)\mathbf{1}
-4 e_{\theta}(-1)e_{\alpha_4}(-2)e_{\alpha_4}(-1)e_{\alpha_3}(-1)e_{-\alpha_2}(-1)\mathbf{1}
+8 e_{\theta}(-1)e_{\alpha_4}(-2)e_{\alpha_4}(-1)^{2}e_{-\alpha_3}(-1)\mathbf{1}\\
&-6 e_{\theta}(-1)e_{\alpha_4}(-1)e_{\alpha_3}(-2)e_{\alpha_1}(-1)^{2}\mathbf{1}
-12 e_{\theta}(-1)e_{\alpha_4}(-1)^{2}e_{\alpha_1}(-2)h_1(-1)\mathbf{1}
-13 e_{\theta}(-1)e_{\alpha_4}(-1)^{2}e_{\alpha_1}(-2)h_2(-1)\mathbf{1}\\
&+8 e_{\theta}(-1)e_{\alpha_4}(-1)^{2}e_{\alpha_1}(-1)h_1(-2)\mathbf{1}
+6 e_{\theta}(-1)e_{\alpha_4}(-1)^{2}e_{\alpha_1}(-1)h_2(-2)\mathbf{1}
+16 e_{\theta}(-1)e_{\alpha_4}(-1)^{2}e_{\alpha_3}(-2)e_{-\alpha_2}(-1)\mathbf{1}\\
&-6 e_{\theta}(-1)e_{\alpha_4}(-1)^{2}e_{\alpha_3}(-1)e_{-\alpha_2}(-2)\mathbf{1}
+4 e_{\theta}(-1)e_{\alpha_4}(-1)^{3}e_{-\alpha_3}(-2)\mathbf{1}
+6 e_{\theta}(-1)e_{\alpha_5}(-2)e_{\alpha_2}(-1)e_{\alpha_1}(-1)^{2}\mathbf{1}\\
&+48 e_{\theta}(-1)e_{\alpha_5}(-2)e_{\alpha_3}(-1)e_{\alpha_1}(-1)h_1(-1)\mathbf{1}
+90 e_{\theta}(-1)e_{\alpha_5}(-2)e_{\alpha_3}(-1)e_{\alpha_1}(-1)h_2(-1)\mathbf{1}
+42 e_{\theta}(-1)e_{\alpha_5}(-2)e_{\alpha_3}(-1)^{2}e_{-\alpha_2}(-1)\mathbf{1}\\
&+4 e_{\theta}(-1)e_{\alpha_5}(-2)e_{\alpha_4}(-1)e_{\alpha_1}(-1)e_{-\alpha_1}(-1)\mathbf{1}
-72 e_{\theta}(-1)e_{\alpha_5}(-2)e_{\alpha_4}(-1)e_{\alpha_2}(-1)e_{-\alpha_2}(-1)\mathbf{1}\\
&+28 e_{\theta}(-1)e_{\alpha_5}(-2)e_{\alpha_4}(-1)e_{\alpha_3}(-1)e_{-\alpha_3}(-1)\mathbf{1}
+40 e_{\theta}(-1)e_{\alpha_5}(-2)e_{\alpha_4}(-1)^{2}e_{-\alpha_4}(-1)\mathbf{1}
+16 e_{\theta}(-1)e_{\alpha_5}(-2)e_{\alpha_4}(-1)h_1(-1)^{2}\mathbf{1} \\
&+36 e_{\theta}(-1)e_{\alpha_5}(-2)e_{\alpha_4}(-1)h_1(-1)h_2(-1)\mathbf{1}
-108 e_{\theta}(-1)e_{\alpha_5}(-2)e_{\alpha_5}(-1)e_{\alpha_2}(-1)e_{-\alpha_3}(-1)\mathbf{1}\\
&-48 e_{\theta}(-1)e_{\alpha_5}(-2)e_{\alpha_5}(-1)e_{\alpha_3}(-1)e_{-\alpha_4}(-1)\mathbf{1}
+126 e_{\theta}(-1)e_{\alpha_5}(-2)e_{\alpha_5}(-1)e_{\alpha_4}(-1)e_{-\alpha_5}(-1)\mathbf{1}\\
&+48 e_{\theta}(-1)e_{\alpha_5}(-2)e_{\theta}(-1)h_1(-1)e_{-\alpha_1}(-1)\mathbf{1}
+36 e_{\theta}(-1)e_{\alpha_5}(-2)e_{\theta}(-1)h_2(-1)e_{-\alpha_1}(-1)\mathbf{1}
-12 e_{\theta}(-1)e_{\alpha_5}(-1)e_{\alpha_2}(-2)e_{\alpha_1}(-1)^{2}\mathbf{1}\\
&+24 e_{\theta}(-1)e_{\alpha_5}(-1)e_{\alpha_2}(-1)e_{\alpha_1}(-2)e_{\alpha_1}(-1)\mathbf{1}
-12 e_{\theta}(-1)e_{\alpha_5}(-1)e_{\alpha_3}(-2)e_{\alpha_1}(-1)h_1(-1)\mathbf{1}\\
&+6 e_{\theta}(-1)e_{\alpha_5}(-1)e_{\alpha_3}(-2)e_{\alpha_1}(-1)h_2(-1)\mathbf{1}
-12 e_{\theta}(-1)e_{\alpha_5}(-1)e_{\alpha_3}(-2)e_{\alpha_3}(-1)e_{-\alpha_2}(-1)\mathbf{1}\\
&+48 e_{\theta}(-1)e_{\alpha_5}(-1)e_{\alpha_3}(-1)e_{\alpha_1}(-2)h_1(-1)\mathbf{1}
+30 e_{\theta}(-1)e_{\alpha_5}(-1)e_{\alpha_3}(-1)e_{\alpha_1}(-2)h_2(-1)\mathbf{1}\\
&-32 e_{\theta}(-1)e_{\alpha_5}(-1)e_{\alpha_3}(-1)e_{\alpha_1}(-1)h_1(-2)\mathbf{1}
-30 e_{\theta}(-1)e_{\alpha_5}(-1)e_{\alpha_3}(-1)e_{\alpha_1}(-1)h_2(-2)\mathbf{1}
+6 e_{\theta}(-1)e_{\alpha_5}(-1)e_{\alpha_3}(-1)^{2}e_{-\alpha_2}(-2)\mathbf{1}\\
&+16 e_{\theta}(-1)e_{\alpha_5}(-1)e_{\alpha_4}(-2)e_{\alpha_1}(-1)e_{-\alpha_1}(-1)\mathbf{1}
-36 e_{\theta}(-1)e_{\alpha_5}(-1)e_{\alpha_4}(-2)e_{\alpha_2}(-1)e_{-\alpha_2}(-1)\mathbf{1}\\
&-24 e_{\theta}(-1)e_{\alpha_5}(-1)e_{\alpha_4}(-2)e_{\alpha_3}(-1)e_{-\alpha_3}(-1)\mathbf{1} 
-6 e_{\theta}(-1)e_{\alpha_5}(-1)e_{\alpha_4}(-2)e_{\alpha_4}(-1)e_{-\alpha_4}(-1)\mathbf{1}
-32 e_{\theta}(-1)e_{\alpha_5}(-1)e_{\alpha_4}(-2)h_1(-1)^{2}\mathbf{1}\\
&-48 e_{\theta}(-1)e_{\alpha_5}(-1)e_{\alpha_4}(-2)h_1(-1)h_2(-1)\mathbf{1}
-15 e_{\theta}(-1)e_{\alpha_5}(-1)e_{\alpha_4}(-1)e_{\alpha_1}(-2)e_{-\alpha_1}(-1)\mathbf{1}\\
&+13 e_{\theta}(-1)e_{\alpha_5}(-1)e_{\alpha_4}(-1)e_{\alpha_1}(-1)e_{-\alpha_1}(-2)\mathbf{1}
-90 e_{\theta}(-1)e_{\alpha_5}(-1)e_{\alpha_4}(-1)e_{\alpha_2}(-2)e_{-\alpha_2}(-1)\mathbf{1}\\
&+36 e_{\theta}(-1)e_{\alpha_5}(-1)e_{\alpha_4}(-1)e_{\alpha_2}(-1)e_{-\alpha_2}(-2)\mathbf{1}
-6 e_{\theta}(-1)e_{\alpha_5}(-1)e_{\alpha_4}(-1)e_{\alpha_3}(-2)e_{-\alpha_3}(-1)\mathbf{1}\\
&-13 e_{\theta}(-1)e_{\alpha_5}(-1)e_{\alpha_4}(-1)e_{\alpha_3}(-1)e_{-\alpha_3}(-2)\mathbf{1}
+10 e_{\theta}(-1)e_{\alpha_5}(-1)e_{\alpha_4}(-1)^{2}e_{-\alpha_4}(-2)\mathbf{1}\\
&+3 e_{\theta}(-1)e_{\alpha_5}(-1)e_{\alpha_4}(-1)h_1(-2)h_2(-1)\mathbf{1} 
-12 e_{\theta}(-1)e_{\alpha_5}(-1)e_{\alpha_4}(-1)h_1(-1)h_2(-2)\mathbf{1}
-90 e_{\theta}(-1)e_{\alpha_5}(-1)^{2}e_{\alpha_2}(-2)e_{-\alpha_3}(-1)\mathbf{1}\\
&+9 e_{\theta}(-1)e_{\alpha_5}(-1)^{2}e_{\alpha_2}(-1)e_{-\alpha_3}(-2)\mathbf{1}
-54 e_{\theta}(-1)e_{\alpha_5}(-1)^{2}e_{\alpha_3}(-2)e_{-\alpha_4}(-1)\mathbf{1}
-30 e_{\theta}(-1)e_{\alpha_5}(-1)^{2}e_{\alpha_3}(-1)e_{-\alpha_4}(-2)\mathbf{1}\\
&-126 e_{\theta}(-1)e_{\alpha_5}(-1)^{2}e_{\alpha_4}(-2)e_{-\alpha_5}(-1)\mathbf{1}
+36 e_{\theta}(-1)e_{\alpha_5}(-1)^{2}h_2(-2)e_{-\alpha_1}(-1)\mathbf{1}
-9 e_{\theta}(-1)e_{\alpha_5}(-1)^{2}h_2(-1)e_{-\alpha_1}(-2)\mathbf{1}\\
&+48 e_{\theta}(-1)^{2}e_{\alpha_1}(-2)e_{\alpha_1}(-1)h_1(-1)\mathbf{1}
+84 e_{\theta}(-1)^{2}e_{\alpha_1}(-2)e_{\alpha_1}(-1)h_2(-1)\mathbf{1}
-16 e_{\theta}(-1)^{2}e_{\alpha_1}(-1)^{2}h_1(-2)\mathbf{1}\\
&-21 e_{\theta}(-1)^{2}e_{\alpha_1}(-1)^{2}h_2(-2)\mathbf{1}
-48 e_{\theta}(-1)^{2}e_{\alpha_3}(-2)e_{\alpha_1}(-1)e_{-\alpha_2}(-1)\mathbf{1}
+54 e_{\theta}(-1)^{2}e_{\alpha_3}(-1)e_{\alpha_1}(-2)e_{-\alpha_2}(-1)\mathbf{1}\\
&-12 e_{\theta}(-1)^{2}e_{\alpha_3}(-1)e_{\alpha_1}(-1)e_{-\alpha_2}(-2)\mathbf{1}
-22 e_{\theta}(-1)^{2}e_{\alpha_4}(-2)e_{\alpha_1}(-1)e_{-\alpha_3}(-1)\mathbf{1}
-78 e_{\theta}(-1)^{2}e_{\alpha_4}(-2)h_1(-1)e_{-\alpha_2}(-1)\mathbf{1}\\
&-126 e_{\theta}(-1)^{2}e_{\alpha_4}(-2)h_2(-1)e_{-\alpha_2}(-1)\mathbf{1}
+15 e_{\theta}(-1)^{2}e_{\alpha_4}(-1)e_{\alpha_1}(-2)e_{-\alpha_3}(-1)\mathbf{1}
-13 e_{\theta}(-1)^{2}e_{\alpha_4}(-1)e_{\alpha_1}(-1)e_{-\alpha_3}(-2)\mathbf{1}\\
&-3 e_{\theta}(-1)^{2}e_{\alpha_4}(-1)h_1(-2)e_{-\alpha_2}(-1)\mathbf{1}
+12 e_{\theta}(-1)^{2}e_{\alpha_4}(-1)h_1(-1)e_{-\alpha_2}(-2)\mathbf{1}
-63 e_{\theta}(-1)^{2}e_{\alpha_4}(-1)h_2(-2)e_{-\alpha_2}(-1)\mathbf{1}\\
&+63 e_{\theta}(-1)^{2}e_{\alpha_4}(-1)h_2(-1)e_{-\alpha_2}(-2)\mathbf{1}
-48 e_{\theta}(-1)^{2}e_{\alpha_5}(-2)e_{\alpha_1}(-1)e_{-\alpha_4}(-1)\mathbf{1}
+90 e_{\theta}(-1)^{2}e_{\alpha_5}(-2)e_{-\alpha_2}(-1)e_{-\alpha_1}(-1)\mathbf{1}\\
&+126 e_{\theta}(-1)^{2}e_{\alpha_5}(-2)e_{\alpha_4}(-1)e_{-\theta}(-1)\mathbf{1}
-48 e_{\theta}(-1)^{2}e_{\theta}(-2)h_1(-1)e_{-\alpha_3}(-1)\mathbf{1}
-126 e_{\theta}(-1)^{2}e_{\theta}(-2)h_2(-1)e_{-\alpha_3}(-1)\mathbf{1}\\
&-30 e_{\theta}(-1)^{2}e_{\alpha_5}(-1)e_{\alpha_1}(-1)e_{-\alpha_4}(-2)\mathbf{1} 
-9 e_{\theta}(-1)^{2}e_{\alpha_5}(-1)e_{-\alpha_2}(-2)e_{-\alpha_1}(-1)\mathbf{1}
+9 e_{\theta}(-1)^{2}e_{\alpha_5}(-1)e_{-\alpha_2}(-1)e_{-\alpha_1}(-2)\mathbf{1}\\
&-126 e_{\theta}(-1)^{2}e_{\alpha_5}(-1)e_{\alpha_4}(-2)e_{-\theta}(-1)\mathbf{1}
-63 e_{\theta}(-1)^{2}e_{\theta}(-1)h_2(-2)e_{-\alpha_3}(-1)\mathbf{1}
+9 e_{\theta}(-1)e_{\theta}(-1)e_{\theta}(-1)h_2(-1)e_{-\alpha_3}(-2)\mathbf{1}\\
&-9 e_{\theta}(-1)^{3}e_{-\alpha_3}(-2)e_{-\alpha_2}(-1)\mathbf{1}
+63 e_{\theta}(-1)^{3}e_{-\alpha_3}(-1)e_{-\alpha_2}(-2)\mathbf{1}
+e_{\alpha_4}(-1)^{3}e_{\alpha_2}(-1)e_{\alpha_1}(-1)^{2}\mathbf{1}
-e_{\alpha_4}(-1)^{3}e_{\alpha_3}(-1)e_{\alpha_1}(-1)h_2(-1)\mathbf{1}\\
&-e_{\alpha_4}(-1)^{3}e_{\alpha_3}(-1)^{2}e_{-\alpha_2}(-1)\mathbf{1}
-e_{\alpha_4}(-1)^{4}e_{\alpha_1}(-1)e_{-\alpha_1}(-1)\mathbf{1}
+e_{\alpha_4}(-1)^{4}e_{\alpha_2}(-1)e_{-\alpha_2}(-1)\mathbf{1}\\
&-e_{\alpha_4}(-1)^{4}e_{\alpha_3}(-1)e_{-\alpha_3}(-1)\mathbf{1}
-e_{\alpha_4}(-1)^{5}e_{-\alpha_4}(-1)\mathbf{1}
-e_{\alpha_4}(-1)^{4}h_1(-1)^{2}\mathbf{1}-3 e_{\alpha_4}(-1)^{4}h_1(-1)h_2(-1)\mathbf{1}
-2 e_{\alpha_4}(-1)^{4}h_2(-1)^{2}\mathbf{1}\\
&-2 e_{\alpha_5}(-1)e_{\alpha_4}(-1)e_{\alpha_3}(-1)e_{\alpha_2}(-1)e_{\alpha_1}(-1)^{2}\mathbf{1}
+2 e_{\alpha_5}(-1)e_{\alpha_4}(-1)e_{\alpha_3}(-1)^{2}e_{\alpha_1}(-1)h_2(-1)\mathbf{1}
+2 e_{\alpha_5}(-1)e_{\alpha_4}(-1)e_{\alpha_3}(-1)^{3}e_{-\alpha_2}(-1)\mathbf{1}\\
&+10 e_{\alpha_5}(-1)e_{\alpha_4}(-1)^{2}e_{\alpha_2}(-1)e_{\alpha_1}(-1)h_1(-1)\mathbf{1}
+15 e_{\alpha_5}(-1)e_{\alpha_4}(-1)^{2}e_{\alpha_2}(-1)e_{\alpha_1}(-1)h_2(-1)\mathbf{1}\\
&+3 e_{\alpha_5}(-1)e_{\alpha_4}(-1)^{2}e_{\alpha_3}(-1)e_{\alpha_1}(-1)e_{-\alpha_1}(-1)\mathbf{1}
+6 e_{\alpha_5}(-1)e_{\alpha_4}(-1)^{2}e_{\alpha_3}(-1)e_{\alpha_2}(-1)e_{-\alpha_2}(-1)\mathbf{1}\\
&+3 e_{\alpha_5}(-1)e_{\alpha_4}(-1)^{2}e_{\alpha_3}(-1)^{2}e_{-\alpha_3}(-1)\mathbf{1}
+8 e_{\alpha_5}(-1)e_{\alpha_4}(-1)^{2}e_{\alpha_3}(-1)h_1(-1)^{2}\mathbf{1}\\
&+19 e_{\alpha_5}(-1)e_{\alpha_4}(-1)^{2}e_{\alpha_3}(-1)h_1(-1)h_2(-1)\mathbf{1}
+12 e_{\alpha_5}(-1)e_{\alpha_4}(-1)^{2}e_{\alpha_3}(-1)h_2(-1)^{2}\mathbf{1}\\
&+8 e_{\alpha_5}(-1)e_{\alpha_4}(-1)^{3}e_{\alpha_2}(-1)e_{-\alpha_3}(-1)\mathbf{1}
+4 e_{\alpha_5}(-1)e_{\alpha_4}(-1)^{3}e_{\alpha_3}(-1)e_{-\alpha_4}(-1)\mathbf{1} 
-7 e_{\alpha_5}(-1)e_{\alpha_4}(-1)^{4}e_{-\alpha_5}(-1)\mathbf{1}\\
&-4 e_{\alpha_5}(-1)e_{\alpha_4}(-1)^{3}h_1(-1)e_{-\alpha_1}(-1)\mathbf{1}
-2 e_{\alpha_5}(-1)e_{\alpha_4}(-1)^{3}h_2(-1)e_{-\alpha_1}(-1)\mathbf{1}
-9 e_{\alpha_5}(-1)^{2}e_{\alpha_2}(-1)^{2}e_{\alpha_1}(-1)^{2}\mathbf{1}\\
&-30 e_{\alpha_5}(-1)^{2}e_{\alpha_3}(-1)e_{\alpha_2}(-1)e_{\alpha_1}(-1)h_1(-1)\mathbf{1}
-36 e_{\alpha_5}(-1)^{2}e_{\alpha_3}(-1)e_{\alpha_2}(-1)e_{\alpha_1}(-1)h_2(-1)\mathbf{1}\\
&-e_{\alpha_5}(-1)^{2}e_{\alpha_3}(-1)^{2}e_{\alpha_1}(-1)e_{-\alpha_1}(-1)\mathbf{1}
-9 e_{\alpha_5}(-1)^{2}e_{\alpha_3}(-1)^{2}e_{\alpha_2}(-1)e_{-\alpha_2}(-1)\mathbf{1}
-e_{\alpha_5}(-1)^{2}e_{\alpha_3}(-1)^{3}e_{-\alpha_3}(-1)\mathbf{1}\\
&-16 e_{\alpha_5}(-1)^{2}e_{\alpha_3}(-1)^{2}h_1(-1)^{2}\mathbf{1}
-33 e_{\alpha_5}(-1)^{2}e_{\alpha_3}(-1)^{2}h_1(-1)h_2(-1)\mathbf{1}
-18 e_{\alpha_5}(-1)^{2}e_{\alpha_3}(-1)^{2}h_2(-1)^{2}\mathbf{1}\\
&+3 e_{\alpha_5}(-1)^{2}e_{\alpha_4}(-1)e_{\alpha_2}(-1)e_{\alpha_1}(-1)e_{-\alpha_1}(-1)\mathbf{1}
-27 e_{\alpha_5}(-1)^{2}e_{\alpha_4}(-1)e_{\alpha_2}(-1)^{2}e_{-\alpha_2}(-1)\mathbf{1}\\
&+3 e_{\alpha_5}(-1)^{2}e_{\alpha_4}(-1)e_{\alpha_2}(-1)h_1(-1)^{2}\mathbf{1}
+9 e_{\alpha_5}(-1)^{2}e_{\alpha_4}(-1)e_{\alpha_2}(-1)h_1(-1)h_2(-1)\mathbf{1}\\
&-18 e_{\alpha_5}(-1)^{2}e_{\alpha_4}(-1)e_{\alpha_3}(-1)e_{\alpha_2}(-1)e_{-\alpha_3}(-1)\mathbf{1}
-3 e_{\alpha_5}(-1)^{2}e_{\alpha_4}(-1)e_{\alpha_3}(-1)^{2}e_{-\alpha_4}(-1)\mathbf{1}\\
&+13 e_{\alpha_5}(-1)^{2}e_{\alpha_4}(-1)e_{\alpha_3}(-1)h_1(-1)e_{-\alpha_1}(-1)\mathbf{1}
+9 e_{\alpha_5}(-1)^{2}e_{\alpha_4}(-1)e_{\alpha_3}(-1)h_2(-1)e_{-\alpha_1}(-1)\mathbf{1}
+5 e_{\alpha_5}(-1)^{2}e_{\alpha_4}(-1)^{2}e_{-\alpha_1}(-1)^{2}\mathbf{1}\\
&+18 e_{\alpha_5}(-1)^{2}e_{\alpha_4}(-1)^{2}e_{\alpha_2}(-1)e_{-\alpha_4}(-1)\mathbf{1}
+42 e_{\alpha_5}(-1)^{2}e_{\alpha_4}(-1)^{2}e_{\alpha_3}(-1)e_{-\alpha_5}(-1)\mathbf{1}
-27 e_{\alpha_5}(-1)^{3}e_{\alpha_2}(-1)^{2}e_{-\alpha_3}(-1)\mathbf{1}\\
&-9 e_{\alpha_5}(-1)^{3}e_{\alpha_2}(-1)h_1(-1)e_{-\alpha_1}(-1)\mathbf{1}
-27 e_{\alpha_5}(-1)^{3}e_{\alpha_2}(-1)h_2(-1)e_{-\alpha_1}(-1)\mathbf{1}
-15 e_{\alpha_5}(-1)^{3}e_{\alpha_3}(-1)e_{-\alpha_1}(-1)^{2}\mathbf{1}\\
&-54 e_{\alpha_5}(-1)^{3}e_{\alpha_3}(-1)e_{\alpha_2}(-1)e_{-\alpha_4}(-1)\mathbf{1}
-63 e_{\alpha_5}(-1)^{3}e_{\alpha_3}(-1)^{2}e_{-\alpha_5}(-1)\mathbf{1}
-2 e_{\theta}(-1)e_{\alpha_4}(-1)e_{\alpha_2}(-1)e_{\alpha_1}(-1)^{3}\mathbf{1}\\
&+2 e_{\theta}(-1)e_{\alpha_4}(-1)e_{\alpha_3}(-1)e_{\alpha_1}(-1)^{2}h_2(-1)\mathbf{1}
+2 e_{\theta}(-1)e_{\alpha_4}(-1)e_{\alpha_3}(-1)^{2}e_{\alpha_1}(-1)e_{-\alpha_2}(-1)\mathbf{1}
+3 e_{\theta}(-1)e_{\alpha_4}(-1)^{2}e_{\alpha_1}(-1)^{2}e_{-\alpha_1}(-1)\mathbf{1}\\
&+8 e_{\theta}(-1)e_{\alpha_4}(-1)^{2}e_{\alpha_1}(-1)h_1(-1)^{2}\mathbf{1}
+29 e_{\theta}(-1)e_{\alpha_4}(-1)^{2}e_{\alpha_1}(-1)h_1(-1)h_2(-1)\mathbf{1}
+27 e_{\theta}(-1)e_{\alpha_4}(-1)^{2}e_{\alpha_1}(-1)h_2(-1)^{2}\mathbf{1}\\
&+6 e_{\theta}(-1)e_{\alpha_4}(-1)^{2}e_{\alpha_2}(-1)e_{\alpha_1}(-1)e_{-\alpha_2}(-1)\mathbf{1}
+3 e_{\theta}(-1)e_{\alpha_4}(-1)^{2}e_{\alpha_3}(-1)e_{\alpha_1}(-1)e_{-\alpha_3}(-1)\mathbf{1}\\
&+10 e_{\theta}(-1)e_{\alpha_4}(-1)^{2}e_{\alpha_3}(-1)h_1(-1)e_{-\alpha_2}(-1)\mathbf{1}
+15 e_{\theta}(-1)e_{\alpha_4}(-1)^{2}e_{\alpha_3}(-1)h_2(-1)e_{-\alpha_2}(-1)\mathbf{1}
+4 e_{\theta}(-1)e_{\alpha_4}(-1)^{3}e_{\alpha_1}(-1)e_{-\alpha_4}(-1)\mathbf{1}\\
&-8 e_{\theta}(-1)e_{\alpha_4}(-1)^{3}e_{-\alpha_2}(-1)e_{-\alpha_1}(-1)\mathbf{1}
-7 e_{\theta}(-1)e_{\alpha_4}(-1)^{4}e_{-\theta}(-1)\mathbf{1}
+4 e_{\theta}(-1)e_{\alpha_4}(-1)^{3}h_1(-1)e_{-\alpha_3}(-1)\mathbf{1}\\
&+10 e_{\theta}(-1)e_{\alpha_4}(-1)^{3}h_2(-1)e_{-\alpha_3}(-1)\mathbf{1}
-30 e_{\theta}(-1)e_{\alpha_5}(-1)e_{\alpha_2}(-1)e_{\alpha_1}(-1)^{2}h_1(-1)\mathbf{1}
-54 e_{\theta}(-1)e_{\alpha_5}(-1)e_{\alpha_2}(-1)e_{\alpha_1}(-1)^{2}h_2(-1)\mathbf{1}\\
&-2 e_{\theta}(-1)e_{\alpha_5}(-1)e_{\alpha_3}(-1)e_{\alpha_1}(-1)^{2}e_{-\alpha_1}(-1)\mathbf{1}
-32 e_{\theta}(-1)e_{\alpha_5}(-1)e_{\alpha_3}(-1)e_{\alpha_1}(-1)h_1(-1)^{2}\mathbf{1}\\
&-96 e_{\theta}(-1)e_{\alpha_5}(-1)e_{\alpha_3}(-1)e_{\alpha_1}(-1)h_1(-1)h_2(-1)\mathbf{1}
-72 e_{\theta}(-1)e_{\alpha_5}(-1)e_{\alpha_3}(-1)e_{\alpha_1}(-1)h_2(-1)^{2}\mathbf{1}\\
&-36 e_{\theta}(-1)e_{\alpha_5}(-1)e_{\alpha_3}(-1)e_{\alpha_2}(-1)e_{\alpha_1}(-1)e_{-\alpha_2}(-1)\mathbf{1}
-2 e_{\theta}(-1)e_{\alpha_5}(-1)e_{\alpha_3}(-1)^{2}e_{\alpha_1}(-1)e_{-\alpha_3}(-1)\mathbf{1}\\
&-30 e_{\theta}(-1)e_{\alpha_5}(-1)e_{\alpha_3}(-1)^{2}h_1(-1)e_{-\alpha_2}(-1)\mathbf{1}
-36 e_{\theta}(-1)e_{\alpha_5}(-1)e_{\alpha_3}(-1)^{2}h_2(-1)e_{-\alpha_2}(-1)\mathbf{1}\\
&+13 e_{\theta}(-1)e_{\alpha_5}(-1)e_{\alpha_4}(-1)e_{\alpha_1}(-1)h_1(-1)e_{-\alpha_1}(-1)\mathbf{1}
+12 e_{\theta}(-1)e_{\alpha_5}(-1)e_{\alpha_4}(-1)e_{\alpha_1}(-1)h_2(-1)e_{-\alpha_1}(-1)\mathbf{1}\\
&-21 e_{\theta}(-1)e_{\alpha_5}(-1)e_{\alpha_4}(-1)e_{\alpha_2}(-1)e_{\alpha_1}(-1)e_{-\alpha_3}(-1)\mathbf{1}
-27 e_{\theta}(-1)e_{\alpha_5}(-1)e_{\alpha_4}(-1)e_{\alpha_2}(-1)h_2(-1)e_{-\alpha_2}(-1)\mathbf{1}\\
&-6 e_{\theta}(-1)e_{\alpha_5}(-1)e_{\alpha_4}(-1)e_{\alpha_3}(-1)e_{\alpha_1}(-1)e_{-\alpha_4}(-1)\mathbf{1} 
+21 e_{\theta}(-1)e_{\alpha_5}(-1)e_{\alpha_4}(-1)e_{\alpha_3}(-1)e_{-\alpha_2}(-1)e_{-\alpha_1}(-1)\mathbf{1}\\
&-13 e_{\theta}(-1)e_{\alpha_5}(-1)e_{\alpha_4}(-1)e_{\alpha_3}(-1)h_1(-1)e_{-\alpha_3}(-1)\mathbf{1}
-27 e_{\theta}(-1)e_{\alpha_5}(-1)e_{\alpha_4}(-1)e_{\alpha_3}(-1)h_2(-1)e_{-\alpha_3}(-1)\mathbf{1}\\
&+42 e_{\theta}(-1)e_{\alpha_5}(-1)e_{\alpha_4}(-1)^{2}e_{\alpha_1}(-1)e_{-\alpha_5}(-1)\mathbf{1}
-10 e_{\theta}(-1)e_{\alpha_5}(-1)e_{\alpha_4}(-1)^{2}e_{-\alpha_3}(-1)e_{-\alpha_1}(-1)\mathbf{1}\\
&+42 e_{\theta}(-1)e_{\alpha_5}(-1)e_{\alpha_4}(-1)^{2}e_{\alpha_3}(-1)e_{-\theta}(-1)\mathbf{1}
+18 e_{\theta}(-1)e_{\alpha_5}(-1)e_{\alpha_4}(-1)^{2}h_2(-1)e_{-\alpha_4}(-1)\mathbf{1}\\
&+3 e_{\theta}(-1)e_{\alpha_5}(-1)e_{\alpha_4}(-1)h_1(-1)^{2}h_2(-1)\mathbf{1}
+9 e_{\theta}(-1)e_{\alpha_5}(-1)e_{\alpha_4}(-1)h_1(-1)h_2(-1)^{2}\mathbf{1}
-15 e_{\theta}(-1)e_{\alpha_5}(-1)^{2}e_{\alpha_1}(-1)e_{-\alpha_1}(-1)^{2}\mathbf{1}\\
&-54 e_{\theta}(-1)e_{\alpha_5}(-1)^{2}e_{\alpha_2}(-1)e_{\alpha_1}(-1)e_{-\alpha_4}(-1)\mathbf{1}
+27 e_{\theta}(-1)e_{\alpha_5}(-1)^{2}e_{\alpha_2}(-1)e_{-\alpha_2}(-1)e_{-\alpha_1}(-1)\mathbf{1}\\
&+9 e_{\theta}(-1)e_{\alpha_5}(-1)^{2}e_{\alpha_2}(-1)h_1(-1)e_{-\alpha_3}(-1)\mathbf{1}
-27 e_{\theta}(-1)e_{\alpha_5}(-1)^{2}e_{\alpha_2}(-1)h_2(-1)e_{-\alpha_3}(-1)\mathbf{1}\\
&-126 e_{\theta}(-1)e_{\alpha_5}(-1)^{2}e_{\alpha_3}(-1)e_{\alpha_1}(-1)e_{-\alpha_5}(-1)\mathbf{1}
+30 e_{\theta}(-1)e_{\alpha_5}(-1)^{2}e_{\alpha_3}(-1)e_{-\alpha_3}(-1)e_{-\alpha_1}(-1)\mathbf{1}\\
&-63 e_{\theta}(-1)e_{\alpha_5}(-1)^{2}e_{\alpha_3}(-1)^{2}e_{-\theta}(-1)\mathbf{1}
-54 e_{\theta}(-1)e_{\alpha_5}(-1)^{2}e_{\alpha_3}(-1)h_2(-1)e_{-\alpha_4}(-1)\mathbf{1}\\
&-9 e_{\theta}(-1)e_{\alpha_5}(-1)^{2}h_1(-1)h_2(-1)e_{-\alpha_1}(-1)\mathbf{1}
-27 e_{\theta}(-1)e_{\alpha_5}(-1)^{2}h_2(-1)^{2}e_{-\alpha_1}(-1)\mathbf{1}\\ 
&-e_{\theta}(-1)^{2}e_{\alpha_1}(-1)^{3}e_{-\alpha_1}(-1)\mathbf{1}
-16 e_{\theta}(-1)^{2}e_{\alpha_1}(-1)^{2}h_1(-1)^{2}\mathbf{1}
-63 e_{\theta}(-1)^{2}e_{\alpha_1}(-1)^{2}h_1(-1)h_2(-1)\mathbf{1}\\
&-63 e_{\theta}(-1)^{2}e_{\alpha_1}(-1)^{2}h_2(-1)^{2}\mathbf{1}
-9 e_{\theta}(-1)^{2}e_{\alpha_2}(-1)e_{\alpha_1}(-1)^{2}e_{-\alpha_2}(-1)\mathbf{1}
-e_{\theta}(-1)^{2}e_{\alpha_3}(-1)e_{\alpha_1}(-1)^{2}e_{-\alpha_3}(-1)\mathbf{1}\\
&-30 e_{\theta}(-1)^{2}e_{\alpha_3}(-1)e_{\alpha_1}(-1)h_1(-1)e_{-\alpha_2}(-1)\mathbf{1}
-54 e_{\theta}(-1)^{2}e_{\alpha_3}(-1)e_{\alpha_1}(-1)h_2(-1)e_{-\alpha_2}(-1)\mathbf{1}
-9 e_{\theta}(-1)^{2}e_{\alpha_3}(-1)^{2}e_{-\alpha_2}(-1)^{2}\mathbf{1}\\
&-3 e_{\theta}(-1)^{2}e_{\alpha_4}(-1)e_{\alpha_1}(-1)^{2}e_{-\alpha_4}(-1)\mathbf{1}
+18 e_{\theta}(-1)^{2}e_{\alpha_4}(-1)e_{\alpha_1}(-1)e_{-\alpha_2}(-1)e_{-\alpha_1}(-1)\mathbf{1}\\
&-13 e_{\theta}(-1)^{2}e_{\alpha_4}(-1)e_{\alpha_1}(-1)h_1(-1)e_{-\alpha_3}(-1)\mathbf{1}
-30 e_{\theta}(-1)^{2}e_{\alpha_4}(-1)e_{\alpha_1}(-1)h_2(-1)e_{-\alpha_3}(-1)\mathbf{1}\\
&+27 e_{\theta}(-1)^{2}e_{\alpha_4}(-1)e_{\alpha_2}(-1)e_{-\alpha_2}(-1)^{2}\mathbf{1}
-3 e_{\theta}(-1)^{2}e_{\alpha_4}(-1)e_{\alpha_3}(-1)e_{-\alpha_3}(-1)e_{-\alpha_2}(-1)\mathbf{1}
+42 e_{\theta}(-1)^{2}e_{\alpha_4}(-1)^{2}e_{\alpha_1}(-1)e_{-\theta}(-1)\mathbf{1}\\
&-18 e_{\theta}(-1)^{2}e_{\alpha_4}(-1)^{2}e_{-\alpha_4}(-1)e_{-\alpha_2}(-1)\mathbf{1}
+5 e_{\theta}(-1)^{2}e_{\alpha_4}(-1)^{2}e_{-\alpha_3}(-1)^{2}\mathbf{1}
-3 e_{\theta}(-1)^{2}e_{\alpha_4}(-1)h_1(-1)^{2}e_{-\alpha_2}(-1)\mathbf{1}\\
&-9 e_{\theta}(-1)^{2}e_{\alpha_4}(-1)h_1(-1)h_2(-1)e_{-\alpha_2}(-1)\mathbf{1} 
-63 e_{\theta}(-1)^{2}e_{\alpha_5}(-1)e_{\alpha_1}(-1)^{2}e_{-\alpha_5}(-1)\\
&+30 e_{\theta}(-1)^{2}e_{\alpha_5}(-1)e_{\alpha_1}(-1)e_{-\alpha_3}(-1)e_{-\alpha_1}(-1)\mathbf{1}
-54 e_{\theta}(-1)^{2}e_{\alpha_5}(-1)e_{\alpha_1}(-1)h_2(-1)e_{-\alpha_4}(-1)\mathbf{1}\\
&+27 e_{\theta}(-1)^{2}e_{\alpha_5}(-1)e_{\alpha_2}(-1)e_{-\alpha_3}(-1)e_{-\alpha_2}(-1)\mathbf{1}
-126 e_{\theta}(-1)^{2}e_{\alpha_5}(-1)e_{\alpha_3}(-1)e_{\alpha_1}(-1)e_{-\theta}(-1)\mathbf{1}\\
&+54 e_{\theta}(-1)^{2}e_{\alpha_5}(-1)e_{\alpha_3}(-1)e_{-\alpha_4}(-1)e_{-\alpha_2}(-1)\mathbf{1}
-15 e_{\theta}(-1)^{2}e_{\alpha_5}(-1)e_{\alpha_3}(-1)e_{-\alpha_3}(-1)^{2}\mathbf{1}\\
&+9 e_{\theta}(-1)^{2}e_{\theta}(-1)h_1(-1)e_{-\alpha_2}(-1)e_{-\alpha_1}(-1)\mathbf{1}
+9 e_{\theta}(-1)^{2}e_{\theta}(-1)h_1(-1)h_2(-1)e_{-\alpha_3}(-1)\mathbf{1}\\
&+54 e_{\theta}(-1)^{2}e_{\theta}(-1)h_2(-1)e_{-\alpha_2}(-1)e_{-\alpha_1}(-1)\mathbf{1} 
-63 e_{\theta}(-1)^{3}e_{\alpha_1}(-1)^{2}e_{-\theta}(-1)\mathbf{1}+54 e_{\theta}(-1)^{3}
e_{\alpha_1}(-1)e_{-\alpha_4}(-1)e_{-\alpha_2}(-1)\mathbf{1}\\
&-15 e_{\theta}(-1)^{3}e_{\alpha_1}(-1)e_{-\alpha_3}(-1)^{2}\mathbf{1}
-27 e_{\theta}(-1)^{3}e_{-\alpha_2}(-1)^{2}e_{-\alpha_1}(-1)\mathbf{1}
-9 e_{\theta}(-1)^{3}h_1(-1)e_{-\alpha_3}(-1)e_{-\alpha_2}(-1)\mathbf{1} 
\end{align*}}

\medskip

{\tiny 
\begin{align*}
&v'_{\sing}
=72e_{\alpha_4}e_{\alpha_5}e_{\theta}-96 e_{\alpha_1}e_{\alpha_1}e_{\theta}e_{\theta}
-72 e_{\alpha_1}e_{\alpha_3}e_{\alpha_5}e_{\theta}
+36 e_{\alpha_1}e_{\alpha_4}e_{\alpha_4}e_{\theta}
+6 e_{-\alpha_3}e_{\alpha_5}e_{\theta}e_{\theta}
+72 e_{-\alpha_2}e_{\alpha_4}e_{\theta}e_{\theta}
-36 e_{\alpha_3}e_{\alpha_3}e_{\alpha_5}e_{\alpha_5}\\
&+24 e_{\alpha_3}e_{\alpha_4}e_{\alpha_4}e_{\alpha_5}
-4 e_{\alpha_4}e_{\alpha_4}e_{\alpha_4}e_{\alpha_4}
+100 h_1e_{\alpha_4}e_{\alpha_5}e_{\theta}
+10 e_{-\alpha_1}e_{\alpha_1}e_{\alpha_4}e_{\alpha_5}e_{\theta}
+18 e_{\alpha_1}e_{-\alpha_2}e_{\alpha_5}e_{\theta}e_{\theta}
-45 e_{-\alpha_1}e_{\alpha_2}e_{\alpha_5}e_{\alpha_5}e_{\alpha_5}\\
&+5 e_{-\alpha_1}e_{\alpha_3}e_{\alpha_4}e_{\alpha_5}e_{\alpha_5}
-60 e_{\alpha_1}e_{\alpha_1}e_{\alpha_2}e_{\alpha_5}e_{\theta}
+6 e_{\alpha_1}e_{\alpha_1}e_{\alpha_3}e_{\alpha_4}e_{\theta}
-42e_{\alpha_1}e_{\alpha_2}e_{\alpha_3}e_{\alpha_5}e_{\alpha_5}
+20 e_{\alpha_1}e_{\alpha_2}e_{\alpha_4}e_{\alpha_4}e_{\alpha_5}\\
&+108 h_2e_{\alpha_4}e_{\alpha_5}e_{\theta}
-45 e_{-\alpha_1}h_2e_{\alpha_5}e_{\alpha_5}e_{\theta}
+4 e_{\alpha_1}e_{\alpha_3}e_{\alpha_3}e_{\alpha_4}e_{\alpha_5}
-2 e_{\alpha_1}e_{\alpha_3}e_{\alpha_4}e_{\alpha_4}e_{\alpha_4}
+126 e_{-\theta}e_{\alpha_4}e_{\alpha_5}e_{\theta}e_{\theta}\\
&+126 e_{-\alpha_5}e_{\alpha_4}e_{\alpha_5}e_{\alpha_5}e_{\theta}
+30 e_{-\alpha_4}e_{\alpha_1}e_{\alpha_5}e_{\theta}e_{\theta}
+84 e_{-\alpha_4}e_{\alpha_3}e_{\alpha_5}e_{\alpha_5}e_{\theta}
-4 e_{-\alpha_4}e_{\alpha_4}e_{\alpha_4}e_{\alpha_5}e_{\theta}
-4 e_{-\alpha_3}e_{\alpha_1}e_{\alpha_4}e_{\theta}e_{\theta}
+99 e_{-\alpha_3}e_{\alpha_2}e_{\alpha_5}e_{\alpha_5}e_{\theta}\\
&+19 e_{-\alpha_3}e_{\alpha_3}e_{\alpha_4}e_{\alpha_5}e_{\theta}
-6 e_{-\alpha_3}e_{\alpha_4}e_{\alpha_4}e_{\alpha_4}e_{\theta}
-42 e_{-\alpha_2}e_{\alpha_1}e_{\alpha_3}e_{\theta}e_{\theta}
-72 e_{-\alpha_2}.e_{-\alpha_3}e_{\theta}e_{\theta}e_{\theta}
+90 e_{-\alpha_2}e_{\alpha_2}e_{\alpha_4}e_{\alpha_5}e_{\theta}
-42 e_{-\alpha_2}e_{\alpha_3}e_{\alpha_3}e_{\alpha_5}e_{\theta}\\
&+6e_{-\alpha_2}e_{\alpha_3}e_{\alpha_4}e_{\alpha_4}e_{\theta}
+72 e_{-\alpha_3}h_2e_{\alpha_5}e_{\theta}e_{\theta}
+45 e_{-\alpha_2}h_1e_{\alpha_4}e_{\theta}e_{\theta}
+90 e_{-\alpha_2}h_2e_{\alpha_4}e_{\theta}e_{\theta}
-80 h_1e_{\alpha_1}e_{\alpha_1}e_{\theta}e_{\theta}
-15 e_{-\alpha_1}e_{-\alpha_1}e_{\alpha_1}e_{\alpha_5}e_{\alpha_5}e_{\theta}\\
&-15e_{-\alpha_1}e_{-\alpha_1}e_{\alpha_3}e_{\alpha_5}e_{\alpha_5}e_{\alpha_5}
+5 e_{-\alpha_1}e_{-\alpha_1}e_{\alpha_4}e_{\alpha_4}e_{\alpha_5}e_{\alpha_5}
-e_{-\alpha_1}e_{\alpha_1}e_{\alpha_1}e_{\alpha_1}e_{\theta}e_{\theta}
-2 e_{\alpha_1}e_{\alpha_1}e_{\alpha_1}e_{\alpha_3}e_{\alpha_5}e_{\theta}
-100 h_1e_{\alpha_1}e_{\alpha_3}e_{\alpha_5}e_{\theta}\\
&+40 h_1e_{\alpha_1}e_{\alpha_4}e_{\alpha_4}e_{\theta}
+15 h_1e_{\alpha_2}e_{\alpha_4}e_{\alpha_5}e_{\alpha_5}
-50 h_1e_{\alpha_3}e_{\alpha_3}e_{\alpha_5}e_{\alpha_5}
+30 h_1e_{\alpha_3}e_{\alpha_4}e_{\alpha_4}e_{\alpha_5}
-5 h_1e_{\alpha_4}e_{\alpha_4}e_{\alpha_4}e_{\alpha_4}
+32 h_1h_1e_{\alpha_4}e_{\alpha_5}e_{\theta}\\
&-159 h_2e_{\alpha_1}e_{\alpha_1}e_{\theta}e_{\theta}
-150 h_2e_{\alpha_1}e_{\alpha_3}e_{\alpha_5}e_{\theta}
+67 h_2e_{\alpha_1}e_{\alpha_4}e_{\alpha_4}e_{\theta}
-54 h_2e_{\alpha_3}e_{\alpha_3}e_{\alpha_5}e_{\alpha_5}
+36 h_2e_{\alpha_3}e_{\alpha_4}e_{\alpha_4}e_{\alpha_5}
-6h_2e_{\alpha_4}e_{\alpha_4}e_{\alpha_4}e_{\alpha_4}\\
&+81 h_2h_1e_{\alpha_4}e_{\alpha_5}e_{\theta}
+36 h_2h_2e_{\alpha_4}e_{\alpha_5}e_{\theta}
+3 e_{-\alpha_1}.e_{\alpha_1}e_{\alpha_1}e_{\alpha_4}e_{\alpha_4}e_{\theta}
+3 e_{\alpha_1}e_{\alpha_1}e_{\alpha_2}e_{\alpha_4}e_{\alpha_5}e_{\alpha_5}
-e_{-\alpha_1}e_{\alpha_1}e_{\alpha_3}e_{\alpha_3}e_{\alpha_5}e_{\alpha_5}\\
&+3 e_{-\alpha_1}e_{\alpha_1}e_{\alpha_3}e_{\alpha_4}e_{\alpha_4}e_{\alpha_5}
-e_{-\alpha_1}e_{\alpha_1}e_{\alpha_4}e_{\alpha_4}e_{\alpha_4}e_{\alpha_4}
+30 e_{-\alpha_1}e_{-\alpha_3}e_{\alpha_1}e_{\alpha_5}e_{\theta}e_{\theta}
+30 e_{-\alpha_1}e_{-\alpha_3}e_{\alpha_3}e_{\alpha_5}e_{\alpha_5}e_{\theta}
-10 e_{\alpha_1}e_{-\alpha_3}e_{\alpha_4}e_{\alpha_4}e_{\alpha_5}e_{\theta}\\
&+18e_{-\alpha_1}e_{-\alpha_2}e_{\alpha_1}e_{\alpha_4}e_{\theta}e_{\theta}
-27e_{-\alpha_1}e_{-\alpha_2}e_{-\alpha_2}e_{\theta}e_{\theta}e_{\theta}
+27e_{-\alpha_1}e_{-\alpha_2}e_{\alpha_2}e_{\alpha_5}e_{\alpha_5}e_{\theta}
+21e_{-\alpha_1}e_{-\alpha_2}e_{\alpha_3}e_{\alpha_4}e_{\alpha_5}e_{\theta}
-8e_{-\alpha_1}e_{-\alpha_2}e_{\alpha_4}e_{\alpha_4}e_{\alpha_4}e_{\theta}\\
&+9e_{-\alpha_1}e_{-\alpha_2}h_1e_{\alpha_5}e_{\theta}e_{\theta}
+54 e_{\alpha_1}e_{-\alpha_2}h_2e_{\alpha_5}e_{\theta}e_{\theta}
+13 e_{\alpha_1}h_1e_{\alpha_1}e_{\alpha_4}e_{\alpha_5}e_{\theta}
-9 e_{-\alpha_1}h_1e_{\alpha_2}e_{\alpha_5}e_{\alpha_5}e_{\alpha_5}
+13 e_{-\alpha_1}h_1e_{\alpha_3}e_{\alpha_4}e_{\alpha_5}e_{\alpha_5}\\
&-4 e_{-\alpha_1}h_1e_{\alpha_4}e_{\alpha_4}e_{\alpha_4}e_{\alpha_5}
+12 e_{-\alpha_1}h_2e_{\alpha_1}e_{\alpha_4}e_{\alpha_5}e_{\theta}
-27 e_{-\alpha_1}h_2.e_{\alpha_2}e_{\alpha_5}e_{\alpha_5}e_{\alpha_5}
+9 e_{\alpha_1}h_2e_{\alpha_3}e_{\alpha_4}e_{\alpha_5}e_{\alpha_5}
-2 e_{-\alpha_1}h_2e_{\alpha_4}e_{\alpha_4}e_{\alpha_4}e_{\alpha_5}\\
&-9 e_{-\alpha_1}h_2h_1e_{\alpha_5}e_{\alpha_5}e_{\theta}
-27 e_{-\alpha_1}h_2h_2e_{\alpha_5}e_{\alpha_5}e_{\theta}
-2 e_{\alpha_1}e_{\alpha_1}e_{\alpha_1}e_{\alpha_2}e_{\alpha_4}e_{\theta}
-9 e_{\alpha_1}e_{\alpha_1}e_{\alpha_2}e_{\alpha_2}e_{\alpha_5}e_{\alpha_5}
-2 e_{\alpha_1}e_{\alpha_1}e_{\alpha_2}e_{\alpha_3}e_{\alpha_4}e_{\alpha_5}\\
&+e_{\alpha_1}e_{\alpha_1}e_{\alpha_2}e_{\alpha_4}e_{\alpha_4}e_{\alpha_4}
-63 e_{-\theta}e_{\alpha_1}e_{\alpha_1}e_{\theta}e_{\theta}e_{\theta}
-126 e_{-\theta}e_{\alpha_1}e_{\alpha_3}e_{\alpha_5}e_{\theta}e_{\theta}
+42 e_{-\theta}e_{\alpha_1}e_{\alpha_4}e_{\alpha_4}e_{\theta}e_{\theta}
-63 e_{-\theta}e_{\alpha_3}.{\alpha_3}e_{\alpha_5}e_{\alpha_5}e_{\theta}\\
&+42 e_{-\theta}e_{\alpha_3}e_{\alpha_4}e_{\alpha_4}e_{\alpha_5}e_{\theta}
-7 e_{-\theta}e_{\alpha_4}e_{\alpha_4}e_{\alpha_4}e_{\alpha_4}e_{\theta}
-63 e_{-\alpha_5}e_{\alpha_1}e_{\alpha_1}e_{\alpha_5}e_{\theta}e_{\theta}
-126 e_{-\alpha_5}e_{\alpha_1}e_{\alpha_3}e_{\alpha_5}e_{\alpha_5}e_{\theta}
+42 e_{-\alpha_5}e_{\alpha_1}e_{\alpha_4}e_{\alpha_4}e_{\alpha_5}e_{\theta}\\
&-63 e_{-\alpha_5}e_{\alpha_3}e_{\alpha_3}e_{\alpha_5}e_{\alpha_5}e_{\alpha_5}
+42 e_{-\alpha_5}e_{\alpha_3}e_{\alpha_4}e_{\alpha_4}e_{\alpha_5}e_{\alpha_5}
-7 e_{-\alpha_5}e_{\alpha_4}e_{\alpha_4}e_{\alpha_4}e_{\alpha_4}e_{\alpha_5}
-3 e_{-\alpha_4}e_{\alpha_1}e_{\alpha_1}e_{\alpha_4}e_{\theta}e_{\theta}
-54 e_{-\alpha_4}e_{\alpha_1}e_{\alpha_2}e_{\alpha_5}e_{\alpha_5}e_{\theta}\\
&-6 e_{-\alpha_4}e_{\alpha_1}e_{\alpha_3}e_{\alpha_4}e_{\alpha_5}e_{\theta}
+4 e_{-\alpha_4}e_{\alpha_1}e_{\alpha_4}e_{\alpha_4}e_{\alpha_4}e_{\theta}
-54 e_{-\alpha_4}e_{\alpha_2}e_{\alpha_3}e_{\alpha_5}e_{\alpha_5}e_{\alpha_5}
+18 e_{-\alpha_4}e_{\alpha_2}e_{\alpha_4}e_{\alpha_4}e_{\alpha_5}e_{\alpha_5}
-3 e_{-\alpha_4}e_{\alpha_3}e_{\alpha_3}e_{\alpha_4}e_{\alpha_5}e_{\alpha_5}\\
&+4 e_{-\alpha_4}e_{\alpha_3}e_{\alpha_4}e_{\alpha_4}e_{\alpha_4}e_{\alpha_5}
-e_{-\alpha_4}e_{\alpha_4}e_{\alpha_4}e_{\alpha_4}e_{\alpha_4}e_{\alpha_4}
-54 e_{-\alpha_4}h_2.e_{\alpha_1}e_{\alpha_5}e_{\theta}e_{\theta}
-54 e_{-\alpha_4}h_2e_{\alpha_3}e_{\alpha_5}e_{\alpha_5}e_{\theta}
+18 e_{-\alpha_4}h_2e_{\alpha_4}e_{\alpha_4}e_{\alpha_5}e_{\theta}\\
&-e_{-\alpha_3}e_{\alpha_1}e_{\alpha_1}e_{\alpha_3}e_{\theta}e_{\theta}
-21 e_{-\alpha_3}e_{\alpha_1}e_{\alpha_2}e_{\alpha_4}e_{\alpha_5}e_{\theta}
-2 e_{-\alpha_3}e_{\alpha_1}e_{\alpha_3}e_{\alpha_3}e_{\alpha_5}e_{\theta}
+3e_{-\alpha_3}e_{\alpha_1}e_{\alpha_3}e_{\alpha_4}e_{\alpha_4}e_{\theta}
-15 e_{-\alpha_3}e_{-\alpha_3}e_{\alpha_1}e_{\theta}e_{\theta}e_{\theta}\\
&-15 e_{-\alpha_3}e_{-\alpha_3}e_{\alpha_3}e_{\alpha_5}e_{\theta}e_{\theta}
+5 e_{-\alpha_3}e_{-\alpha_3}e_{\alpha_4}e_{\alpha_4}e_{\theta}e_{\theta}
-27 e_{-\alpha_3}e_{\alpha_2}e_{\alpha_2}e_{\alpha_5}e_{\alpha_5}e_{\alpha_5}
-18e_{-\alpha_3}e_{\alpha_2}e_{\alpha_3}e_{\alpha_4}e_{\alpha_5}e_{\alpha_5}
+8 e_{-\alpha_3}e_{\alpha_2}e_{\alpha_4}e_{\alpha_4}e_{\alpha_4}e_{\alpha_5}\\
&-e_{-\alpha_3}e_{\alpha_3}e_{\alpha_3}e_{\alpha_3}e_{\alpha_5}e_{\alpha_5}
+3 e_{-\alpha_3}e_{\alpha_3}e_{\alpha_3}e_{\alpha_4}e_{\alpha_4}e_{\alpha_5}
-e_{-\alpha_3}e_{\alpha_3}e_{\alpha_4}e_{\alpha_4}e_{\alpha_4}e_{\alpha_4}
-13 e_{-\alpha_3}h_1 e_{\alpha_1}e_{\alpha_4}e_{\theta}e_{\theta}
+9 e_{-\alpha_3}h_1e_{\alpha_2}e_{\alpha_5}e_{\alpha_5}e_{\theta}\\
&-13 e_{-\alpha_3}h_1.e_{\alpha_3}e_{\alpha_4}e_{\alpha_5}e_{\theta}
+4 e_{-\alpha_3}h_1e_{\alpha_4}e_{\alpha_4}e_{\alpha_4}e_{\theta}
-30 e_{-\alpha_3}h_2e_{\alpha_1}e_{\alpha_4}e_{\theta}e_{\theta}
-27 e_{-\alpha_3}h_2e_{\alpha_2}e_{\alpha_5}e_{\alpha_5}e_{\theta}
-27 e_{-\alpha_3}h_2e_{\alpha_3}e_{\alpha_4}e_{\alpha_5}e_{\theta}\\
&+10 e_{-\alpha_3}h_2e_{\alpha_4}e_{\alpha_4}e_{\alpha_4}e_{\theta}
+9 e_{-\alpha_3}h_2h_1e_{\alpha_5}e_{\theta}e_{\theta}
-9 e_{-\alpha_2}e_{\alpha_1}e_{\alpha_1}e_{\alpha_2}e_{\theta}e_{\theta}
-36 e_{-\alpha_2}e_{\alpha_1}e_{\alpha_2}e_{\alpha_3}e_{\alpha_5}e_{\theta}
+6 e_{-\alpha_2}e_{\alpha_1}e_{\alpha_2}e_{\alpha_4}e_{\alpha_4}e_{\theta}\\
&+2 e_{-\alpha_2}e_{\alpha_1}e_{\alpha_3}e_{\alpha_3}e_{\alpha_4}e_{\theta}
+54 e_{-\alpha_2}e_{-\alpha_4}e_{\alpha_1}e_{\theta}e_{\theta}e_{\theta}
+54 e_{-\alpha_2}e_{-\alpha_4}e_{\alpha_3}e_{\alpha_5}e_{\theta}e_{\theta}
-18 e_{-\alpha_2}e_{-\alpha_4}e_{\alpha_4}e_{\alpha_4}e_{\theta}e_{\theta}
+27 e_{-\alpha_2}e_{-\alpha_3}e_{\alpha_2}e_{\alpha_5}e_{\theta}e_{\theta}\\
&-3 e_{-\alpha_2}e_{-\alpha_3}e_{\alpha_3}e_{\alpha_4}e_{\theta}e_{\theta}
-9 e_{-\alpha_2}e_{-\alpha_3}h_1e_{\theta}e_{\theta}e_{\theta}
+27 e_{-\alpha_2}e_{-\alpha_2}e_{\alpha_2}e_{\alpha_4}e_{\theta}e_{\theta}
-9 e_{-\alpha_2}e_{-\alpha_2}e_{\alpha_3}e_{\alpha_3}e_{\theta}e_{\theta}
-27 e_{-\alpha_2}e_{\alpha_2}e_{\alpha_2}e_{\alpha_4}e_{\alpha_5}e_{\alpha_5}\\
&-9 e_{-\alpha_2}e_{\alpha_2}e_{\alpha_3}e_{\alpha_3}e_{\alpha_5}e_{\alpha_5}
+6 e_{-\alpha_2}e_{\alpha_2}e_{\alpha_3}e_{\alpha_4}e_{\alpha_4}e_{\alpha_5}
+e_{-\alpha_2}e_{\alpha_2}e_{\alpha_4}e_{\alpha_4}e_{\alpha_4}e_{\alpha_4}
+2e_{-\alpha_2}e_{\alpha_3}e_{\alpha_3}e_{\alpha_3}e_{\alpha_4}e_{\alpha_5}
-e_{-\alpha_2}e_{\alpha_3}e_{\alpha_3}e_{\alpha_4}e_{\alpha_4}e_{\alpha_4}\\
&-30 e_{-\alpha_2}h_1e_{\alpha_1}e_{\alpha_3}e_{\theta}e_{\theta}
-30 e_{-\alpha_2}h_1e_{\alpha_3}e_{\alpha_3}e_{\alpha_5}e_{\theta}
+10 e_{-\alpha_2}h_1e_{\alpha_3}e_{\alpha_4}e_{\alpha_4}e_{\theta}
-3 e_{-\alpha_2}h_1h_1e_{\alpha_4}e_{\theta}e_{\theta}
-54 e_{-\alpha_2}h_2e_{\alpha_1}e_{\alpha_3}e_{\theta}e_{\theta}\\
&-27 e_{-\alpha_2}h_2e_{\alpha_2}e_{\alpha_4}e_{\alpha_5}e_{\theta}
-36 e_{-\alpha_2}h_2e_{\alpha_3}e_{\alpha_3}e_{\alpha_5}e_{\theta}
+15 e_{-\alpha_2}h_2e_{\alpha_3}e_{\alpha_4}e_{\alpha_4}e_{\theta}
-9 e_{-\alpha_2}h_2h_1e_{\alpha_4}e_{\theta}e_{\theta}
-30 h_1e_{\alpha_1}e_{\alpha_1}e_{\alpha_2}e_{\alpha_5}e_{\theta}\\
&-30 h_1e_{\alpha_1}e_{\alpha_2}e_{\alpha_3}e_{\alpha_5}e_{\alpha_5}
+10 h_1e_{\alpha_1}e_{\alpha_2}e_{\alpha_4}e_{\alpha_4}e_{\alpha_5}
-16 h_1h_1e_{\alpha_1}e_{\alpha_1}e_{\theta}e_{\theta}
-32 h_1h_1e_{\alpha_1}e_{\alpha_3}e_{\alpha_5}e_{\theta}
+8 h_1h_1e_{\alpha_1}e_{\alpha_4}e_{\alpha_4}e_{\theta}\\
&+3 h_1 h_1 e_{\alpha_2}e_{\alpha_4}e_{\alpha_5}e_{\alpha_5}
-16 h_1h_1e_{\alpha_3}e_{\alpha_3}e_{\alpha_5}e_{\alpha_5}
+8 h_1h_1e_{\alpha_3}e_{\alpha_4}e_{\alpha_4}e_{\alpha_5}
-h_1h_1e_{\alpha_4}e_{\alpha_4}e_{\alpha_4}e_{\alpha_4}
-54 h_2e_{\alpha_1}e_{\alpha_1}e_{\alpha_2}e_{\alpha_5}e_{\theta}\\
&+2 h_2e_{\alpha_1}e_{\alpha_1}e_{\alpha_3}e_{\alpha_4}e_{\theta}
-36 h_2e_{\alpha_1}e_{\alpha_2}e_{\alpha_3}e_{\alpha_5}e_{\alpha_5}
+15 h_2e_{\alpha_1}e_{\alpha_2}e_{\alpha_4}e_{\alpha_4}e_{\alpha_5}
+2h_2e_{\alpha_1}e_{\alpha_3}e_{\alpha_3}e_{\alpha_4}e_{\alpha_5}
-h_2e_{\alpha_1}e_{\alpha_3}e_{\alpha_4}e_{\alpha_4}e_{\alpha_4}\\
&-63 h_2h_1e_{\alpha_1}e_{\alpha_1}e_{\theta}e_{\theta}
-96 h_2h_1e_{\alpha_1}e_{\alpha_3}e_{\alpha_5}e_{\theta}
+29 h_2h_1e_{\alpha_1}e_{\alpha_4}e_{\alpha_4}e_{\theta}
+9 h_2h_1e_{\alpha_2}e_{\alpha_4}e_{\alpha_5}e_{\alpha_5}
-33 h_2h_1e_{\alpha_3}e_{\alpha_3}e_{\alpha_5}e_{\alpha_5}\\
&+19 h_2h_1e_{\alpha_3}e_{\alpha_4}e_{\alpha_4}e_{\alpha_5}
-3 h_2h_1e_{\alpha_4}e_{\alpha_4}e_{\alpha_4}e_{\alpha_4}
+3h_2h_1h_1e_{\alpha_4}e_{\alpha_5}e_{\theta}
-63 h_2h_2e_{\alpha_1}e_{\alpha_1}e_{\theta}e_{\theta}
-72 h_2h_2e_{\alpha_1}e_{\alpha_3}e_{\alpha_5}e_{\theta}\\
&+27 h_2h_2e_{\alpha_1}e_{\alpha_4}e_{\alpha_4}e_{\theta}
-18 h_2h_2e_{\alpha_3}e_{\alpha_3}e_{\alpha_5}e_{\alpha_5}
+12 h_2h_2e_{\alpha_3}e_{\alpha_4}e_{\alpha_4}e_{\alpha_5}
-2 h_2h_2e_{\alpha_4}e_{\alpha_4}e_{\alpha_4}e_{\alpha_4}
+9 h_2h_2h_1e_{\alpha_4}e_{\alpha_5}e_{\theta}
\end{align*}}

\section{Polynomials for subsingular vector of $V^{-2}(B_{3})$.}
\label{App:pol_B3}
We give in this appendix the explicit form of the polynomials 
in the symmetric algebra of the Cartan of $B_3$ 
appearing in Lemma \ref{lem2.4}. 

\smallskip

{\tiny 
\begin{align*}
		p^{B_{3}}_{4}&=-24 (16 h_1^6+112 h_2 h_1^5+64 h_3 h_1^5+96 h_1^5+320 h_2^2 h_1^4+104 h_3^2 h_1^4+496 h_2 h_1^4+368 h_2
		h_3 h_1^4+264 h_3 h_1^4+156 h_1^4+480 h_2^3 h_1^3\\
		&+88 h_3^3 h_1^3+984 h_2^2 h_1^3+472 h_2 h_3^2 h_1^3+260
		h_3^2 h_1^3+444 h_2 h_1^3+832 h_2^2 h_3 h_1^3+1032 h_2 h_3 h_1^3+156 h_3 h_1^3-16 h_1^3+400 h_2^4
		h_1^2\\
		&+41 h_3^4 h_1^2+912 h_2^3 h_1^2+296 h_2 h_3^3 h_1^2+106 h_3^3 h_1^2+246 h_2^2 h_1^2+792 h_2^2 h_3^2
		h_1^2+696 h_2 h_3^2 h_1^2\\
		&-77 h_3^2 h_1^2-464 h_2 h_1^2+928 h_2^3 h_3 h_1^2+1412 h_2^2 h_3 h_1^2+12 h_2
		h_3 h_1^2-358 h_3 h_1^2-180 h_1^2+176 h_2^5 h_1+10 h_3^5 h_1\\
		&+376 h_2^4 h_1+91 h_2 h_3^4 h_1+13 h_3^4
		h_1-150 h_2^3 h_1+328 h_2^2 h_3^3 h_1+154 h_2 h_3^3 h_1-96 h_3^3 h_1-706 h_2^2 h_1+584 h_2^3 h_3^2
		h_1\\
		&+542 h_2^2 h_3^2 h_1-415 h_2 h_3^2 h_1-269 h_3^2 h_1-476 h_2 h_1+512 h_2^4 h_3 h_1+760 h_2^3 h_3
		h_1-488 h_2^2 h_3 h_1-886 h_2 h_3 h_1\\
		&-306 h_3 h_1-72 h_1+32 h_2^6+h_3^6+48 h_2^5+11 h_2 h_3^5-h_3^5-76
		h_2^4+50 h_2^2 h_3^4+h_2 h_3^4-13 h_3^4-72 h_2^3+120 h_2^3 h_3^3\\
		&+33 h_2^2 h_3^3-111 h_2 h_3^3+13 h_3^3+44
		h_2^2+160 h_2^4 h_3^2+100 h_2^3 h_3^2-265 h_2^2 h_3^2-27 h_2 h_3^2+36 h_3^2+24 h_2+112 h_2^5 h_3\\
		&+116
		h_2^4 h_3-244 h_2^3 h_3-110 h_2^2 h_3+54 h_2 h_3-36 h_3)
\\
		p^{B_{3}}_{5}&=-2(32 h_1^6+204 h_2 h_1^5+128 h_3 h_1^5+192 h_1^5+544 h_2^2 h_1^4+200 h_3^2 h_1^4+986 h_2 h_1^4+676 h_2
		h_3 h_1^4+536 h_3 h_1^4+312 h_1^4\\
		&+776 h_2^3 h_1^3+152 h_3^3 h_1^3+1934 h_2^2 h_1^3+839 h_2 h_3^2
		h_1^3+500 h_3^2 h_1^3+888 h_2 h_1^3+1432 h_2^2 h_3 h_1^3+2077 h_2 h_3 h_1^3+356 h_3 h_1^3\\
		&-32 h_1^3+624
		h_2^4 h_1^2+56 h_3^4 h_1^2+1774 h_2^3 h_1^2+474 h_2 h_3^3 h_1^2+160 h_3^3 h_1^2+592 h_2^2 h_1^2+1322
		h_2^2 h_3^2 h_1^2+1301 h_2 h_3^2 h_1^2\\
		&-156 h_3^2 h_1^2-814 h_2 h_1^2+1520 h_2^3 h_3 h_1^2+2808 h_2^2 h_3
		h_1^2+209 h_2 h_3 h_1^2-636 h_3 h_1^2-360 h_1^2+268 h_2^5 h_1+8 h_3^5 h_1\\
		&+730 h_2^4 h_1+115 h_2 h_3^4
		h_1+4 h_3^4 h_1-120 h_2^3 h_1+493 h_2^2 h_3^3 h_1+196 h_2 h_3^3 h_1-200 h_3^3 h_1-1262 h_2^2 h_1+927
		h_2^3 h_3^2 h_1\\
		&+974 h_2^2 h_3^2 h_1-697 h_2 h_3^2 h_1-580 h_3^2 h_1-824 h_2 h_1+808 h_2^4 h_3 h_1+1497
		h_2^3 h_3 h_1-637 h_2^2 h_3 h_1-1718 h_2 h_3 h_1\\
		&-528 h_3 h_1-144 h_1+48 h_2^6+96 h_2^5+8 h_2 h_3^5-72
		h_2^4+59 h_2^2 h_3^4-14 h_2 h_3^4-144 h_2^3+171 h_2^3 h_3^3+18 h_2^2 h_3^3\\
		&-120 h_2 h_3^3+24 h_2^2+244
		h_2^4 h_3^2+167 h_2^3 h_3^2-319 h_2^2 h_3^2-118 h_2 h_3^2+48 h_2+172 h_2^5 h_3+230 h_2^4 h_3-272 h_2^3
		h_3\\
		&-254 h_2^2 h_3+28 h_2 h_3)
\\
		p^{B_{3}}_{6}&=-4(16 h_1^6+97 h_2 h_1^5+64 h_3 h_1^5+96 h_1^5+246 h_2^2 h_1^4+96 h_3^2 h_1^4+462 h_2 h_1^4+323 h_2 h_3
		h_1^4+272 h_3 h_1^4+156 h_1^4\\
		&+334 h_2^3 h_1^3+64 h_3^3 h_1^3+860 h_2^2 h_1^3+387 h_2 h_3^2 h_1^3+240
		h_3^2 h_1^3+477 h_2 h_1^3+654 h_2^2 h_3 h_1^3+989 h_2 h_3 h_1^3+200 h_3 h_1^3\\
		&-16 h_1^3+256 h_2^4 h_1^2+16
		h_3^4 h_1^2+760 h_2^3 h_1^2+193 h_2 h_3^3 h_1^2+48 h_3^3 h_1^2+402 h_2^2 h_1^2+586 h_2^2 h_3^2 h_1^2+576
		h_2 h_3^2 h_1^2\\
		&-68 h_3^2 h_1^2-312 h_2 h_1^2+664 h_2^3 h_3 h_1^2+1268 h_2^2 h_3 h_1^2+192 h_2 h_3
		h_1^2-284 h_3 h_1^2-180 h_1^2+105 h_2^5 h_1+308 h_2^4 h_1\\
		&+32 h_2 h_3^4 h_1-16 h_3^4 h_1+43 h_2^3 h_1+194
		h_2^2 h_3^3 h_1+33 h_2 h_3^3 h_1-112 h_3^3 h_1-488 h_2^2 h_1+395 h_2^3 h_3^2 h_1+393 h_2^2 h_3^2 h_1\\
		&-335
		h_2 h_3^2 h_1-268 h_3^2 h_1-400 h_2 h_1+338 h_2^4 h_3 h_1+649 h_2^3 h_3 h_1-203 h_2^2 h_3 h_1-734 h_2 h_3
		h_1-252 h_3 h_1-72 h_1\\
		&+18 h_2^6+42 h_2^5-6 h_2^4+16 h_2^2 h_3^4-16 h_2 h_3^4-42 h_2^3+65 h_2^3 h_3^3-15
		h_2^2 h_3^3-50 h_2 h_3^3-12 h_2^2+100 h_2^4 h_3^2+57 h_2^3 h_3^2\\
		&-119 h_2^2 h_3^2-38 h_2 h_3^2+69 h_2^5
		h_3+98 h_2^4 h_3-77 h_2^3 h_3-86 h_2^2 h_3-4 h_2 h_3)
\\
		p^{B_{3}}_{7}&=-2h_1 \left(h_1+2 h_2+h_3+2\right) (32 h_1^4+180 h_2 h_1^3+96 h_3 h_1^3+128 h_1^3+372 h_2^2 h_1^2+104
		h_3^2 h_1^2+422 h_2 h_1^2\\
		&+392 h_2 h_3 h_1^2+216 h_3 h_1^2+56 h_1^2+332 h_2^3 h_1+48 h_3^3 h_1+332 h_2^2
		h_1+273 h_2 h_3^2 h_1+76 h_3^2 h_1-176 h_2 h_1\\
		&+520 h_2^2 h_3 h_1+341 h_2 h_3 h_1-132 h_3 h_1-144 h_1+108
		h_2^4+8 h_3^4-26 h_2^3+61 h_2 h_3^3-12 h_3^3-512 h_2^2\\
		&+175 h_2^2 h_3^2-34 h_2 h_3^2-176 h_3^2-402 h_2+224
		h_2^3 h_3-37 h_2^2 h_3-613 h_2 h_3-228 h_3-72)
		\\ 
		p^{B_{3}}_{8}&=-2h_1 \left(h_1+2 h_2+h_3+2\right)(16 h_1^4+80 h_2 h_1^3+48 h_3 h_1^3+64 h_1^3+148 h_2^2 h_1^2+48
		h_3^2 h_1^2+217 h_2 h_1^2\\
		&+176 h_2 h_3 h_1^2+112 h_3 h_1^2+28 h_1^2+120 h_2^3 h_1+16 h_3^3 h_1+188 h_2^2
		h_1+112 h_2 h_3^2 h_1+32 h_3^2 h_1-66 h_2 h_1\\
		&+212 h_2^2 h_3 h_1+171 h_2 h_3 h_1-52 h_3 h_1-72 h_1+36
		h_2^4+3 h_2^3+16 h_2 h_3^3-16 h_3^3-216 h_2^2+64 h_2^2 h_3^2\\
		&-46 h_2 h_3^2-80 h_3^2-195 h_2+84 h_2^3
		h_3-31 h_2^2 h_3-261 h_2 h_3-108 h_3-36)
\\ 
		p^{B_{3}}_{9}&=-4\left(h_1-1\right) h_1 \left(h_1+2 h_2+h_3+2\right) \left(h_1+2 h_2+h_3+3\right)(16 h_1^2+63 h_2
		h_1+32 h_3 h_1+32 h_1+63 h_2^2\\
		&+16 h_3^2+63 h_2+63 h_2 h_3+32 h_3+12)
\end{align*}}

\end{document}